\documentclass[10pt, english, a4paper]{article}
\usepackage[utf8]{inputenc}
\PassOptionsToPackage{hyphens}{url}
\usepackage[hidelinks]{hyperref} %colorlinks=true, linkcolor=blue!90!black
\usepackage{graphicx}
\usepackage{varioref, babel, fancyvrb, listings, amsmath,amsthm, amssymb, mathtools}
\usepackage{array}
\usepackage{enumitem}
\usepackage[a4paper, total={6in, 8in}]{geometry}
\usepackage{mathabx}
\usepackage{listings}
\usepackage{todonotes}
\usepackage{subcaption}
\usepackage{csquotes}
\usepackage{tikz}
\usepackage{booktabs}
\usepackage[font=footnotesize]{caption}

\usetikzlibrary{patterns}
\usetikzlibrary{arrows}
\usetikzlibrary{patterns}
\usetikzlibrary{shapes.misc}
\usetikzlibrary{arrows.meta} % requiers tikz v3.0 or newer
\tikzset{cross/.style={cross out, draw, 
         minimum size=2*(#1-\pgflinewidth), 
         inner sep=0pt, outer sep=0pt},
         cross/.default={3pt}}

\usepackage[normalem]{ulem}

\usepackage{algorithm}
\usepackage[noend]{algpseudocode}

\def\textsizefig{footnotesize}

\AtBeginEnvironment{figure}{\setlength{\tabcolsep}{2pt}}
\AtBeginEnvironment{figure*}{\setlength{\tabcolsep}{2pt}}
\AtBeginEnvironment{table}{\setlength{\tabcolsep}{2pt}}

\DeclareMathOperator{\Supp}{supp} % Do not use \supp with a lower case letter
\DeclareMathOperator{\Span}{span}

\def\C{\mathbb{C}}
\def\Z{\mathbb{Z}}
\def\N{\mathbb{N}}

\def\R{\mathbb{R}}
\def\d{~\textnormal{d}}
\newcommand{\Hs}{\mathcal{H}}
\newcommand{\Rs}{\mathcal{R}}
\newcommand{\Ss}{\mathcal{S}}

\def\({\left(}
\def\){\right)}

\newcommand{\ind}[1]{\left\langle #1 \right\rangle}
\newcommand{\indi}[1]{\langle #1 \rangle} % inline
\newcommand{\floor}[1]{\left\lfloor #1 \right\rfloor}
\newcommand{\ceil}[1]{\left\lceil #1 \right\rceil}

\theoremstyle{plain}
\newtheorem{theorem}{Theorem}[section]
\newtheorem{lemma}[theorem]{Lemma}
\newtheorem{proposition}[theorem]{Proposition}

\theoremstyle{definition}
\newtheorem{definition}[theorem]{Definition}

\theoremstyle{remark}
\newtheorem{remark}[theorem]{Remark}

\title{\textbf{Recovering wavelet coefficients from binary samples using fast transforms}}

\author{Vegard Antun\thanks{Department of Mathematics, University of Oslo, Norway. (\texttt{vegarant@math.uio.no})}}

\begin{document}

\maketitle

\begin{abstract}
Recovering a signal (function) from finitely many binary or Fourier samples is one of the core problems in modern medical imaging, and by now there exist a plethora of methods for recovering a signal from such samples. Examples of methods, which can utilise wavelet reconstruction, include generalised sampling, infinite-dimensional compressive sensing, the parameterised-background data-weak (PBDW) method etc. However, for any of these methods to be applied in practice, accurate and fast modelling of an $N \times M$ section of the infinite-dimensional change-of-basis matrix between the sampling basis (Fourier or Walsh-Hadamard samples) and the wavelet reconstruction basis is paramount. In this work, we derive an algorithm, which bypasses the $NM$ storage requirement and the $\mathcal{O}(NM)$ computational cost of matrix-vector multiplication with this matrix when using Walsh-Hadamard samples and wavelet reconstruction. The proposed algorithm computes the matrix-vector multiplication in $\mathcal{O}(N\log N)$ operations and has a storage requirement of $\mathcal{O}(2^q)$, where $N=2^{dq} M$, (usually $q \in \{1,2\}$) and $d=1,2$ is the dimension. As matrix-vector multiplications is the computational bottleneck for iterative algorithms used by the mentioned reconstruction methods, the proposed algorithm speeds up the reconstruction of wavelet coefficients from Walsh-Hadamard samples considerably. 
\end{abstract}



\section{Introduction}

Approximating a function from finitely many samples is one of the fundamental problems in approximation theory, and, by now, there exist myriads of conditions and algorithms for obtaining good function approximation. The problem is often motivated by the many applications in natural sciences where one is given a finite set of samples of an underlying unknown signal (function) that one wants to recover (approximate).

In this work, we consider the recovery of signals, where physical constraints dictate the type of samples one can acquire. This is a well-studied problem with numerous applications in medical imaging. Examples include Magnetic Resonance Imaging (MRI) \cite{mri_principles, cs_MRI_lustig}, surface scattering \cite{jardine2009helium, jones2016continuous}, X-ray Computed Tomography (CT) \cite{epstein2007introduction} and electron microscopy \cite{leary2013compressed}, all of which employ Fourier sampling. Other examples, employing binary samples, include fluorescence microscopy \cite{studer2012compressive, muller2006introduction}, lensless imaging \cite{lensless_im} and compressive holography \cite{Clemente:13}. 

Given the long list of applications, there are many efficient methods for reconstructing a function from a fixed sampling modality. Examples of such methods include \emph{generalised sampling} \cite{adcock2012generalized, adcock2015linear, Beyond13, adcock2014stability,  hrycak10, ma2017generalized}, studied by Adcock, Hansen,
Hrycak, Gröchenig, Kutyniok, Ma, Poon, Shadrin and others, its predecessor; \emph{consistent sampling} \cite{eldar2003sampling, eldar2004sampling, eldar2005general, hirabayashi2007consistent, unser1994general, unser1998generalized}, developed by Aldroubi, Eldar, Unser and others. More recently Adcock, Antun, Hansen, Kutyniok, Lim, Poon, Thesing and many others have developed reconstruction methods based on \emph{infinite-dimensional compressive sensing} \cite{AntunUniform, adcock2016generalized, adcock2017breaking, kutyniok2018optimal, poon2014consistent, thesing2021non}. Other approaches can be found within data assimilation. A first approach here was introduced by Maday \& Mula in \cite{maday2013generalized}, called \emph{generalised empirical interpolation method}, this was later followed by the \emph{Parametrized Background Data-Weak} (PBDW) method, developed by Maday, Patera, Penn \& Yano in \cite{maday2015pbdw, maday15}, and later analysed by Binev, Cohen, Dahmen, DeVore, Petrova, and Wojtaszczyk in \cite{Binev17, devore2017data}.

We model the problem as follows. Let $\Hs$ be an infinite-dimensional separable Hilbert space with inner product $\ind{\cdot,\cdot}$ and norm $\|\cdot\|$. Let $\{s_{k} : k \in \mathbb{N}\}$ and $\{r_{k} : k \in \mathbb{N}\}$ be two orthonormal bases for $\Hs$, called  the sampling and reconstruction basis, respectively.  Furthermore define the sampling space, as the linear span $\Ss_{N} = \Span\{s_{1},\ldots, s_{N}\}$ and the reconstruction space as $\Rs_{M}= \Span\{r_{1}, \ldots, r_{M}\}$.

Suppose that we can only observe the function $f \in \Hs$, using \emph{finitely many} linear measurements $\ind{f,s_k}$, $k=1,\ldots,N$. Since $\{s_k : k \in \mathbb{N}\}$ is an orthonormal basis, this immediately gives the truncated series approximation
\begin{equation}
\label{eq:approx_f_N}
f_{N} = y_1s_1 + \cdots + y_N s_N \in \Ss_{N} 
\end{equation}
where $y_k = \ind{f,s_k}$. In all the applications mentioned above, we have limited freedom in designing the sampling basis $\{s_k : k \in \mathbb{N}\}$ and the approximation $f_N$ may, therefore, suffer from unpleasant reconstruction artefacts due to the characteristics of the sampling basis, slow convergence rates or the Gibbs phenomenon. 

\begin{figure}[tb]
    \centering
    \begin{\textsizefig}
    {\setlength{\tabcolsep}{15pt}
    \begin{tabular}{@{}>{\centering}m{0.30\textwidth}>{\centering\arraybackslash}m{0.30\textwidth}@{}}
    $f(t)$ & $g(t)$ \\
     \includegraphics[width=\linewidth]{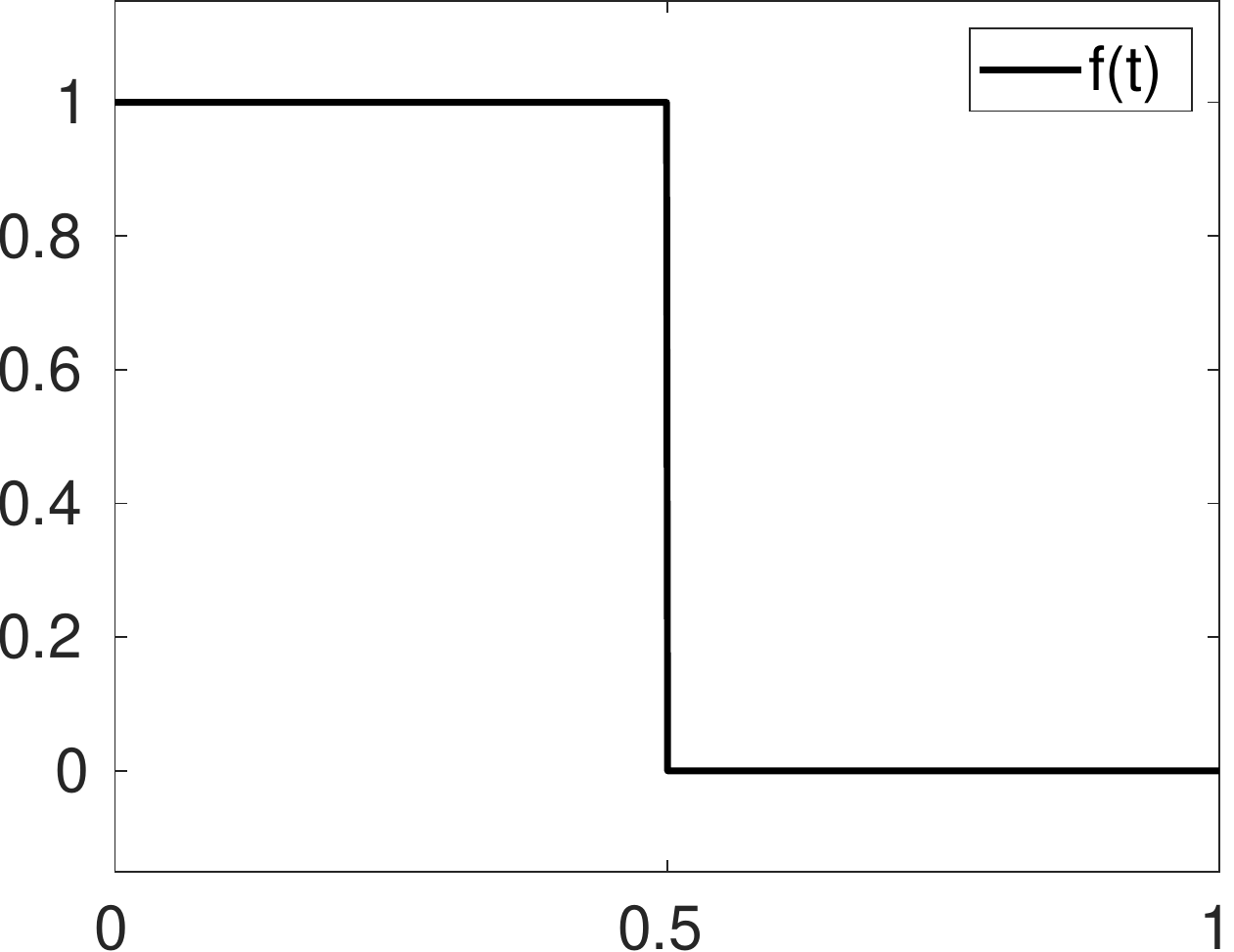} 
    &\includegraphics[width=\linewidth]{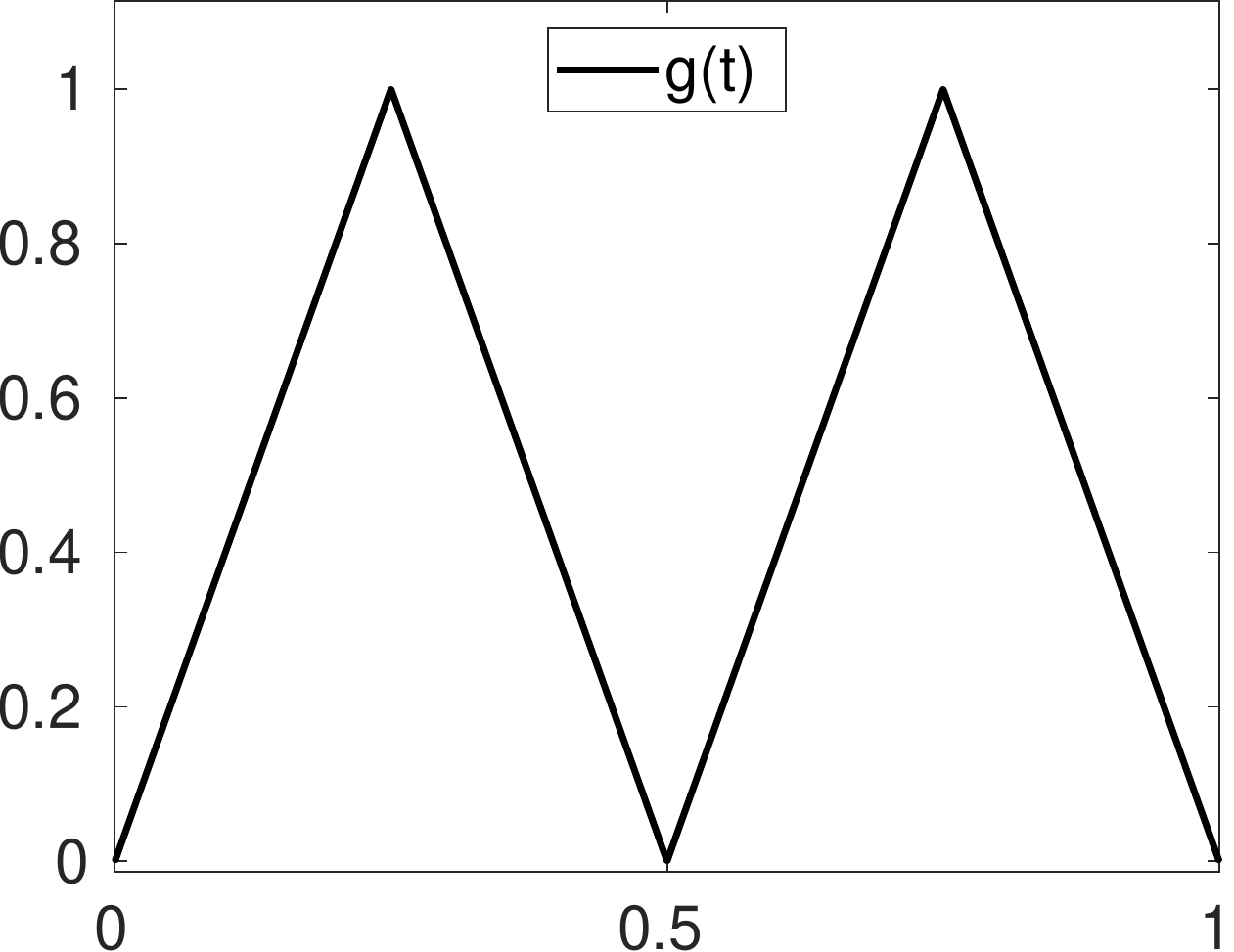} 
    \end{tabular}
    }
    \begin{tabular}{@{}>{\centering}m{0.30\textwidth}>{\centering}m{0.30\textwidth}>{\centering\arraybackslash}m{0.30\textwidth}@{}}
     $\text{Fourier: } f_{N}, ~ N = 16$ 
    &$\text{Fourier: } f_{N}, ~ N = 256$ 
    &$\text{Walsh: } g_{N}, ~ N = 64$ \\
     \includegraphics[width=\linewidth]{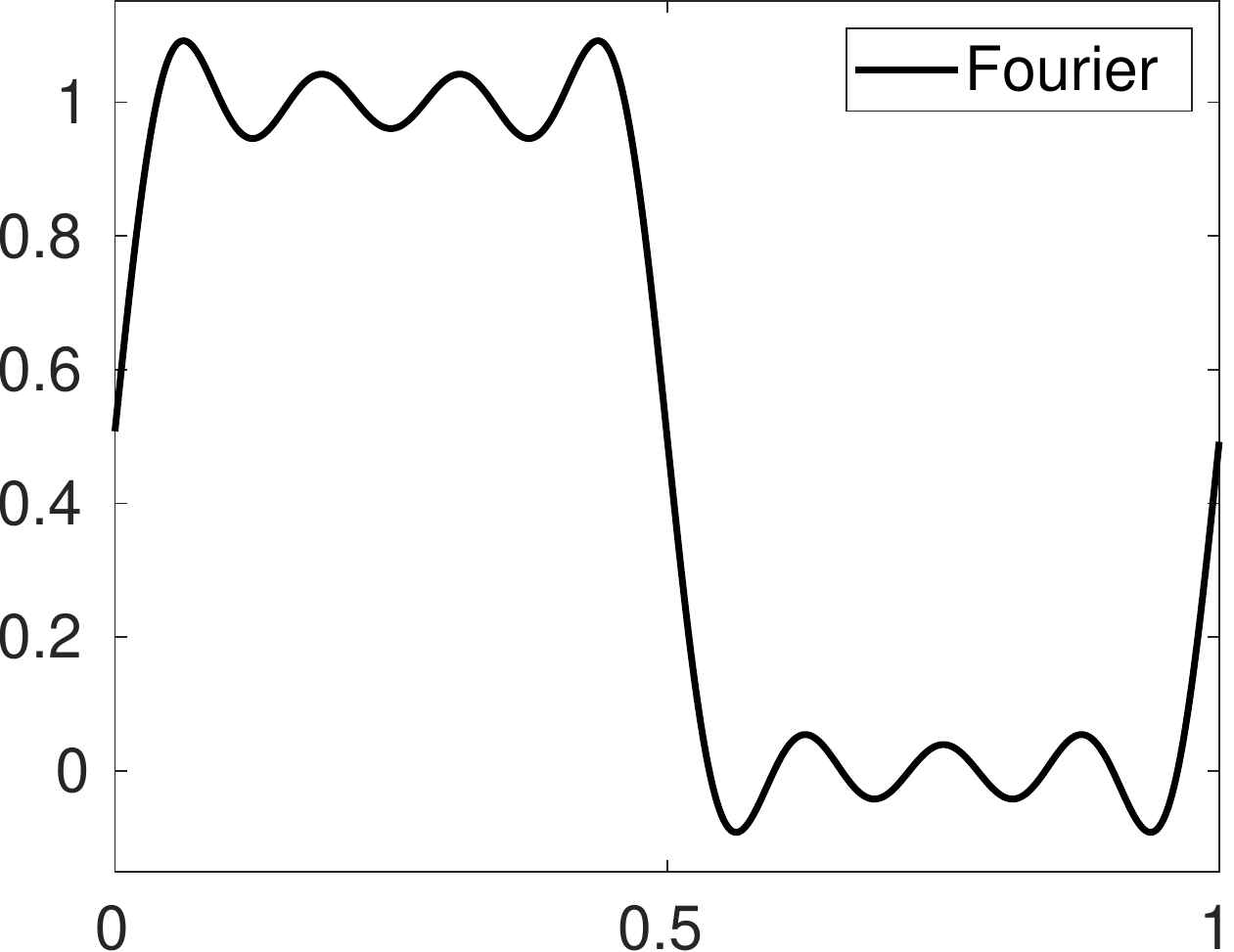} 
    &\includegraphics[width=\linewidth]{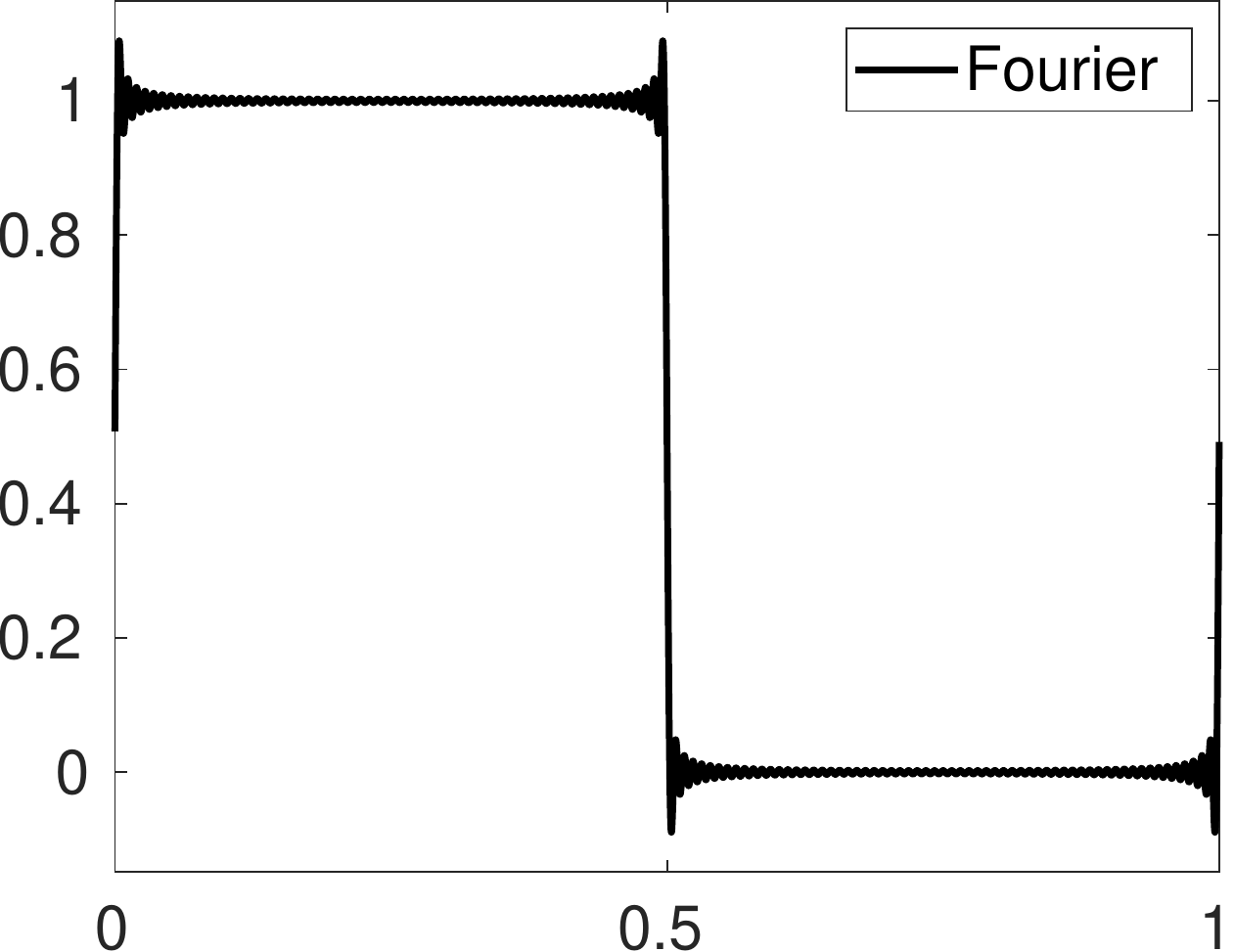} 
    &\includegraphics[width=\linewidth]{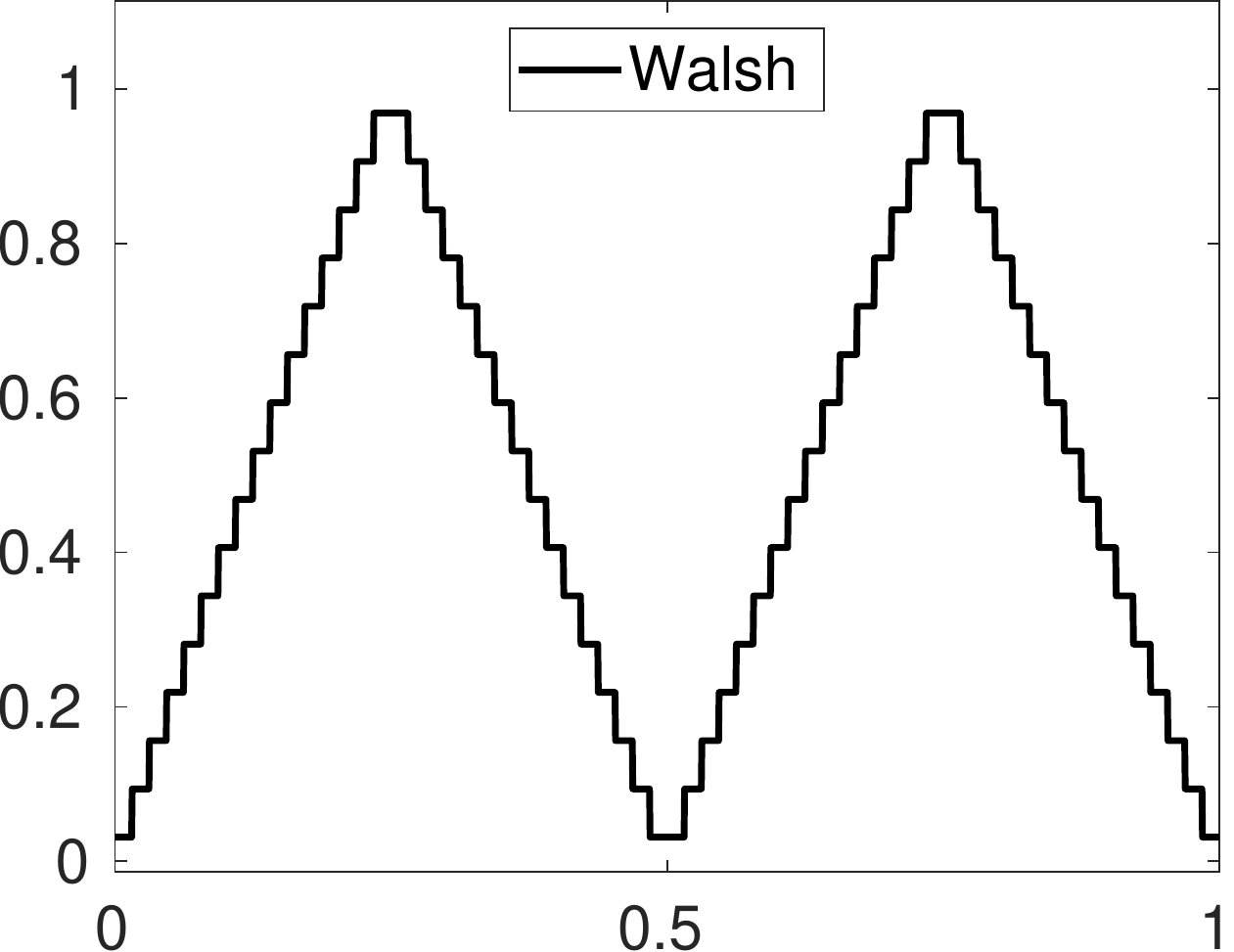} 
    \end{tabular}
    \end{\textsizefig}
    \vspace{-3mm}
    \caption{\label{fig:f_N_app}
\textbf{(Undesirable artefacts)}. The two functions $f$ and $g$  (top row) are sampled using a Fourier and Walsh sampling basis, respectively. Given the acquired samples, we use the native truncated Fourier series $f_N$ and truncated Walsh series $g_{N}$, known from \eqref{eq:approx_f_N}, to approximate the functions. On the bottom row we show reconstructions $f_N$ and $g_N$, for different values of $N$. Notice how the truncated Fourier series causes $\mathcal{O}(1)$ Gibbs oscillations around the discontinuity for every choice of $N$, and how the truncated Walsh series produce a reconstruction with blocky artefacts.
}
%\vspace{-4mm}
\end{figure}

An example of such artefacts can be seen in Figure \ref{fig:f_N_app}. Here we have chosen $\mathcal{H} = L^2([0,1])$ and consider the Fourier sampling basis $\{ (2\pi)^{-1/2} e^{2\pi \mathrm{i} n}: n \in \mathbb{Z}\}$ and the Walsh sampling basis $\{w_n : n \in \Z_+ \coloneqq \{0,1,\ldots,\} \}$, where the $w_n$'s are Walsh functions (see \S\ref{ss:wal} for more on these functions, and their relation to Hadamard matrices). In the figure, we can see how the Walsh sampling basis gives a blocky  approximation to the continuous hat functions and how the Fourier sampling basis, (no matter how large we choose $N$), always produce the very characteristic $\mathcal{O}(1)$ Gibbs oscillations around the discontinuity. This is because $f_N$ only converges to $f$ in the $\ell^2$-norm, rather than the stronger uniform norm.

\begin{figure}
    \centering
    \begin{\textsizefig}
    \setlength{\tabcolsep}{10pt}
    \begin{tabular}{@{}>{\centering}m{0.30\textwidth}>{\centering\arraybackslash}m{0.30\textwidth}@{}}
     GS rec. $\widetilde{f}$ $(M=8, N=16)$ 
    &GS rec. $\widetilde{g}$ $(M=32, N=64)$ \\
     from  Fourier samples  
    &from Walsh samples \\
     \includegraphics[width=\linewidth]{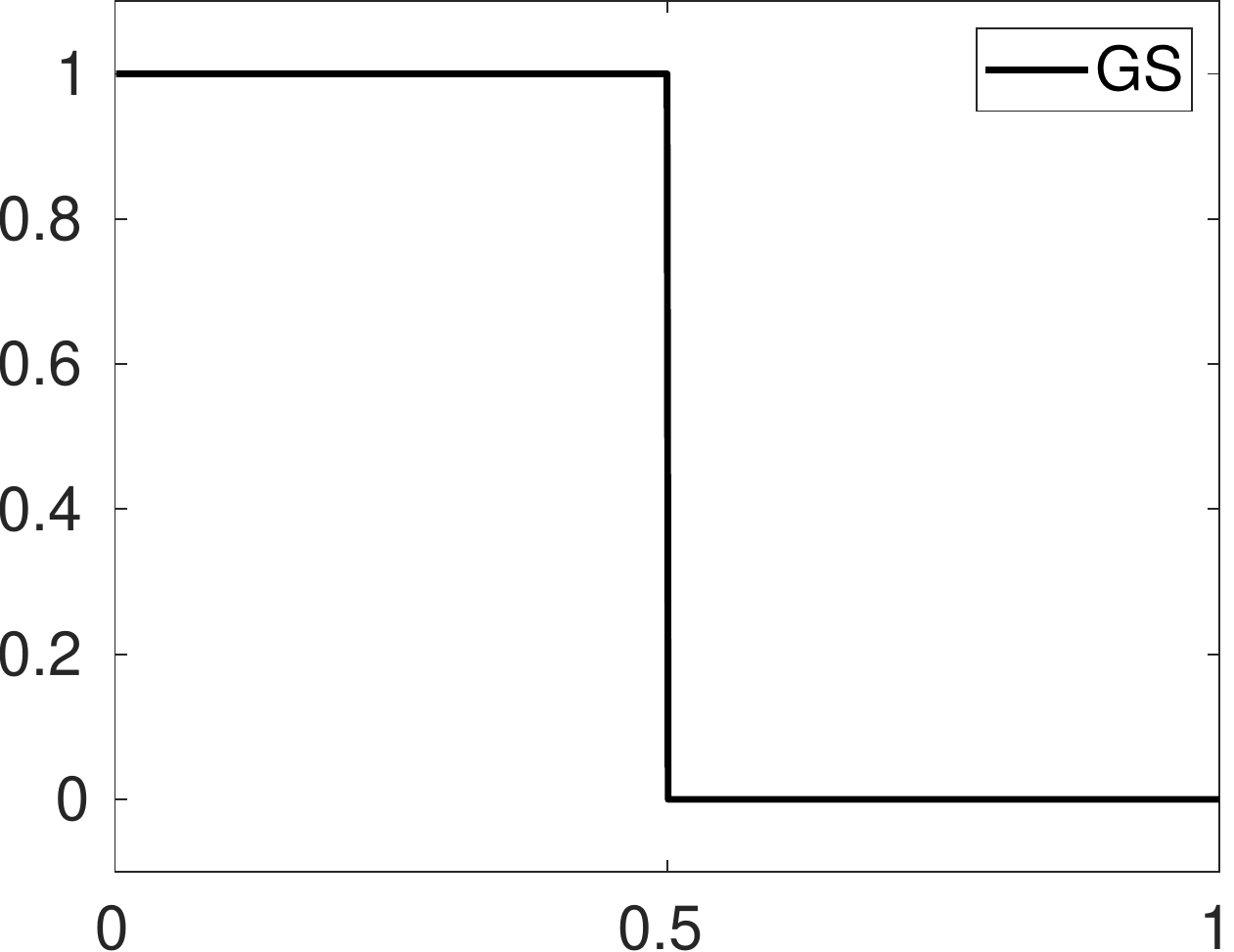} 
    &\includegraphics[width=\linewidth]{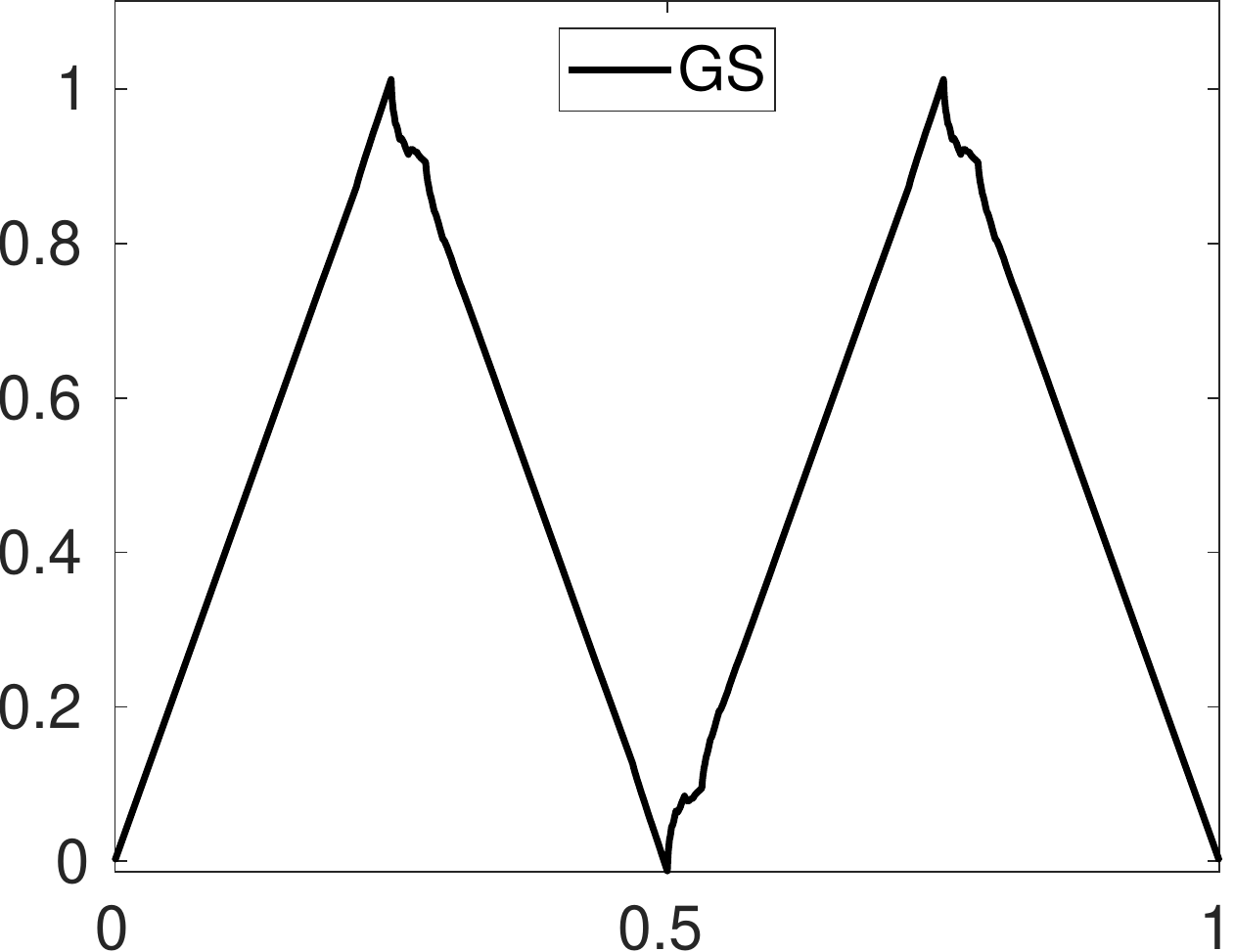} 
    \end{tabular}
    \end{\textsizefig}
    \vspace{-3mm}
    \caption{\label{fig:GS_rec}
\textbf{(Improved reconstruction in $\Rs_{M}$)}.
Given $N=16$ Fourier samples from the function $f$ in Figure \ref{fig:f_N_app} and $N=64$ Walsh samples from the function $g$ in the same figure, we compute approximations $\widetilde{f}$ and $\widetilde{g}$, respectively, using generalised sampling (GS). Here we use a Haar wavelet basis with $M=8$ functions for $\widetilde{f}$, and a Daubechies 2 (DB2) wavelet basis with $M=32$ basis functions for the function $\widetilde{g}$. Note how increasing $N$ in Figure \ref{fig:f_N_app} can not remedy the $\mathcal{O}(1)$ Gibbs oscillation for the discontinuous Haar scaling function $f$ seen in the figure, whereas choosing a basis which spans this function enables us to capture $f$, using only $N=16$ samples.   
}
%\vspace{-4mm}
\end{figure}

To resolve this issue, the idea of the aforementioned reconstruction techniques is to utilise prior knowledge on $f$, to compute a better approximation in the reconstruction space $\Rs_{M}$, using the samples $\{y_1, \ldots, y_{N}\}$. In this work $\Rs_{M}$ is spanned by orthonormal wavelets and we consider $\Hs = L^2([0,1]^d)$, for $d = 1,2$. This reconstruction space has several advantages.

\begin{enumerate}[label=(\roman*)]
\item Orthonormal wavelets can be computed with any desired degree of smoothness, ranging from the discontinuous Haar wavelet to higher-order Daubechies wavelets or symlets. This means that we can tailor-make the smoothness of the reconstruction space. 

\item In one dimension, orthonormal wavelets allows for optimal non-linear approximation of functions with bounded variations \cite[Ch.\ 10]{CSBook} (see also \cite{devore_1998}) and while wavelets are not provably optimal in two dimensions, their use and applicability in imaging is ubiquitous \cite{fastWaveletRecUnser, cs_MRI_lustig, ravishankar2019image}.

\item For Walsh sampling (considered in this work) and orthonormal wavelet reconstruction, the so-called \emph{stable sampling rate} (see Def.\ \ref{def:ssr}) is linear \cite{ThesingSSR}. That is, to recover $M$ wavelet coefficients using, e.g., generalised sampling, we require $N \geq CM$ Walsh samples, where $C \geq 1$ is a constant. We note that this rate is not necessarily linear for all reconstruction bases. For Fourier sampling and polynomial reconstruction, the requirement is quadratic in $M$, i.e., $N \geq C M^2$ samples are required \cite{hrycak10}. For Walsh sampling and polynomial reconstruction, the stable sampling rate is not known.

\end{enumerate}

\subsection{Notation}
Let $\ell^2(N)$ denote the usual set of square summable sequences, and let $\mathcal{B}(\ell^2(\mathbb{N}))$, denote the set of bounded linear operators between such sequences. 
For $\Omega \subseteq \{1, \ldots, N\}$, we let $P_{\Omega} \colon \ell^2(\mathbb{N}) \to \ell^2(\mathbb{N})$ be the projection onto the coordinates indexed by $\Omega$. That is, for $z \in \ell^2(\mathbb{N})$, $(P_{\Omega}z)_i = z_i$ if $i \in \Omega$, and 0 otherwise.  
%\[(P_{\Omega}z)_i = \begin{cases} z_i &\text{if } i \in \Omega \\ 0 &\text{otherwise}\end{cases}.\]
Let $m=| \Omega |$. We sometimes abuse notation slightly and say that $P_{\Omega}\colon \ell^{2}(\N) \to \C^{m}$, by simply ignoring all the zero entries. Furthermore, if $\Omega = \{1,\ldots, N\}$ we simply write $P_{N}$. Often we do not specify the domain and range of $P_{N}$, and let this be given by the context. Thus for an operator $U \in \mathcal{B}(\ell^2(\mathbb{N}))$, we write $P_{N}UP_{M}$ both to mean a finite dimensional $N\times M$ matrix and an operator in $\mathcal{B}(\ell^2(\mathbb{N}))$, depending on the context. When $P_{M} \colon \ell^2(\mathbb{N}) \to \C^M$, we have that $P_{M}^* \colon \C^M \to \ell^2(\mathbb{N})$, however, to unify the notation we still write $P_{N}UP_{M}$, rather than $P_{N}UP_{M}^*$. 

Finally, for some closed subspace $\mathcal{V} \subset \Hs$ we let $P_{\mathcal{V}}\colon \Hs \to \Hs$ denote the projection onto $\mathcal{V}$.  

\subsection{Computing approximations in $\mathcal{R}_{M}$}
For $f \in \mathcal{H}$, let $x_{k} = \ind{f,r_k}$ and $y_{k} = \ind{f,s_k}$ be the coefficients of $f$ in the reconstruction basis  and sampling basis, respectively. Let $x = \{x_k\}_{k\in \N}$ and $y = \{y_k\}_{k\in \N}$ and notice that $x,y \in \ell^2(\mathbb{\N})$. The \emph{change-of-basis matrix} $U \in \mathcal{B}(\ell^2(\mathbb{N})$ between $\{r_k : k\in \N\}$ and  $\{s_k : k\in \N\}$, is given by 
\[ U_{i,j} =\ind{r_j, s_i}, \quad\text{and}\quad y = Ux, \] 
where $U$ is unitary, since both bases are orthonormal.

Given a finite set of (noiseless) samples, the previously mentioned reconstruction techniques compute an approximation to $f$, by utilising the reconstruction space $\mathcal{R}_{M}$. We review three of the most modern approaches. 

\begin{enumerate}[label=(\roman*)]

\item (Generalised sampling). In generalised sampling \cite{adcock2012generalized, Beyond13} one has access to the $N$ samples $P_{N}y$ and using these we solve the least squares problem 
\begin{equation} 
\label{eq:GS_approx}
    \min_{z \in \C^{M}} \| P_{N} U P_M z - P_{N} y \|^{2}_{\ell^2}, \quad \text{where}\quad N \geq M.  
\end{equation}
Let $\widetilde{x} = \{\widetilde{x}_k\}_{k=1}^{M}$ be the minimiser of \eqref{eq:GS_approx}. In generalised sampling we approximate $f$ with $\widetilde{f} = \widetilde{x}_1 r_1 + \cdots + \widetilde{x}_{M}r_{M} \in \mathcal{R}_{M}$. Moreover, the error committed by $\widetilde{f}$, is upper bounded by  \cite[Thm.\ 4.5]{Beyond13}
\begin{equation}\label{eq:errorGS}
\|f-\widetilde{f}\| \leq C_1 \| f-P_{\mathcal{R}_{M}}f \|,
\end{equation}
where $C_1>0$ is a constant depending on the subspace angle between $\mathcal{S}_N$ and $\mathcal{R}_M$ (see \S\ref{s:sub_ang} for details).

\item (PBDW-method).
The PBDW-method \cite{Binev17, maday15} is a data consistent method, which approximates $f$ using the same $N$ samples $P_{N}y$ as in generalised sampling. The approximation is computed as  
$ \widehat{f} = P_{\mathcal{S}_N} f + P_{\mathcal{S}_{N}^{\perp}} \widetilde{f} $
where $\widetilde{f}$ is the generalised sampling approximation. As $\widehat{f} \in \mathcal{H}$, does not lie in a finite dimensional subspace, it can not be represented on a computer.  We may, however, approximate $\widehat{f}$, by choosing some large $K > N$, and use the truncated sum 
$ \widehat{f} \approx \sum_{k=1}^{N} y_k s_k + \sum_{k=N+1}^{K} (P_{K}UP_{M}\widetilde{x})_{k} s_k  $
where $\widetilde{x}$ is the minimizer form \eqref{eq:GS_approx}. It was shown in \cite{maday15}, that the error committed by $\widehat{f}$ is upper bounded by 
\begin{equation}\label{eq:errorPBDW}
\|f-\widehat{f}\| \leq C_1 \|f - P_{\mathcal{R}_{M} \oplus (\mathcal{S}_{N}\cap \mathcal{R}_{M}^{\perp})}f\|,
\end{equation}
where $C_1$ is the same constant as in the generalised sampling error bound above.  

\item (Infinite-dimensional compressive sensing). While the two methods above are linear reconstruction methods, compressive sensing (and more generally sparse regularization), is an example of a non-linear reconstruction method. In compressive sensing one computes an approximation in $\mathcal{R}_{M}$ using $m < N$ samples. Let $\Omega \subset \{1,\ldots, N\}$ have cardinality $m=|\Omega|$ and consider the measurements $P_{\Omega} y$. A standard way of computing a compressive sensing reconstruction is by solving the quadratically constrained basis pursuit optimisation problem
\begin{equation}\label{eq:QCBP} 
    \min_{z \in \C^M} \|z\|_{\ell^1}  \quad\text{subject to}\quad 
    \|P_{\Omega} UP_{M}z - P_{\Omega} y \|_{\ell^2}^{2} \leq \eta. 
\end{equation}
Here $\eta$ is chosen so that $\eta \geq \|P_{\Omega} UP_{M}^{\perp}x\|_{\ell^2}^{2}$, to ensure that $P_{M}x$ is a feasible point. Given a minimizer $x^{\sharp}$ of \eqref{eq:QCBP}, one approximates $f$ with $f^{\sharp} = x^{\sharp}_1 r_1 + \ldots x^{\sharp}_M r_M \in \mathcal{R}_{M}$.
Error bounds for compressive sensing reconstructions are probabilistic in nature and depend on the number of measurements $m$, and the bases $\{s_k\}_{k\in \N}$ and $\{r_k\}_{k\in \N}$ used. For concrete error bounds for Walsh sampling and wavelet reconstruction, we refer to \cite{thesing2021non} for non-uniform
and \cite{AntunUniform} uniform recovery guarantees in infinite-dimensions. For a more general treatment of the subject, we refer to \cite{CSBook, Foucart13}.
\end{enumerate}

\subsection{Contributions}
In this work, we let $\mathcal{H}=L^2([0,1]^d)$, $d=1,2$ and consider the recovery of orthonormal wavelet coefficients from Walsh samples (also called Walsh-Hadamard, or just Hadamard samples). As outlined above, this setup has numerous applications in binary imaging \cite{lensless_im, Clemente:13, muller2006introduction, studer2012compressive}. However, for any of the reconstruction methods mentioned above to work in practice, it is crucial to solve one of the optimisation problems \eqref{eq:GS_approx} or \eqref{eq:QCBP}. To do this, we need to form the matrix $P_{N}UP_{M}$ (potentially also $P_{\Omega}UP_{M}$), for different values of $N$ and $M$. This can be computationally challenging since the entries of $P_{N}UP_{M}$ are given as the solution of $MN$ integrals. Furthermore, -- ignoring the computational burden of computing these integral -- using a densely stored matrix $P_{N}UP_{M}$ has several disadvantages.
\begin{enumerate}[label=(\roman*)]
    \item (Storage). In imaging applications it is not uncommon to have large dimensions, say 
$N=512^2$ and $M=256^2$. However, naively storing a dense matrix $P_{N}UP_{M} \in \mathbb{C}^{N\times M}$ with these dimensions requires approximately $137\,$GB of memory. This is substantially more than most workstations can handle.

    \item (Computational complexity). When solving \eqref{eq:GS_approx} or \eqref{eq:QCBP}, iterative algorithms are often applied. For \eqref{eq:GS_approx}, the conjugate gradient method \cite{hestenes1952methods} is a popular choice, and for \eqref{eq:QCBP} SPGL1 \cite{splg1_paper} or Chambolle and Pock’s primal-dual \cite{chambolle2011first} algorithm are well-known choices. However, all of these algorithms rely on fast matrix-vector multiplications with $P_{N}UP_{M}$ or $P_{\Omega}UP_{M}$, and their adjoins. However, standard matrix-vector multiplication with a $N\times M$ matrix require $\mathcal{O}(MN)$ operations, and for large dimensions this cost can be substantial. 
\end{enumerate}  
While some of these issues can be reduced in higher dimensions ($d > 1$), by considering tensor decompositions of the linear map $P_{N}UP_{M}$, none of these approaches can obtain a computational complexity of $\mathcal{O}(N \log N)$ and avoid storing the matrix $P_{N}UP_{M}$ altogether. In this work, we do exactly this. We present an algorithm, which can compute matrix-vector multiplications with the matrix $P_{N}UP_{M}$ in $\mathcal{O}(N\log N)$\footnote{Note that our bound here, is independent of $M$, but due to the stable sampling rate (see \S\ref{s:sub_ang}), we can take $N=2^{dq}M$ for small values of $q$, usually $q \in \{1,2,3,4\}$ (see Rem. \ref{r:scalingMN}).} operations for Walsh sampling and orthonormal wavelet reconstruction in one and two dimensions without storing the matrix $P_{N}UP_{M}$. Applying the reconstruction methods outlined above allows for fast reconstruction of wavelet coefficients from Walsh samples with minimal memory usage and computational complexity. 

Our work extends the work of Gataric \& Poon \cite{Gataric16}, which derives a similar algorithm for Fourier sampling and wavelet reconstruction. However, our work differs from \cite{Gataric16} in that we utilise special properties of the Walsh functions and derive an algorithm that can be used for both vanishing moments preserving wavelets on the interval \cite{AntunRyan, Cohen93} and periodic wavelets on the interval \cite[Sec.\ 7.5.1]{Mallat09}. The paper is also accompanied by a software implementation in \textsc{MatLab}, demonstrating how this can be implemented in practice. It is an well known issue that \textsc{MatLab}'s implementation of the fast Walsh-Hadamard transform (FWHT), is extremely slow\footnote{See \href{https://ch.mathworks.com/matlabcentral/answers/395334-why-does-the-fwht-function-calculate-slower-than-the-fft-function-even-though-the-documentation-say}{https://ch.mathworks.com/matlabcentral/answers/395334-why-does-the-fwht-function-calculate-slower-than-the-fft-function-even-though-the-documentation-say}}. To mitigate this issue, the implementation also includes a \textsc{MatLab} interface to the C++ library FXT (\url{https://www.jjj.de/fxt/}) \cite{mattersComp}, for speeding up this part of the code. Other time-critical parts of the code have also been written in C++ and interfaced with \textsc{MatLab}. All accompanying code and data are accessible from 
\begin{center}
\vspace{-2mm}
\url{https://github.com/vegarant/cww} 
\vspace{-2mm}
\end{center}
and
\begin{center}
\vspace{-2mm}
\url{https://github.com/vegarant/fastwht}.
%\vspace{-2mm}
\end{center}

\begin{remark}[Avoiding inverse crimes]
    Note that the proposed model avoids certain inverse crimes stemming from too early discretisation of the considered inverse problem. Indeed, by considering an infinite-dimensional model, we model measurements $y_k$ that come from continuous integral transforms $y_k =\int_{0}^{1}f(x)s_{k}(x)\d x$, rather than discrete inner-products. This model is motivated by the observation
that most sensors do not compute pointwise samples of $f$, but rather integrate $f$ over
a short time or area \cite{GLPU_phantom, jones2016continuous}. Discretising the problem at a too early stage using discrete inner products can result in measurement mismatch \cite{chi2011sensitivity}. 
\end{remark}
\begin{remark}[Measurement noise]
Above, we have focused on noiseless measurements to make the mathematical model clear. However, any realistic measurement model should also incorporate noisy measurements. Our overall goal in this manuscript is to develop an algorithm that can compute matrix-vector multiplications with the matrix $P_{N}UP_{M}$ in $\mathcal{O}(N\log N)$ operations. We will, therefore, not discuss noisy measurements in any detail. We refer to the literature on each of the specific reconstruction methods for further discussions on how the methods handle noisy measurements.  
\end{remark}

\subsection{Outline of the paper}
In \S \ref{s:sub_ang} we define the subspace angle and the stable sampling rate, and we explain how these quantities dictate how we must choose $N$ in relation to $M$ to achieve stable and accurate reconstruction. This is followed by the definitions of the Walsh and Wavelet sampling bases in \S\ref{s:s_and_r_space}, along with a key lemma used extensively in the derivation of the algorithm. We then describe the algorithm in one and two dimensions in \S\ref{sec:one_dim_alg} and \S\ref{sec:two_dim_alg}, respectively, followed numerical examples in \S\ref{s:num_exp}.

\section{The subspace angle and the stable sampling rate}
\label{s:sub_ang}
It is important to realize that stable and accurate recovery in $\mathcal{R}_{M}$, from samples $y_{k} =\allowbreak \ind{f,s_{k}}$, $k=1,\ldots N$, is not possible for arbitrary choices of bases $\{s_{k}: k\in \N\}$ and $\{r_k : k \in \N\}$. What is crucial for accurate and stable recovery in $\mathcal{R}_{M}$, is that the subspace angle between $\mathcal{S}_{N}$ and $\mathcal{R}_{M}$ is sufficiently small.
\begin{definition}[Subspace angle]
    Let $\Rs_{M} = \allowbreak \Span \{r_1,\ldots r_{M}\}$ and  $\Ss_{N} = \allowbreak \Span\{s_1, \ldots, \allowbreak s_{N}\}$. The \emph{subspace angle} $\omega \in [0,\pi/2]$
    between $\Rs_{M}$ and $\Ss_{N}$ is 
    \[ \cos(\omega(\Rs_{M}, \Ss_{N})) \coloneqq \inf_{h\in \Rs_{M}, \|h\|=1} \|P_{\Ss_{N}} h\| \] 
    We set the reciprocal value as $\mu(\Rs_{M}, \Ss_{N}) \coloneqq \frac{1}{\cos(\omega(\Rs_{M}, \Ss_{N}))}$, and
    if $\cos(\omega(\Rs_{M}, \Ss_{N})) = 0$, we set $\mu(\Rs_{M}, \Ss_{N}) = \infty$.
\end{definition}
We note that a necessary condition for  $\mu(\Rs_{M},\Ss_{N}) < \allowbreak \infty$ is that $N \geq M$ (see e.g.  \cite[Thm.\ 2.1]{Tang00}). Furthermore, we have that $\mu(\mathcal{R}_{M}, \mathcal{S}_N)$, is related to the condition number of the matrix $P_{M}U^*P_{N}UP_{M}$, used for solving the normal equations in generalised sampling or the PBDW-method. Indeed, let $\sigma_1(A) \geq \cdots \geq \sigma_{M}(A)$ denote the ordered singular values of a matrix $A \in \C^{N\times M}$, with $N > M$. Then, using Parseval's identity, we have that 
\[ \cos(\omega(\Rs_{M}, \Ss_{N})) \coloneqq \inf_{h\in \Rs_{M}, \|h\|=1} \|P_{\Ss_{N}} h\| = \inf_{z\in \C^{M}, \|z\|=1} \|P_{N}UP_M z\|  = \sigma_{M}(P_N U P_M).\]
We also have that $\sigma_{1}(P_N U P_M) = \sup_{z\in \C^{M}, \|z\|=1} \|P_{N}UP_M z\|   \leq 1$, since $U$ is unitary, and hence the condition number
\[\text{cond}(P_{M}U^*P_N U P_M) 
= \frac{\sigma_{1}^{2}(P_{N}UP_M)}{\sigma_{M}^{2}(P_{N}UP_M)} 
\leq \mu^{2}(\Rs_{M}, \Ss_{N}).\]
This directly relates to the numerical stability of the normal equations, used to solve the least-squares problem \eqref{eq:GS_approx}, and compute the generalised sampling and the PBDW-method's solution.  

Furthermore, the accuracy of these two methods is also related to the subspace angle. Indeed, the constant  $C_1$ fund in the error bounds \eqref{eq:errorGS} and \eqref{eq:errorPBDW} equals $C_1 = \mu(\mathcal{R}_M, \mathcal{S}_{N})$. See \cite[Thm.\ 4.5]{Beyond13} and \cite[Eq.\ (1.7)]{Binev17} (and \cite{maday15} for earlier work). Thus both the numerical stability and accuracy of these two methods hinges on choosing $\mathcal{R}_M$ in relation to the samples one can acquire. 

The situation is the same in infinite-dimensional compressive sensing, but the quantity $\mu (\mathcal{R}_{M}, \mathcal{S}_{N})$, is camouflaged via the so-called \emph{balancing property}, introduced in \cite{adcock2016generalized}. In infinite-dimensional compressive sensing, the balancing property typically governs the required number of samples needed to satisfy the restricted isometry property (RIP) \cite{Foucart13}, and its generalisations \cite{AntunUniform, Bastounis17, Traonmilin15}, for certain constants. These constants will again affect the constants found in the error bound for the minimiser $x^{\sharp}$ of \eqref{eq:QCBP}, see, e.g., \cite{Foucart13} for details. To see the relation between the subspace angle and the balancing property, we refer to the proof of Proposition 4.4 in \cite{AntunUniform}.

From the above discussion, it is evident that the subspace angle between $\mathcal{R}_{M}$ and $\mathcal{S}_{N}$, affects both the accuracy and the stability of all the reconstruction methods. Thus, an important question is, therefore, how we should choose $N$ in relation to $M$, to ensure that $\mu(\mathcal{R}_M, \mathcal{S}_{N}) \leq \gamma$ stays bounded. This relates to the so-called \emph{stable sampling rate} \cite{adcock2016generalized, Beyond13}.

\begin{definition}[Stable sampling rate]
\label{def:ssr}
Let $\Rs_{M} = \Span \{r_1,\ldots r_{M}\}$ and $\Ss_{N} = \Span\{s_1,\ldots, s_{N}\}$. 
The \emph{stable sampling rate} 
for $M\in \mathbb{N}$ and $\gamma > 1 $ is 
\[ \Gamma(M, \gamma) = \min\{N \in \mathbb{N}: \mu(\Rs_{M}, \Ss_{N}) \leq \gamma\}. \]   
\end{definition}

For Walsh sampling and orthonormal wavelet reconstruction in $\mathcal{H} = L^{2}([0,1]^d)$, $d\geq 1$, it was shown by Hansen \& Thesing \cite{ThesingSSR} that the stable sampling rate scales linearly in $M$.  That is, for a fixed $\gamma > 1$, there exist a constant $q_{\gamma} \geq 0$ such that whenever $N = 2^{d(r+q_{\gamma})} \geq 2^{dr} = M$ for $r \in \mathbb{N}$, we have $\mu(\Rs_{M}, \Ss_{N}) \leq \gamma$. Hence for a fixed $q_{\gamma}>0$, we get a fixed upper bound on $\mu(\Rs_{M},\Ss_{N})$, for all $M$ and $N$ on the form above.  

This is important, since it tells us that for a fixed number of reconstruction coefficients $M$, we need no more than $N= C M$ samples, where $C = 2^{dq_{\gamma}}$ is a constant, to ensure that $\mu(\mathcal{R}_M, \mathcal{S}_N) \leq \gamma$. 
In Table \ref{t:sing_val}, we have computed $1/\sigma_M(P_N U P_M) = \mu(\mathcal{R}_M, \mathcal{S}_N)$, for $N=2^{dq}M$, for $d=1,2$ and $q=1,2,3,4$, for Walsh sampling and different wavelet reconstruction bases. From the table, we see that in all cases the choice $q=1$ or $q=2$ is sufficient to ensure that $1 < \gamma < 2$, indicating that the constant $C$ is not necessarily very large for these bases. 

\section{The sampling and reconstruction spaces}
\label{s:s_and_r_space}
This section introduces the necessary notation and background on the Walsh sampling basis and the orthonormal wavelet reconstruction bases. We also present a few useful results, needed to derive the final algorithm in later sections.
\subsection{Walsh functions}
\label{ss:wal}

Walsh functions (see \cite{beauchamp75} or \cite{Golubov91} for an introduction) are closely related to dyadic representations of numbers. For an integer $n \in \Z_{+} = \{0, 1, 2, \ldots\}$ its dyadic series is
$ n = n^{(1)}2^{0} + n^{(2)}2^{1} + n^{(3)}2^{2} + \cdots$,
where the $n^{(j)}$'s are binary numbers. Similarly for $x \in [0,1)$ we can express its dyadic series as $ x = x^{(1)}2^{-1} + x^{(2)}2^{-2} + x^{(3)}2^{-3} + \cdots$,
for $x^{(j)} \in \{0,1\}$. For rational numbers $x$, this expansion is not unique and in such cases we consider the expansion not ending with infinitely many repeating 1's.
 
There exist different orderings of Walsh functions, all of which leads to slightly different definitions. In this manuscript, we use the \emph{sequency} ordered Walsh functions. This ordering has the advantage that the $n$'th Walsh function $w_n$ has $n$ sign changes. 

\begin{definition}
    \label{def:walsh_func}
    Let $n \in \Z_{+}$ and $x \in [0,1)$. The \emph{Walsh function} 
    $w_n\colon [0,1) \to \{+1, -1\}$ is given by
         $w_n(x) \coloneqq (-1)^{\sum_{j=1}^{\infty} (n^{(j)} + n^{(j+1)})x^{(j)}}$
\end{definition}

We note that $\{w_{n}: n \in \mathbb{Z}_{+}\}$ is an orthonormal basis for $L^2([0,1])$, and we let 
\begin{equation*}
    \mathcal{W}f(n) = \int_{0}^{1} f(x) w_{n}(x) \d x 
\end{equation*}
denote the Walsh transform of a function $f \in L^2([0,1])$.

When working with Walsh functions, the XOR operation applied to binary sequences has many uses. We denote it by $\oplus$ and define it as follows.
\begin{definition}
    Let $x = \{x^{(j)}\}_{j=1}^{\infty} \in \{0,1\}^{\mathbb{N}}$ and $y = \{y^{(j)}\}_{j=1}^{\infty}\in \{0,1\}^{\mathbb{N}}$ be binary sequences. 
The operation $\oplus$ applied to these sequences is given by
      $  x \oplus y \coloneqq \{ |x^{(j)} - y^{(j)}|\}_{j=1}^{\infty}$. 
    For $x,y \in \Z_{+}$ or $x,y \in [0,1)$, the operation $x \oplus y$ is understood in the sense of $x$ and $y$'s representation as binary sequences.
\end{definition}

\begin{lemma}
\label{l:wal_prop}
For $x,y \in [0,1)$, $n,j,l\in \Z_{+}$, the following three equalities holds
\begin{align} 
    \label{eq:w_prop1}
    w_{n} (x\oplus y) &= w_{n}(x)w_{n}(y), \\
    \label{eq:wal_shift2}
    w_{n}(2^{-j}l) &= w_{l}(2^{-j}n)\quad\quad\text{if } n,l < 2^{j}, \\
    w_n(2^{-j}x) &= w_{\floor{n/2^j}}(x).
\end{align}
\end{lemma}

\begin{proof}
    The two first equalities can be found in any book on Walsh functions, see e.g.
\cite{Golubov91}. The last equality follows from direct computations, see e.g. \cite[Prop. 6.4]{AntunUniform}.
\end{proof}
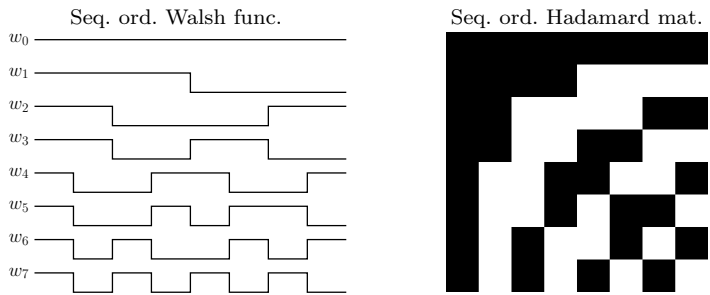
\begin{figure}
    \centering
    \begin{\textsizefig}
    \setlength{\tabcolsep}{10pt}
    \begin{tabular}{@{}>{\centering}m{0.30\textwidth}>{\centering\arraybackslash}m{0.30\textwidth}@{}}
    Seq.\ ord.\ Walsh func.  & Seq.\ ord.\ Hadamard mat. \\ 
 %   & \\
    \resizebox{\linewidth}{!}{
    \begin{tikzpicture}
        \draw[very thick]  (0, 17) -- (0.625, 17) -- (1.25, 17)  -- (1.875, 17)  -- (2.5, 17)  -- (3.125, 17)  -- (3.75, 17)  -- (4.375, 17)  -- (5, 17)  -- (5.625, 17)  -- (6.25, 17)  -- (6.875, 17)  -- (7.5, 17)  -- (8.125, 17)  -- (8.75, 17)  -- (9.375, 17)  -- (10, 17) ;
        \draw (-0.5, 17) node {{\LARGE $w_0$}};
        \draw[very thick]  (0, 15.9233) -- (0.625, 15.9233) -- (1.25, 15.9233)  -- (1.875, 15.9233)  -- (2.5, 15.9233)  -- (3.125, 15.9233)  -- (3.75, 15.9233)  -- (4.375, 15.9233)  -- (5, 15.9233)  -- (5, 15.2983) -- (5.625, 15.2983) -- (6.25, 15.2983)  -- (6.875, 15.2983)  -- (7.5, 15.2983)  -- (8.125, 15.2983)  -- (8.75, 15.2983)  -- (9.375, 15.2983)  -- (10, 15.2983) ;
        \draw (-0.5, 15.9233) node {{\LARGE $w_1$}};
        \draw[very thick]  (0, 14.8467) -- (0.625, 14.8467) -- (1.25, 14.8467)  -- (1.875, 14.8467)  -- (2.5, 14.8467)  -- (2.5, 14.2217) -- (3.125, 14.2217) -- (3.75, 14.2217)  -- (4.375, 14.2217)  -- (5, 14.2217)  -- (5.625, 14.2217)  -- (6.25, 14.2217)  -- (6.875, 14.2217)  -- (7.5, 14.2217)  -- (7.5, 14.8467) -- (8.125, 14.8467) -- (8.75, 14.8467)  -- (9.375, 14.8467)  -- (10, 14.8467) ;
        \draw (-0.5, 14.8467) node {{\LARGE $w_2$}};
        \draw[very thick]  (0, 13.77) -- (0.625, 13.77) -- (1.25, 13.77)  -- (1.875, 13.77)  -- (2.5, 13.77)  -- (2.5, 13.145) -- (3.125, 13.145) -- (3.75, 13.145)  -- (4.375, 13.145)  -- (5, 13.145)  -- (5, 13.77) -- (5.625, 13.77) -- (6.25, 13.77)  -- (6.875, 13.77)  -- (7.5, 13.77)  -- (7.5, 13.145) -- (8.125, 13.145) -- (8.75, 13.145)  -- (9.375, 13.145)  -- (10, 13.145) ;
        \draw (-0.5, 13.77) node {{\LARGE $w_3$ }};
        \draw[very thick]  (0, 12.6933) -- (0.625, 12.6933) -- (1.25, 12.6933)  -- (1.25, 12.0683) -- (1.875, 12.0683) -- (2.5, 12.0683)  -- (3.125, 12.0683)  -- (3.75, 12.0683)  -- (3.75, 12.6933) -- (4.375, 12.6933) -- (5, 12.6933)  -- (5.625, 12.6933)  -- (6.25, 12.6933)  -- (6.25, 12.0683) -- (6.875, 12.0683) -- (7.5, 12.0683)  -- (8.125, 12.0683)  -- (8.75, 12.0683)  -- (8.75, 12.6933) -- (9.375, 12.6933) -- (10, 12.6933) ;
        \draw (-0.5, 12.6933) node {{\LARGE $w_4$}};
        \draw[very thick]  (0, 11.6167) -- (0.625, 11.6167) -- (1.25, 11.6167)  -- (1.25, 10.9917) -- (1.875, 10.9917) -- (2.5, 10.9917)  -- (3.125, 10.9917)  -- (3.75, 10.9917)  -- (3.75, 11.6167) -- (4.375, 11.6167) -- (5, 11.6167)  -- (5, 10.9917) -- (5.625, 10.9917) -- (6.25, 10.9917)  -- (6.25, 11.6167) -- (6.875, 11.6167) -- (7.5, 11.6167)  -- (8.125, 11.6167)  -- (8.75, 11.6167)  -- (8.75, 10.9917) -- (9.375, 10.9917) -- (10, 10.9917) ;
        \draw (-0.5, 11.6167) node {{\LARGE $w_5$}};
        \draw[very thick]  (0, 10.54) -- (0.625, 10.54) -- (1.25, 10.54)  -- (1.25, 9.915) -- (1.875, 9.915) -- (2.5, 9.915)  -- (2.5, 10.54) -- (3.125, 10.54) -- (3.75, 10.54)  -- (3.75, 9.915) -- (4.375, 9.915) -- (5, 9.915)  -- (5.625, 9.915)  -- (6.25, 9.915)  -- (6.25, 10.54) -- (6.875, 10.54) -- (7.5, 10.54)  -- (7.5, 9.915) -- (8.125, 9.915) -- (8.75, 9.915)  -- (8.75, 10.54) -- (9.375, 10.54) -- (10, 10.54) ;
        \draw (-0.5, 10.54) node {{\LARGE $w_6$}};
        \draw[very thick]  (0, 9.46333) -- (0.625, 9.46333) -- (1.25, 9.46333)  -- (1.25, 8.83833) -- (1.875, 8.83833) -- (2.5, 8.83833)  -- (2.5, 9.46333) -- (3.125, 9.46333) -- (3.75, 9.46333)  -- (3.75, 8.83833) -- (4.375, 8.83833) -- (5, 8.83833)  -- (5, 9.46333) -- (5.625, 9.46333) -- (6.25, 9.46333)  -- (6.25, 8.83833) -- (6.875, 8.83833) -- (7.5, 8.83833)  -- (7.5, 9.46333) -- (8.125, 9.46333) -- (8.75, 9.46333)  -- (8.75, 8.83833) -- (9.375, 8.83833) -- (10, 8.83833) ;
        \draw (-0.5, 9.46333) node {{\LARGE $w_7$}};
               
    \end{tikzpicture}
    }
    &
     \includegraphics[width=0.75\linewidth]{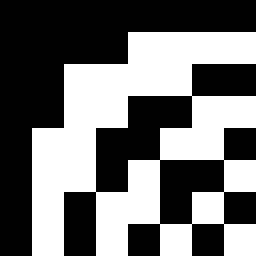} 
    \end{tabular}
    \end{\textsizefig}
    \vspace{-3mm}
    \caption{\label{f:Walsh_had}
(\textbf{Relation between Walsh functions and Hadamard matrices}). Left: The eight first sequency ordered Walsh functions. Right: A $8\times 8$ sequency ordered Hadamard matrix, where black corresponds to $1$ and white to $-1$. We can see that the Walsh functions' sign changes correspond to the sign changes in the matrix. }
%\vspace{-4mm}
\end{figure}

We also note that Walsh functions and Hadamard matrices are closely related and the 
$(n,k)$'th entry of a sequency ordered Hadamard matrix $H \in \R^{2^{j}\times 2^{j}}$ is given by $w_{n-1}(2^{-j}(k-1))$. See Figure \ref{f:Walsh_had} for an illustration of this relationship. Furthermore, for $N=2^j$ we note that a matrix-vector product with $H$ can be computed in $\mathcal{O}(N\log N)$ operations using the fast Walsh-Hadamard transform (FWHT) \cite{beauchamp75}. That is, for $x = \{x_k\}_{k=1}^{N}$, the $N$ sums
\begin{align*}
    \left\{\sum_{k=1}^{N} w_{n}((k-1)/N) x_{k} \right\}_{n=0}^{N-1} 
\end{align*}
can utilize the FWHT algorithm to compute the result with $\mathcal{O}(N\log N)$ operations, and without storing the matrix $H$ in memory.

\subsection{Wavelets}
\label{subsec:wavelets1}

Let $\phi \colon \R \to \R$ and $\psi \colon \R \to \R$ be a compactly supported orthonormal scaling function and wavelet \cite{Daubechies92}, respectively, corresponding to an multiresolution analysis (MRA). We say that the wavelet $\psi$ has $\nu$ vanishing moments if it is orthogonal to all polynomials of degree $\nu-1$. That is, if $\indi{x^k, \psi}_{L^2(\mathbb{R})} = 0$ for $k=0,\ldots, \nu-1$. 
For simplicity, we work with wavelets with minimal support. Thus, for $\nu=1$ the above wavelet is the Haar wavelet, but for $\nu \geq 2$ there are different choices, ranging from the classical
\textit{Daubechies wavelet} (which has minimum-phase) to
\textit{symlets} which are close to being symmetric, but with a larger phase
\cite[p.\ 294]{Mallat09}.

If $\psi$ generates a system of orthonormal wavelets with $\nu$ vanishing moments and minimal support, then the support of $\psi$ and $\phi$ is an interval of size $2\nu-1$. For convenience, we use the convention that 
$\Supp (\phi) = \Supp (\psi) = [-\nu+1, \nu]$.

Let 
$\phi_{j,m} (x) \coloneqq 2^{j/2} \phi(2^j x -m )$ and 
$\psi_{j,m} (x) \coloneqq 2^{j/2} \psi (2^j x - m)$
denote the dilated and translated versions of $\phi$ and $\psi$. To work on the interval $[0,1]$, we need to construct bases on this interval consisting of functions $\phi_{j,m}$ and $\psi_{j,m}$, with $j \geq J_0$ for some $J_0$, chosen so that $\Supp(\phi_{j,m}) =\Supp(\psi_{j,m}) \subset [0,1]$ for at least one choice of $m$. It is readily seen that if $J_0 \geq \ceil{\log_{2}(2\nu)}$ for $\nu\geq 2$ and $J_0 \geq 0$ for $\nu=1$, then this holds for at least one $m$.

Constructing an orthonormal wavelet basis on the interval requires special care at the boundaries, and it is common to replace all wavelets and scaling functions intersecting the boundary with certain \enquote{replacement} functions.  Hence for $j \geq J_0$ we define the set of functions
\begin{equation*}
    \begin{split}
B_{\phi,j} &= \left\{ \phi_{j,m}^{\text{rep}} \right\}_{m=0}^{\nu-1} \bigcup
\left\{\phi_{j,m}\right\}_{m=\nu}^{2^j-\nu-1} \bigcup
\left\{\phi_{j,m}^{\text{rep}}\right\}_{m=2^j-\nu}^{2^j-1}, \\
B_{\psi,j} &= \left\{ \psi_{j,m}^{\text{rep}} \right\}_{m=0}^{\nu-1} \bigcup
\left\{\psi_{j,m}\right\}_{m=\nu}^{2^j-\nu-1} \bigcup
\left\{\psi_{j,m}^{\text{rep}}\right\}_{m=2^j-\nu}^{2^j-1}   
    \end{split}
    \label{eq:wave_bases}
\end{equation*}
where $\psi_{j,m}^{\text{rep}}$ and $\phi_{j,m}^{\text{rep}}$ are replacement wavelets and scaling functions supported on $[0,1]$.  There are several ways to construct these replacement functions so that they retain the orthonormality condition, and we consider both a periodic boundary extension and the vanishing moments preserving (VMP) boundary wavelets introduced by Cohen, Daubechies \& Vial in \cite{Cohen93}.

The advantage of the former is that it is both easy to define and implement. Indeed, to compute a discrete wavelet transform (DWT) using a periodic boundary extension, one simply use a periodic convolutions between between the filters and the signal. The disadvantage of the periodic wavelets basis is that we lose the vanishing moments property at the boundaries.  This may result in a few high amplitude coefficients at each scale.  Another issue with these wavelets is that any $\ell^2$-approximation of a non-periodic function on $[0,1]$ will have certain artefacts at the boundaries due to the underlying assumption of periodicity.

This can be seen in Figure \ref{fig:per_vs_bd}, where we consider a generalised sampling reconstructions of the periodic function $f(t) =\cos(2\pi t)$ and non-periodic function $g(t) = \cos(2\pi t)+t$, on $[0,1]$. With a periodic wavelet basis, we achieve high accuracy for the periodic function $f(t)$, whereas we get artefacts at the boundaries when we reconstruct $g(t)$, due to the underlying periodic assumption. 

The vanishing moments preserving boundary extension introduced in \cite{Cohen93} circumvents this issue by designing special wavelets at the boundaries, which retain both orthonormality, vanishing moments and avoids any assumptions about periodicity. However, as pointed out by Antun \& Ryan in \cite{AntunRyan}, most wavelet libraries do not support these wavelets. In \cite{Gataric16} Gataric \& Poon extended the WaveLab library \cite{wavelab} with a special set of Daubechies wavelets. In this work, we use the implementation from \cite{AntunRyan}, to also include orthonormal wavelets such as symlets. 

\begin{figure}[tb]
    {\centering
    \begin{\textsizefig}
    \begin{tabular}{@{}c@{}}
    \setlength{\tabcolsep}{15pt}
    \begin{tabular}{@{}>{\centering}m{0.30\textwidth}>{\centering\arraybackslash}m{0.30\textwidth}@{}}
        Periodic function $f$ & Non-periodic function $g$ \\
        \includegraphics[width=\linewidth]{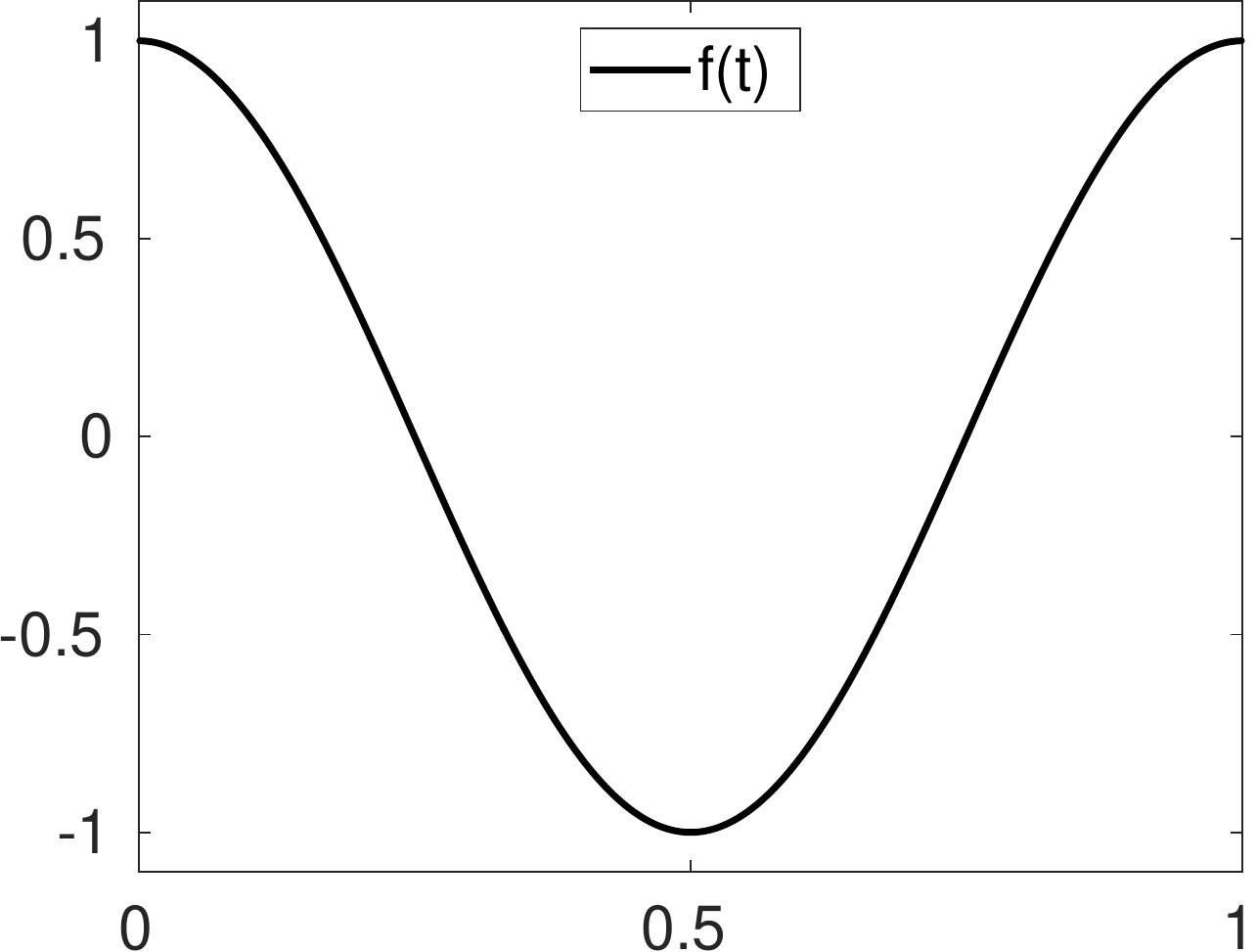} 
       &\includegraphics[width=\linewidth]{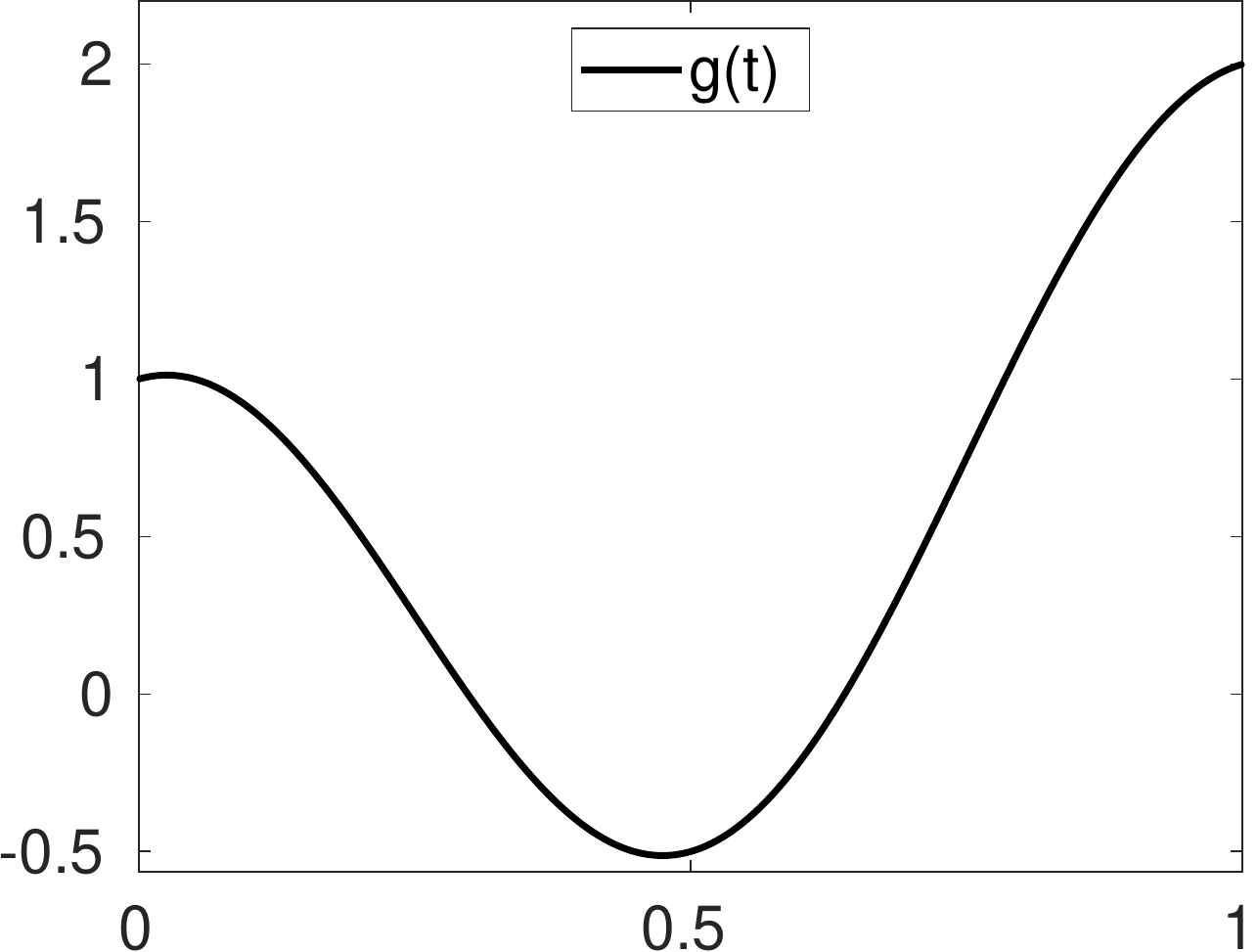} \\ 
    \end{tabular}
    \\
    \begin{tabular}{@{}>{\centering}m{0.30\textwidth}>{\centering}m{0.30\textwidth}>{\centering\arraybackslash}m{0.30\textwidth}@{}}
      GS rec. of $f$ using a periodic wavelet basis 
    & GS rec. of $g$ using a periodic wavelet basis 
    & GS rec. of $g$ using a VMP wavelet basis\\ 
        \includegraphics[width=\linewidth]{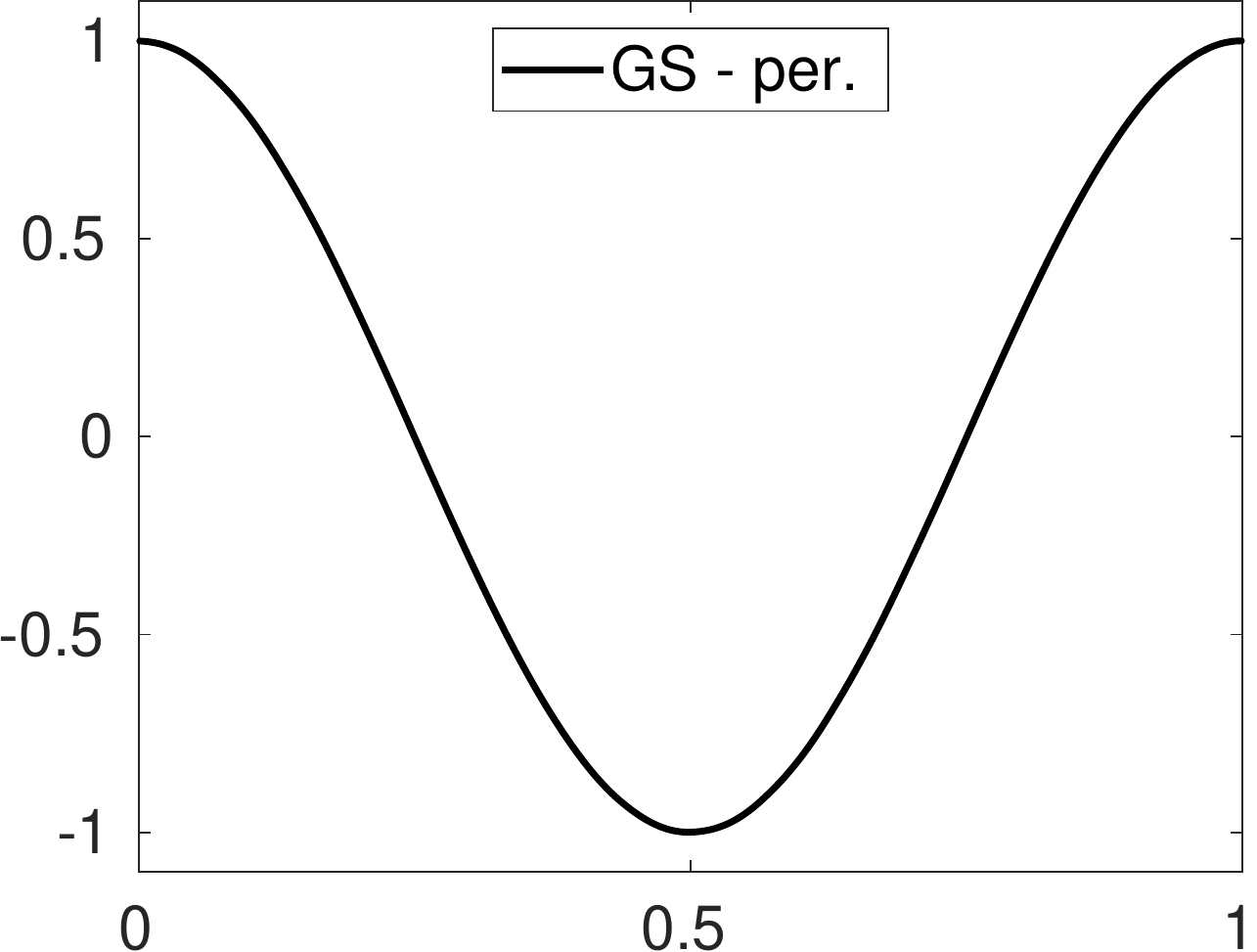} 
       &\includegraphics[width=\linewidth]{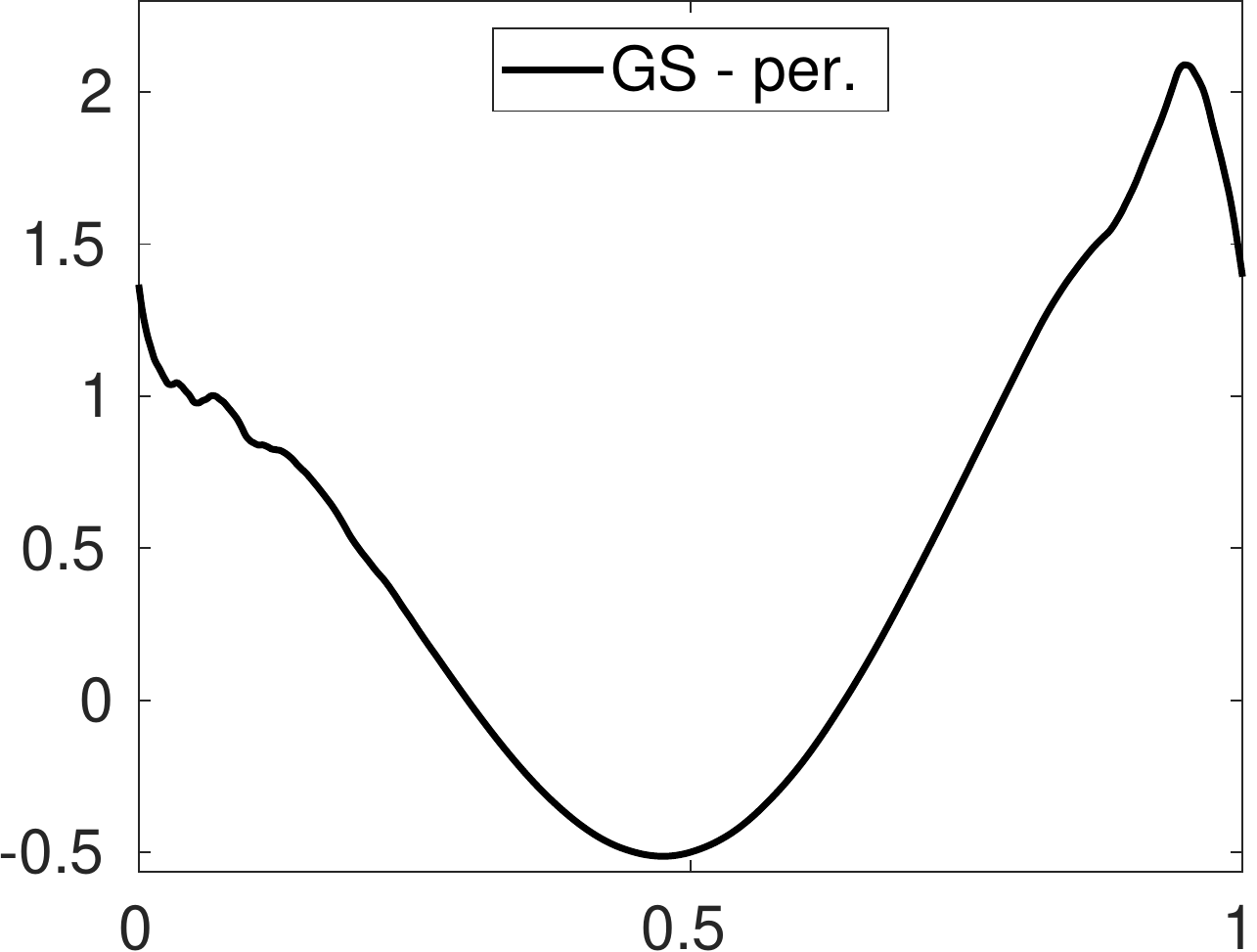} 
       &\includegraphics[width=\linewidth]{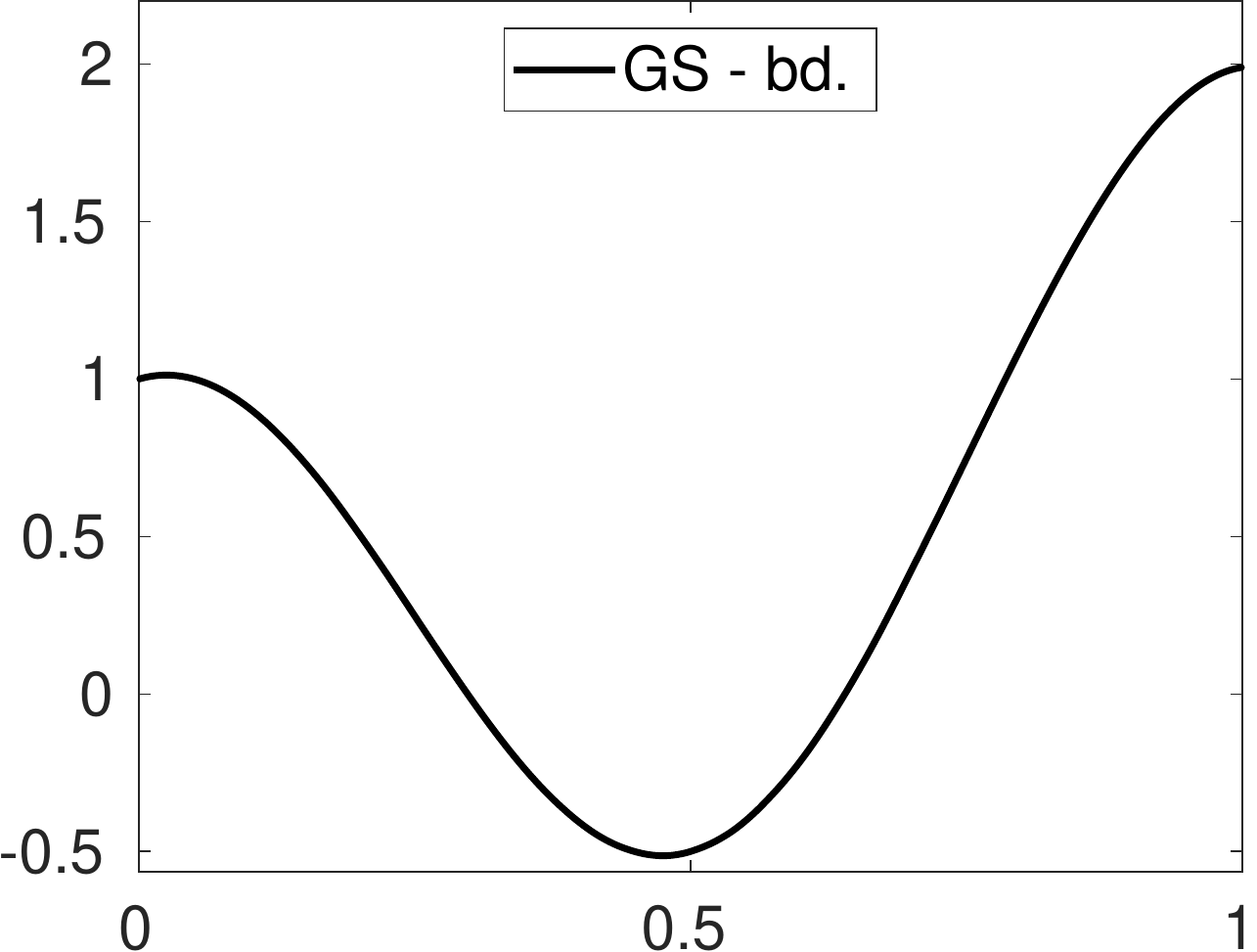} 
    \end{tabular}
    \end{tabular}
    \end{\textsizefig}
    }
    \vspace{-3mm}
    \caption{\label{fig:per_vs_bd}(\textbf{Periodic boundary extension is best suited for periodic functions}). 
We consider the two functions $f(t) = \cos(2\pi t)$  and $g(t) = \cos(2\pi t) + t$ (top row). Given $N=32$ Walsh samples from these functions, we use generalized sampling (GS) with $M=16$ DB4 wavelet-basis functions with different boundary extensions for reconstruction. For the reconstruction basis, we use a periodic boundary extension (bottom row, left and middle) and vanishing moments preserving (VMP) boundary extension (bottom row, right). As we can see from the middle figure at the bottom row, the periodic boundary wavelets create artefacts at the boundaries when approximating a non-periodic function. The VMP boundary wavelets circumvent this issue.  
}
%\vspace{-4mm}
\end{figure}

For the periodic wavelet basis, we extend the wavelets and scaling functions at the boundaries periodically. That is, we let
\begin{align*}
    \phi_{j,m}^{\text{per}} &= \phi_{j,m}\lvert_{[0,1]} + \phi_{j,2^{j} + m}\lvert_{[0,1]} 
    &&\text{for } m = 0,\ldots, \nu-1,  \\
    \phi_{j,m}^{\text{per}} &= \phi_{j,m}\lvert_{[0,1]} + \phi_{j,m - 2^{j}}\lvert_{[0,1]} 
    &&\text{for } m = 2^{j}-\nu, \ldots, 2^{j}-1,
\end{align*}
and similar for $\psi^{\text{per}}_{j,m}$. Here $\lvert_{[a,b]}$ means the restriction to the interval $[a,b]$. Strictly speaking, we could have omitted the definition of $\phi_{j,\nu}^{\text{per}}, \psi_{j,\nu}^{\text{per}}$ and $\phi_{j,2^{j}-\nu}^{\text{per}}$, $\psi_{j,2^{j}-\nu}^{\text{per}}$, as these function are pure interior functions, but we define these functions to unify the notation with the vanishing moments preserving boundary wavelets,

In \cite{Cohen93} one constructs special boundary wavelets and scaling functions $\phi_{m}^{\text{left}}$, $\psi_{m}^{\text{left}}$, $\phi_{m}^{\text{right}}$, and $\psi_{m}^{\text{rigth}}$, for $m=0,\ldots, \nu-1$. These functions are created using finite linear combinations of the interior functions, and their supports are staggered. That is 
$\Supp(\phi^{\text{left}}_{m}) = [0,\nu+m]$ and $\Supp(\phi_{m}^{\text{right}}) = [-m-\nu,0]$
and similar for $\psi^{\text{left}}_{m}$ and $\psi_{m}^{\text{right}}$. The corresponding boundary functions (similar for the wavelets) are defined as 
\begin{equation*}
\label{eq:bdwave}
\begin{split}
\phi_{j,m}^{\text{bd}}(x) &= 2^{j/2}\phi_{m}^{\text{left}}(2^{j} x) 
\quad\quad\quad\quad\quad\quad ~ \text{for }m = 0,\ldots, \nu-1, \\ 
\phi_{j,m}^{\text{bd}}(x) &= 2^{j/2}\phi_{2^{j}-1-m}^{\text{right}}(2^{j}(x-1)) 
\quad\quad\text{for }m = 2^{j}-\nu, \ldots, 2^{j}-1. 
\end{split}
\end{equation*}
With these functions well defined, we let \enquote{rep}, mean either \enquote{per} or \enquote{bd}.  

Let $\mathcal{V}_{j} = \Span\{ B_{\phi, j}\}$ and $  \mathcal{U}_{j} = \Span\{B_{\psi,j}\}$, and note that by construction these satisfy $\mathcal{V}_{j}
\oplus \mathcal{U}_{j} = \mathcal{V}_{j+1}$. 
Now, let $C_{\psi, j} = B_{\phi, J_0} \cup B_{\psi, J_0} \cup \cdots \cup B_{\psi, j-1}$.
It should be clear from the previous discussion that $B_{\phi,j}$ and $C_{\phi, j}$ span the same space.  We can perform a change-of-basis between the two bases using a DWT matrix $W \in \R^{2^j \times 2^j}$. 

Finally, note that there no closed-form formula exists for the compactly supported orthonormal wavelets considered (except for the Haar wavelet). We can, however, compute approximations to $\phi(2^{j}k)$ and $\psi(2^{j}k)$, at dyadic grid points using the cascade algorithm \cite{Daubechies92}. 

\subsection{A useful lemma}

Before we proceed, we prove a lemma that lays the foundation for the fast computations derived in the following sections. We note that the lemma is a generalisation of what is used in the proof of Lemma 6.6 in \cite{AntunUniform}.

\begin{lemma}
\label{lem:fjk}
   Let $h \in L^{2}(\R)$ with $\Supp (h) \subset [a,b]$ for integers $a \leq 0 <  b$. Denote by
$h_{j,m}(x) = 2^{j/2}h(2^{j}x - m)$ a translated and dilated version of $h$. Suppose that $j,m \in \Z$ are chosen so that
$\Supp(h_{j,m}) \subset [0,1]$. Then
\[ 
    \ind{h_{j,m}, w_n} 
    = 2^{-j/2} \sum_{l=a}^{b-1} w_{n}\left( \frac{l+m}{2^{j}} \right) 
    \mathcal{W}h_{0,-l}\lvert_{[0,1)}\left(\floor{2^{-j}n}\right).
\] 
\end{lemma}
\begin{proof}
    First notice that by assumption we have that $\Supp(h_{j,m}) \subset [2^{-j}(a +m), 2^{-j}(b+m)] \subset [0,1]$. This implies that $b-a \leq 2^j$, and that $m \in \{-a, -a+1,\ldots, 2^{j} - b\} \subset \{0, \ldots, 2^j -1\}$, where have used the assumption $a\leq 0 < b$, in the final inclusion. Next notice that
\begin{equation}\label{eq:x_frac}
\begin{split}
   \frac{x}{2^j} + \frac{m}{2^{j}} &= \sum_{i=j}^{\infty} x^{(i-j+1)}2^{-i-1} + \sum_{i=1}^{j} m^{(i)}2^{-j-1+i} \\
    &= \sum_{i=j}^{\infty} x^{(i-j+1)}2^{-i-1} \oplus \sum_{i=1}^{j}m^{(i)}2^{-j-1+i}
     =\frac{x}{2^j} \oplus \frac{m}{2^{j}}. 
\end{split}
\end{equation}
Utilising  \eqref{eq:x_frac} and Lemma \ref{l:wal_prop}, now give 
{ \allowdisplaybreaks
\begin{align*}
    \ind{h_{j,m}, w_{n}} &= \int_{0}^{1} 2^{j/2}h(2^j x -m) w_{n}(x) \d x \\
    &= \sum_{l=a}^{b-1} \int_{2^{-j}(l+m)}^{2^{-j}(l+1+m)} 2^{j/2}h(2^jx-m)w_{n}(x) \d x \\
    &= \sum_{l=a}^{b-1} \int_{l+m}^{l+1+m} 2^{-j/2}h\left(x-m\right)w_{n}\(\frac{x}{2^j}\) \d x \\
    &= \sum_{l=a}^{b-1} \int_{0}^{1} 2^{-j/2}h\(x+l\)w_{n}\(\frac{x+l+m}{2^j}\) \d x \\
    &= \sum_{l=a}^{b-1} \int_{0}^{1} 2^{-j/2}h\(x+l\)w_{n}\(\frac{x}{2^{j}} \oplus \frac{l+m}{2^j}\) \d x \\
    &= 2^{-j/2} \sum_{l=a}^{b-1} w_{n}\( \frac{l+m}{2^{j}} \) 
                 \mathcal{W}h_{0,-l}\lvert_{[0,1)}\(\floor{2^{-j}n}\).
\end{align*} 
}
\end{proof}

\section{The one dimensional algorithm}
\label{sec:one_dim_alg}

Next, we describe an algorithm for computing a matrix-vector multiplication 
with the matrix
\begin{equation}
\label{eq:vec_mat_prod}
P_{N}UP_{M}= \begin{bmatrix} \indi{\phi_{j,0}^{\text{rep}}, w_0} & \cdots & \indi{\phi_{j,M-1}^{\text{rep}}, w_0} \\
    \vdots  & \ddots & \vdots \\
    \indi{\phi_{j,0}^{\text{rep}}, w_{N-1}} & \cdots & \indi{\phi_{j,M-1}^{\text{rep}}, w_{N-1}} \\
\end{bmatrix}
\end{equation}
and its adjoint, using $\mathcal{O}(N\log N)$ operations and without explicitly storing the matrix \eqref{eq:vec_mat_prod} in memory. 
Throughout, we let $M = 2^j$ and $N = 2^{j+q}$ where $j\geq J_0$ and $q > 0$ are integers. Other values of $M$ and $N$ can be considered by ultilizing appropriate zero padding. 
Below, we describe the algorithm stepwise by defining different operators, which we combine to achieve the desired matrix-vector multiplication. The complete algorithm is summarised in Algorithm \ref{alg:one_dim}. 

\begin{remark}[On the scaling between $N$ and $M$]\label{r:scalingMN}
In Table \ref{t:sing_val} we have computed the ratio between the largest and smallest singular value of the matrix $P_{N}UP_M$, for different wavelets and choices for $q$, both in one and two dimensions. We observe that in all cases, the matrix is well-conditioned for the simplest choice of $q=1$. This corresponds to $N=2M$ in one dimension and $N=4M$ in two dimensions. Moreover, since we know that the stable sampling rate for Walsh sampling and wavelet reconstruction is linear, we have that $N=\mathcal{O}(M)$, with a reasonable constant. 
\end{remark}

\begin{remark}[Applications to compressive sensing]
Note that a sparse representation of $f$ is needed for compressive sensing to achieve successful recovery. For this method it is, therefore,  better to represent an approximation to $f$ in the basis $C_{\psi,j}$, than the $B_{\phi,j}$ basis used above. Changing the basis can easily be achieved by using the matrix $P_N U P_M W^{-1}$, where $P_{N}UP_{M}$ is as above, and $W^{-1} \in \C^{M\times M}$ is the inverse discrete wavelet transform (IDWT). As $W^{-1}$ is a change of basis matrix from $C_{\psi,j}$ to $B_{\phi,j}$, this matrix will simulate the desired matrix if $C_{\psi,j}$ is the reconstruction basis. Furthermore, the cost of applying $W^{-1}$ is $\mathcal{O}(M)$, using the cascade algorithm. This means that the overall cost of the matrix-vector multiplication does not grow by applying this change-of-basis.
\end{remark}

\begin{remark}[Haar wavelet reconstruciton]
For $N=M=2^j$, the Haar wavelet basis and Walsh sampling basis, span the same space. For the Haar reconstruction basis, there is, therefore, no benefit of applying generalised sampling or the PBDW-method for reconstruction. Compressive sensing, on the other hand, can be applied since it allows for reconstruction of $M$ wavelet coefficients from $m < M$ samples, under the assumption of sparsity. Since many natural images are sparse in the Haar wavelet basis, this approach is widely studied, see e.g. \cite{7446327, 9088137, Terhaar18}. For Walsh sampling and Haar wavelet reconstruction using the basis $C_{\psi, j}$, the truncated change-of-basis matrix $P_{M}UP_M = HW^{-1}$, where $W^{-1}$ is the Haar IDWT matrix, and $H$ is the Hadamard matrix. This matrix can be computed using fast transforms with the FWHT and DWT algorithms. Below, we do, therefore, not consider Haar wavelet reconstruction.  
\end{remark}

\begin{table}
\begin{center}
 \begin{\textsizefig}
    \begin{tabular}{@{}c@{}}
    \textbf{The value of } $\boldsymbol{\mu(\mathcal{R}_{M},\mathcal{S}_N) = 1/\cos(\omega (\mathcal{R}_{M},\mathcal{S}_N)) = 1 /\sigma_M(P_NUP_M)}$ \\  
    
    \begin{tabular}{@{}>{\centering}m{0.49\textwidth}>{\centering\arraybackslash}m{0.49\textwidth}@{}}
\textbf{One dimension} & \textbf{Two dimensions}\\
$M=2^7$, $N=2^{7+q}$, $L^2([0,1])$ & $M=2^{2\cdot 5}$, $N=2^{2(5+q)}$, $L^2([0,1]^2)$ \\
        \begin{tabular}{@{}>{\centering}m{0.11\textwidth}>{\centering}m{0.07\textwidth}>{\centering}m{0.07\textwidth}>{\centering}m{0.07\textwidth}>{\centering\arraybackslash}m{0.07\textwidth}@{}}
            \toprule
            \textbf{Wavelet} & $\boldsymbol{q=1}$  & $\boldsymbol{q=2}$  & $\boldsymbol{q=3}$  & $\boldsymbol{q=4}$ \\
            \midrule 
DB2 &  1.200 &  1.050 &  1.014 &  1.004 \\ 
DB3 &  2.610 &  1.135 &  1.028 &  1.006 \\ 
DB4 &  1.251 &  1.068 &  1.023 &  1.007 \\ 
DB5 &  1.392 &  1.109 &  1.025 &  1.011 \\ 
DB6 &  6.499 &  1.137 &  1.033 &  1.016 \\ 
sym2 &  1.200 &  1.050 &  1.014 &  1.004 \\ 
sym3 &  2.610 &  1.135 &  1.028 &  1.006 \\ 
sym4 &  1.188 &  1.037 &  1.008 &  1.003 \\ 
sym5 &  1.179 &  1.042 &  1.013 &  1.005 \\ 
sym6 &  1.300 &  1.059 &  1.015 &  1.005 \\ 
            \bottomrule 
        \end{tabular}
&       
        \begin{tabular}{@{}>{\centering}m{0.11\textwidth}>{\centering}m{0.07\textwidth}>{\centering}m{0.07\textwidth}>{\centering}m{0.07\textwidth}>{\centering\arraybackslash}m{0.07\textwidth}@{}}
            \toprule
            \textbf{Wavelet} & $\boldsymbol{q=1}$  & $\boldsymbol{q=2}$  & $\boldsymbol{q=3}$  & $\boldsymbol{q=4}$ \\
            \midrule 
DB2 &  1.439 &  1.102 &  1.028 &  1.008 \\ 
DB3 &  6.814 &  1.289 &  1.057 &  1.013 \\ 
DB4 &  1.565 &  1.141 &  1.047 &  1.013 \\ 
DB5 &  1.937 &  1.230 &  1.050 &  1.022 \\ 
DB6 & 42.233 &  1.292 &  1.068 &  1.032 \\ 
sym2 &  1.439 &  1.102 &  1.028 &  1.008 \\ 
sym3 &  6.814 &  1.289 &  1.057 &  1.013 \\ 
sym4 &  1.412 &  1.075 &  1.016 &  1.006 \\ 
sym5 &  1.389 &  1.085 &  1.026 &  1.009 \\ 
sym6 &  1.690 &  1.121 &  1.029 &  1.009 \\ 
            
            \bottomrule 
        \end{tabular}
\\
    \end{tabular}
    \end{tabular}
    \end{\textsizefig}
    \end{center}
\vspace{1mm}
\caption{\label{t:sing_val}
We compute the fraction $1 / \sigma_M(P_NUP_M) = \mu(\mathcal{R}_M,\mathcal{S}_N)$, for the matrix $P_{N}UP_{M}$, where $U$ is the change-of-basis matrix between a Walsh sampling basis and an orthonormal wavelet basis with vanishing moments preserving boundary wavelets. We consider both one and two-dimensional bases. We see that in all the considered cases, the smallest singular value is reasonably close to 1, indicating good conditioning of the matrix. Here DB$X$ and sym$X$, refer to a Daubechies or symlet wavelet, respectively, with $X$ vanishing moments.} 
\vspace{-6mm}
\end{table}

\subsection{The forward operation}
\label{subsec:forward}
The wavelet basis $B_{\phi,j}$ with $\nu > 1$ vanishing moments, consists of three types of wavelets, the left boundary corrected wavelets, interior wavelets and the right boundary corrected wavelets. The matrix-vector multiplication $P_{N} U P_M\xi$ for $\xi \in \C^{M}$ is, therefore, naturally divided into the three sums
\begin{equation} \label{eq:sum_key_ind}
    \sum_{m=0}^{\nu-1} \indi{\phi_{j,m}^{\text{rep}}, w_n}\xi_{m} +  
    \sum_{m=\nu}^{M-\nu-1} \indi{\phi_{j,m}, w_n}\xi_{m} +  
    \sum_{m=M-\nu}^{M-1} \indi{\phi_{j,m}^{\text{rep}}, w_n}\xi_{m}  
\end{equation}
for each $0\leq n < N$. In this subsection we foucs on how to speed up the computations of the middle summand, as a naive implementation would require $\mathcal{O}(MN)$ operations.  Throughout we take $\nu$ to be some small fixed number, usually in the range \{2, \ldots, 8\}, and we omitt the dependence on $\nu$, whenever we summarize the computatinal cost of the algorithm.  The first and thrid summand require $\mathcal{O}(N\nu) = \mathcal{O}(N)$ operations each, and their dependence is therefore  independent of $M$. We consider the edge scaling functions in \S\ref{subsec:edge}. 

 We start by applying Lemma \ref{lem:fjk} to the middle summand in \eqref{eq:sum_key_ind}. This gives
{ \allowdisplaybreaks
\begin{align*}
        \sum_{m=\nu}^{M-\nu-1} \indi{\phi_{j,m}, w_n} \xi_m  
        &= 2^{-j/2}\sum_{m=\nu}^{M-\nu-1}   
           \sum_{l=-\nu+1}^{\nu-1} \mathcal{W}\phi_{0,-l}\lvert_{[0,1)} (\floor{2^{-j}n}) 
           w_{n}\(\tfrac{l+m}{2^{j}}\) \xi_{m}\\
        &=  2^{-j/2} \sum_{l=-\nu+1}^{\nu-1} \mathcal{W}{\phi}_{0,-l}\lvert_{[0,1)} (\floor{2^{-j}n})
            \sum_{m=\nu}^{M-\nu-1} w_{n}\(\tfrac{l+m}{2^{j}}\) \xi_{m},  
\end{align*}   
}
Recall that $M=2^j$ and $N=2^{j+q}$, and define the linear operator $H_{l} \colon \R^{M} \to \R^N$ by
\begin{equation*}\label{eq:H_l}
H_{l}(\xi) =  \left[ \sum_{m=\nu}^{M-\nu-1}  w_{n}\(\frac{2^q(l+m)}{N}\) \xi_{m}\right]_{n=0}^{N-1}, \quad \xi \in \R^{M},
\end{equation*}
and the linear operator 
$D_l\colon \R^{N} \to \R^{N}$ by 
\begin{equation*}
    D_{l}(\alpha) = \left[ 2^{-j/2}\mathcal{W}\phi_{0,-l}\lvert_{[0,1)}
                       (\floor{2^{-j}n})\alpha_{n} \right]_{n=0}^{N-1}, \quad \alpha \in \R^{N}.
\end{equation*}
Combining these opertors, we can write the middle sum in \eqref{eq:sum_key_ind} as 
\begin{equation}
\label{eq:sum_operators}
\left[ \sum_{m=\nu}^{M-\nu-1} \indi{\phi_{j,m}, w_n}\xi_{m} \right]_{n=0}^{N-1} =
\sum_{l=-\nu+1}^{\nu-1} D_{l}(H_{l}(\xi)).
\end{equation}

Note that $H_l$ can be implemented by embedding  $\xi \in \R^M$ in a zero-padded vector of length $N$, and apply an $N\times N$ fast Walsh-Hadamard transform. Thus, evaluatning $H_{l}$ can be done in $\mathcal{O}(N \log N)$ operations. Also notice that the coefficients  $\mathcal{W}\phi_{0,-l}\lvert_{[0,1)}(\floor{2^{-j}n})$ are independent of the input, and can be computed a priori. This reduced the cost of evaluating $D_{l}$ to at most $\mathcal{O}(N)$ operations.
The cost of computing \eqref{eq:sum_operators} is, therefore, $\mathcal{O}(N \log N)$.

Also note that $n < N=2^{j+q}$, and implying that $\floor{2^{-j}n} \leq 2^{q}-1$. This means that for each $l$ we  only compute $\mathcal{W}\phi_{0,-l}(s)\vert_{[0,1)}$ for $s=0,\ldots,2^{q}-1$. Furthermore, from \S\ref{s:sub_ang} we know that for a fixed $\gamma > 1$ the stable sampling rate scales linearly. Hence for fixed $q$ we may vary $j$ without affecting the stable sampling rate. This implies that we only need to precompute these coefficients for some $q>q_{\gamma'}$ where $\gamma'>1$ is the smallest stable sampling rate of interest. Moreover, from Table \ref{t:sing_val} we see that even the simplest choice of $q=1$, results in $1 < \gamma < 2$, in many  cases.

\subsection{The adjoint operation}
Next we consider the matrix-vector multiplication $P_{M}U^*P_{N}\alpha$ for $\alpha \in \C^{N}$. Since the computational burder is on the $M-2\nu$ middle columns, we once more foucs on these and prostpone the edge wavelet functions, until \S\ref{subsec:edge}. That is, for $m = \nu, \ldots, M-\nu-1$ we can write the matrix-vector product as 
\begin{equation}
    \begin{split}\label{eq:mid_adj}
    \sum_{n=0}^{N-1} \ind{w_{n}, \phi_{j,m}} \alpha_n 
    &= \frac{1}{\sqrt{2^j}}\sum_{n=0}^{N-1}  
    \sum_{l=-\nu+1}^{\nu-1} w_n\( \frac{l+m}{2^j}\)
    \mathcal{W}\phi_{0,-l}\lvert_{[0,1)}\(\floor{\frac{n}{2^j}}\) \alpha_n \\ 
    &= \frac{1}{\sqrt{2^j}}\sum_{l=-\nu+1}^{\nu-1} \sum_{n=0}^{N-1} w_{2^q(l+m)}\(
    \frac{n}{N} \) \mathcal{W}\phi_{0,-l}\lvert_{[0,1)}
    \(\floor{\frac{n}{2^j}}\) \alpha_n 
    \end{split}
\end{equation}
by utilizing Lemma \ref{lem:fjk} and \eqref{eq:wal_shift2}.
Next define the operator $B_{l}\colon \R^{N} \to \R^{M}$ as 
\begin{equation*}
    B_{l}(s) = \begin{cases}  
               \sum_{n=0}^{N-1} w_{2^{q}(l+m)}(n/N) s_{n} & \text{for } m = \nu, \ldots, M-\nu-1\\
               0 & \text{for } m \in \{0,\ldots,\nu-1\}\cup\{M-\nu, \ldots, M-1\}
           \end{cases}
\end{equation*}
for $s \in \R^{N}$, and observe that $B_{l} = H_{l}^*$. Thus, from  
Equation \eqref{eq:mid_adj} we now have  
\[
\left[ \sum_{n=0}^{N_1} \ind{w_{n}, \phi_{j,m}}\alpha_{n}\right]_{m=\nu}^{M-\nu-1} 
= \sum_{l=-\nu+1}^{\nu-1} B_{l}(D_{l}(\alpha)).
\]
Finally, observe that $B_l$ can be computed in $\mathcal{O}(N\log N)$ operations by applying a fast Walsh Hadamard transform of dimension $N$ and selecting the appropriate output from this transform. Since the cost of applying $D_l$ is $\mathcal{O}(N)$, the total cost of computing the output from the middle rows are of order $\mathcal{O}(N\log N)$.

\subsection{The edge operations}
\label{subsec:edge}
We now turn to the edge functions and consider the two boundary extensions given by the periodic and vanishing moments preserving boundary wavelets. This gives us four different edge inner products, one for each edge and boundary extension. 

Consider the forward operation. According to \eqref{eq:sum_key_ind} we can write the sum of the $\nu$ first and $\nu$ last columns as
\begin{equation}
    \label{eq:forw_bd}
    \left[ \sum_{m=0}^{\nu-1} \indi{\phi_{j,m}^{\text{rep}}, w_n}\xi_{m} \right]_{n=0}^{N-1}  
    \quad\text{ and } \quad 
    \left[ \sum_{m=M-\nu}^{M-1} \indi{\phi_{j,m}^{\text{rep}}, w_n}\xi_{m} \right]_{n=0}^{N-1},
\end{equation}
respectively, for $\xi \in \R^{M}$. Likewise for the adjoint operation we can write the $\nu$ first and $\nu$ last rows as 
\begin{equation}
    \label{eq:adjoint_bd}
\left[ \sum_{n=0}^{N-1} \indi{w_n, \phi_{j,m}^{\text{rep}}}\alpha_{n} \right]_{m=0}^{\nu-1}
\quad\text{ and } \quad 
\left[ \sum_{n=0}^{N-1} \indi{w_n, \phi_{j,m}^{\text{rep}}}\alpha_{n} \right]_{m=M-\nu}^{M-1},
\end{equation}
respectively, for $\alpha \in \R^{N}$.

At the edges, we could, -- potentially -- compute the inner products $\indi{\phi_{j,m}^{\text{rep}}, w_{n}}$ a priori and store the result as dense matrices. A challenge with this approach, is that we need to compute and store the inner products for every possible combination of $j$, $m$ and $n$. This is infeasible in general, and would only allows us to do computations for certain dimensions. However, by applying Lemma \ref{lem:fjk} once more, we can
disentangle $j,m,n$ from the integral computation, so that we only need to compute
$\mathcal{W}\phi_{0,-l}^{\text{rep}}\vert_{[0,1]}(s)$ for $s \in \{0,\ldots, 2^{q}-1\}$. In the next proposition we do just this. Note that we use the convention that if $b < a$ and we write $\sum_{l=a}^{b} (\cdots)$, then this should be interpreted as zero. 
\begin{proposition}
    \label{prop:bd_coeff}
    Let $\phi$ be a scaling function, whose wavelet
has $\nu > 1$ vanishing moments. Let $M = 2^j$ and $N = 2^{j+q}$ for positive integers $j \geq \ceil{\log_2 (2\nu)}$ and $q>0$.  Let $n \in \{0,\ldots, N-1\}$.  Then for $m=0,\ldots, \nu-1$, 
    \begin{align*}
    \ind{\phi_{j,m}^{\text{per}}, w_n} =&  \sum_{l=-\nu+1}^{-m -1} 2^{-j/2}w_n\(\frac{2^j+m+l}{2^j}\)
      \mathcal{W}\phi_{0,-l}\lvert_{[0,1)} \(\floor{\frac{n}{2^j}}\) \\
    &+\sum_{l=-m}^{\nu -1} 2^{-j/2}w_n\(\frac{m+l}{2^j}\)
      \mathcal{W}\phi_{0,-l}\lvert_{[0,1)} \(\floor{\frac{n}{2^j}}\) 
    \end{align*}
and
    \begin{align*}
         \ind{\phi_{j,m}^{\text{bd}}, w_n} &=
    \sum_{l=0}^{\nu-1+m} 
   2^{-j/2} w_n\(\frac{l}{2^{j}}\)  \mathcal{W}{\phi}_{m}^{\text{left}}(\cdot + l)\vert_{[0,1)}\(\floor{\frac{n}{2^j}}\).
    \end{align*}
    Furthermore, for $m = 2^{j}-\nu, \ldots, 2^{j}-1$,  
{ \allowdisplaybreaks
    \begin{align*}
    \ind{\phi_{j,m}^{\text{per}}, w_n} =&  
     \sum_{l=-\nu+1}^{2^j-m-1} 2^{-j/2} w_n \( \frac{l+m}{2^j} \) 
    \mathcal{W}\phi_{0,-l}\lvert_{[0,1)}\(\floor{\frac{n}{2^j}}\) \\ 
    &+\sum_{l=2^{j}-m}^{\nu-1} 2^{-j/2} w_n\(\frac{l+m-2^j}{2^j}\)
    \mathcal{W}\phi_{0,-l}\lvert_{[0,1)} \(\floor{\frac{n}{2^j}}\),
    \end{align*}
}
and 
    \begin{align*}
    \ind{\phi_{j,m}^{\text{bd}}, w_n} &= \sum_{l=m-2^j - \nu+1}^{-1} 2^{j/2}w_{n} 
\left(\frac{l+2^{j}}{2^j}\right) \mathcal{W}{\phi}_{2^{j} -1 -m}^{\text{right}}(\cdot + l)\vert_{[0,1)}\(\floor{\frac{n}{2^j}}\).
     \end{align*}
\end{proposition}

\begin{proof}
For an interval $I \subset \R$, let $\chi_{I}$ denote the characteristic funciton on $I$. The result follows by using Lemma \ref{lem:fjk} on all the considered inner products. For all functions intersecting the left edge this is trivial, the result follows by recalling that $\Supp(\phi) = [-\nu+1, \nu]$ and $\Supp (\phi_{m}^{\text{left}}) = [0,\nu+m]$. The same can be said, about the functions  $\phi_{j,\nu-1}^{\text{per}} = \phi_{j,\nu-1}$, and $\phi_{j, 2^j-\nu}^{\text{per}} = \phi_{j,2^j-\nu}$, since these are interior functions. Applying Lemma \ref{lem:fjk} to the right edges require slighly more care, since it is assumed that the function under consideration is supported on an interval $[a,b]$, with $a,b \in \mathbb{N}$, and $b > 0$. On the right edges this can be achived by using the change of variable $y=2^jx - (2^j-1)$. We do not write out the details for all the considered functions, but demonstrate the idea on $\phi_{j,2^{j}+m}\vert_{[0,1]}$, for $m=0,\ldots, \nu-2$ (used in $\phi_{j,m}^{\text{per}}$). 
We start by noticing that $\chi_{[0,1]}(x)= \chi_{[-(2^j-1), 1]}(2^j x - (2^j-1))$. This means that
\[
    \phi_{j,2^j+m}\vert_{[0,1]}(x) = 2^{j/2} \phi(2^jx -(2^j-1) - (m+1)) \chi_{[-(2^j-1), 1]}(2^j x - (2^j-1)),
\]
where the function $\phi(\cdot-(m+1))\chi_{[-(2^j-1),1]}$ has support $[-\nu+1+(m+1), 1]$. Applying Lemma \ref{lem:fjk}, and using that
$
\mathcal{W}\phi_{0,-l}(\cdot -(m+1))\vert_{[0,1)}(s) 
= \mathcal{W}\phi_{0,-l+m+1}\vert_{[0,1)}(s)
$
gives the result.
\end{proof}

Given the inner products $\ind{\phi_{j,m}^{\text{rep}}, w_n}$ and $\ind{w_n, \phi_{j,m}^{\text{rep}}}$, the computational cost of \eqref{eq:forw_bd} and \eqref{eq:adjoint_bd}, is $\mathcal{O}(N)$. Furthermore, to compute these inner products we may use Proposition \ref{prop:bd_coeff} for each $n \in \{0,\ldots, N-1\}$. 

However, this can can be challenging since, evaluating the above sums, require the computation of $w_n(2^{-j}(m+l))$ for many different choices of $m, l$ and $n$, and -- to the best of the author's knowledge -- there are no software packages implementing the pointwise evaluation of Walsh functions.  Moreover, a naive implementation in C++ using Defenition \ref{def:walsh_func} is rather slow. To speed up this part of the code we use the relation between Walsh functions and Hadamard matrices, and use the FWHT algorithm to evaluate $w_n(2^{-j}(m+l))$ for all the relevant values of $m,l$ and $n$, simultaniusly. However, this raises the computational cost of the edge computations to $\mathcal{O}(N \log N)$.

\begin{algorithm}
\begin{algorithmic}[1]
\Procedure{The forward operation}{}
\State $\textbf{Input: } j,q \in \N. ~ M = 2^{j}, ~ N=2^{j+q} \text{ and } \xi \in \R^{M}$.
\State $\textbf{Output: } \alpha = P_{N}UP_{M}\xi \text{ where } P_{N}UP_M \text{ is given by \eqref{eq:vec_mat_prod}.} $ 
\State Compute vectors $\beta^{\text{left}} \in \R^{N}$ and $\beta^{\text{right}} \in\R^{N}$
from \eqref{eq:forw_bd}, using Prop.\ \ref{prop:bd_coeff}.
\State Compute the vector $\beta^{\text{mid}} = \sum_{l=-\nu+1}^{\nu-1} D_{l}H_{l}(\xi)$.
\State \textbf{return} $\alpha = \beta^{\text{left}} + \beta^{\text{mid}} + \beta^{\text{right}}$. 
\EndProcedure
\end{algorithmic}
\begin{algorithmic}[1]
\Procedure{The adjoint operation}{}
\State $\textbf{Input: } j,q \in \N. ~ M = 2^{j}, ~ N=2^{j+q} \text{ and } \alpha \in \R^{N}$.
\State $\textbf{Output: } \xi = P_{M}U^*P_{N}\alpha \text{ where } P_{N}UP_M \text{ is given by \eqref{eq:vec_mat_prod}.} $ 
\State Compute $\xi_{0}, \ldots, \xi_{\nu-1}$ and $\xi_{M-\nu}, \ldots \xi_{M-1}$ from \eqref{eq:adjoint_bd}, using Prop.\ \ref{prop:bd_coeff}.
\State Compute $\xi_{\nu}', \ldots, \xi_{M-\nu-1}'$ extracting elements from 
       $\xi' = \sum_{l=-\nu+1}^{\nu-1} B_{l}(D_{l}(\alpha))$.
\State \textbf{return} $\xi = [\xi_{0}, \ldots, \xi_{\nu-1}, \xi_{\nu}',
       \ldots, \xi_{M-\nu-1}', \xi_{M-\nu}, \ldots, \xi_{M-1}]$
\EndProcedure
\end{algorithmic}
\caption{\label{alg:one_dim}The one-dimensional forward and adjoint operation.}
\end{algorithm}

\section{Extension to two dimensions}
\label{sec:two_dim_alg}
We restrict our attention to $d=2$ dimensions, since it applies to any kind of imaging application. It is certainly possible to extend the algorithm to any $d$-dimensional tensor product space, though, it practical relevance seems limited.  We, therefore, let $\mathcal{H} = L^2([0,1]^2)$ and consider samples from the tensor product basis $\{w_{n_{1}}\otimes w_{n_2} : (n_1,n_2) \in \mathbb{N}^2\}$.
As for the one dimensional algorithm, we consider the case where the sampling and reconstruction spaces are dyadid cubes. That is, for $N=2^{j+q}$ and $M=2^j$ we let the sampling space $ \mathcal{S}_{N^2} = \{w_{n_1}\otimes w_{n_2}: 0\leq n_1, n_2 < N\} $
and the reconstion space 
$
\mathcal{R}_{M^2} = \left\{\phi_{j, m_1}\otimes \phi_{j,m_2}: 0 \leq m_1,m_2 < M \right\}$.

For a tensor $\xi \in \R^{M\times M}$, we can split the change-of-basis computation as 
\begin{equation}\label{eq:tensor}
\begin{split}
    \alpha_{n_1,n_2} &= \sum_{m_1=0}^{M-1}\sum_{m_2 =0}^{M-1} \xi_{m_1,m_2} \ind{\phi_{j,m_1}\otimes \phi_{j,m_2}, w_{n_1}\otimes w_{n_2}}  \\ 
&= \sum_{m_1=0}^{M-1} \ind{\phi_{j, m_1}, w_{n_1}} \sum_{m_2=0}^{M-1} \xi_{m_1,m_2} \ind{\phi_{j, m_2}, w_{n_2}}
\end{split}
\end{equation}
for each $n_1,n_2 \in \{0,\ldots, N-1\}$, so that it is a double sum of one-dimensional inner products. Thus, letting 
$G \in \mathbb{R}^{N\times M}$ denote the forward operator  we derived for 
the one dimensional case, and letting
$\eta \in \C^{M \times
N}$ have components \[\eta_{m_1, n_2} = \sum_{m_2=0}^{M-1} \xi_{m_1, m_2}
\ind{\phi_{j,m_2}, w_{n_2}} = G\left( \left[\xi_{m_1,m_2}\right]_{m_2=0}^{M-1} \right)\]
we see that the above computation simplifies to 
\[
\alpha_{n_1,n_2}  = \left(G \left(\left[ \eta_{m_1, n_2} \right]_{m_1=0}^{M-1}\right)\right)_{n_1}, \quad\text{for }n_1,n_2 \in \{0,\ldots,N-1\},
\]
or simply
$\alpha = G \xi G^*$,  for $\xi \in \R^{M\times M}$.

Now, since $G$ can be evaluated in $\mathcal{O}(N\log N)$ operations, we can compute $\alpha = G\xi G^*$ in $\mathcal{O} (2M N \log N)$ operations. Furthermore, since $N=2^{2q}M$, and $q\in \{1,2\}$ is a resonable choice, this is reduces to $\mathcal{O}(N^2 \log N^2)$, where $N^2$ is the dimension of the sampling space. 

By considering \eqref{eq:tensor}, it should be clear that we can do the same type of splitting also for the adjoint operation. We do not do the full derivation, but notice that as an intermediate step one would need to compute
\begin{equation*}
    \beta_{n_1, m_2} = \sum_{n_2=0}^{N-1} \alpha_{n_1, n_2} \ind{w_{n_2}, \phi_{j, m_2}} 
    = G^{*} \left(\left[ \alpha_{n_1, n_2} \right]_{n_2=0}^{N-1}\right) 
\end{equation*}
for $0\leq m_2 < M$. Applying the same transform in the row direction, leads to a transform which can be computed in $\mathcal{O}(N^2\log N^2)$ operations. The complete algorithm is summarized in Algorithm \ref{alg:two_dim}.  

\begin{algorithm}
\begin{algorithmic}[1]
\State Let $G \in \R^{N \times M}$ be the one dimensional truncated change-of-basis matrix \eqref{eq:vec_mat_prod}.
\Procedure{The forward operation}{}
\State $\textbf{Input: } j,q \in \N. ~ M = 2^{j}, ~ N=2^{j+q} \text{ and } \xi \in \R^{M \times M}$.
\State Compute $\alpha = G \xi G^*$.   
\State \textbf{return } $\alpha$.
\EndProcedure
\end{algorithmic}
\begin{algorithmic}[1]
\Procedure{The adjoint operation}{}
\State $\textbf{Input: } j,q \in \N. ~ M = 2^{j}, ~ N=2^{j+q} \text{ and } \alpha \in \R^{N\times N}$.
\State Compute $\xi = G^* \alpha G$.   
\State \textbf{return } $\xi$.
\EndProcedure
\end{algorithmic}
\caption{\label{alg:two_dim}The two dimensional forward and adjoint operation.}
\end{algorithm}

\section{Numerical examples}
\label{s:num_exp}
We conclude by demonstrating how the proposed fast transform be used by the three reconstruction methods presented in the introduction. The code for producing these figures can be found on the Github page. Throughout the section we let $\chi_{I}$ denote a step function on the set $I \subset \R^{d}$, $d=1,2$. 

\paragraph{Example 1} 
We compare the four reconstruction methods (1) truncated Walsh series, (2) generalised sampling, (3) the PBDW-method and (4) compressive sensing by acquiring $N=32$ Walsh samples from the function 
$
f(t) = \cos(2\pi t) \chi_{[0,1/2]}(t) + \tfrac{1}{2} t \sin(6\pi t) \chi_{(1/2,1]}(t).  
$
The resulting reconstructions can be seen in Figure \ref{fig:comp_methods}. The first three methods are linear reconstruction methods, and we acquire Walsh samples using the $N$ first Walsh functions. Compressive sensing (CS), on the other hand, is an example of a non-linear reconstruction method. For compressive sensing we, therefore, subsample $32$ samples from the first $256$ Walsh samples, using a variable density sampling scheme. For the GS and PBDW reconstructions, we use the first $M=16$ basis functions in the DB4 wavelet basis for reconstruct, whereas for the compressive sensing reconstruction use the $128$ first functions in this basis. In all cases, we use vanishing moments preserving boundary wavelets to minimise the artefacts at the boundaries. It is clear from the figure that the compressive sensing reconstruction causes the least artefacts and best reconstruction, despite some wiggles around the discontinuity at $t=1/2$. The truncated Walsh series cause the very characteristic blocky artefacts, whereas the generalised sampling method produces a smooth approximation to $f$. However, since we only use 16 wavelet functions, we obtain a very poor approximation around the discontinuity with generalised sampling. The PBDW method approximates $f$ better in the smooth areas on the right, using a large number of Walsh functions. Still, it produces severe artefacts around the discontinuity and at the top of the sine curve.  

\begin{figure}[tb]
    \centering
    \begin{\textsizefig}
    \setlength{\tabcolsep}{2pt}
    \begin{tabular}{@{}>{\centering}m{0.32\textwidth}>{\centering}m{0.32\textwidth}>{\centering\arraybackslash}m{0.32\textwidth}@{}}
    $f(t)$ & Compressive sensing (CS) & \\ 
     \includegraphics[width=\linewidth]{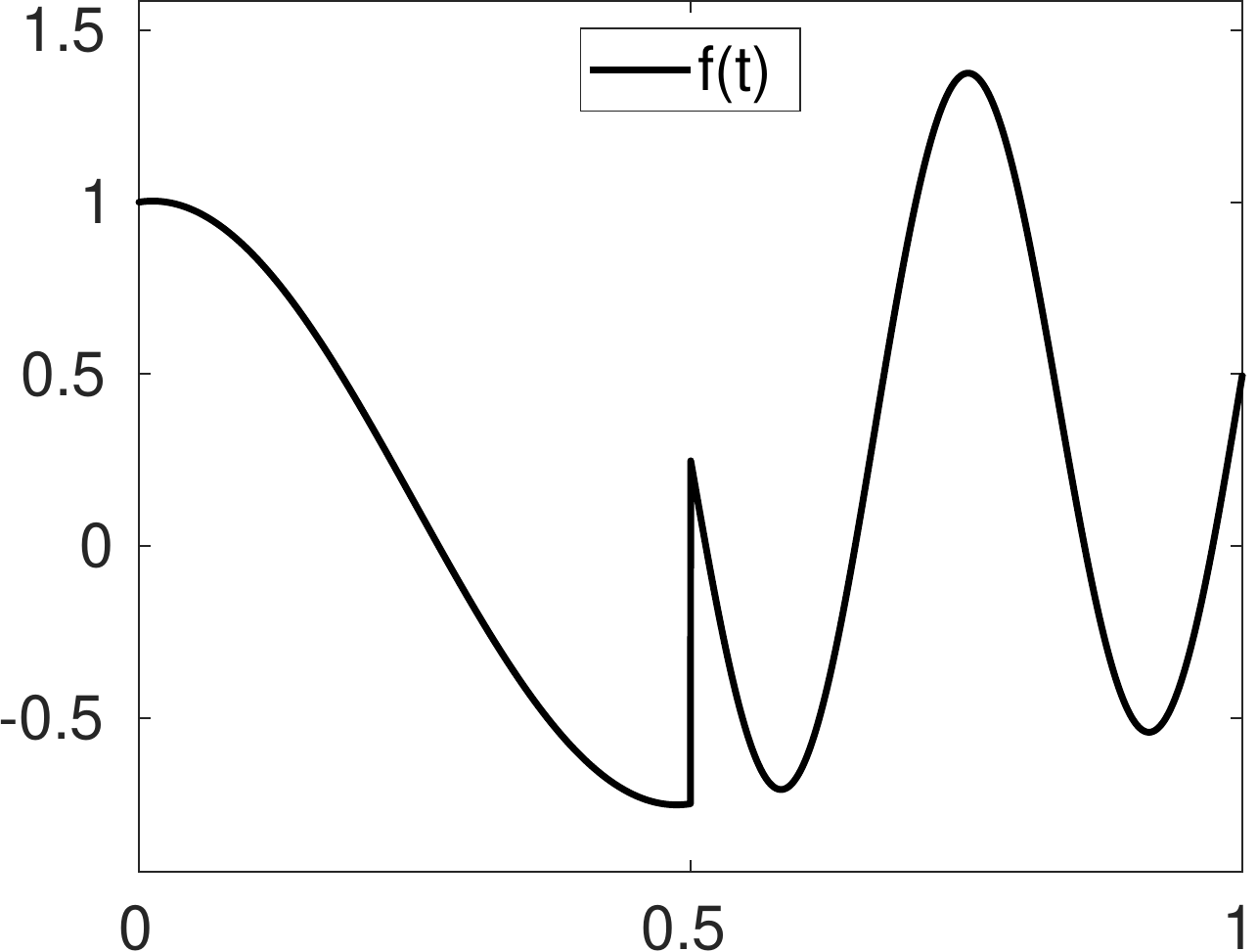} 
    &\includegraphics[width=\linewidth]{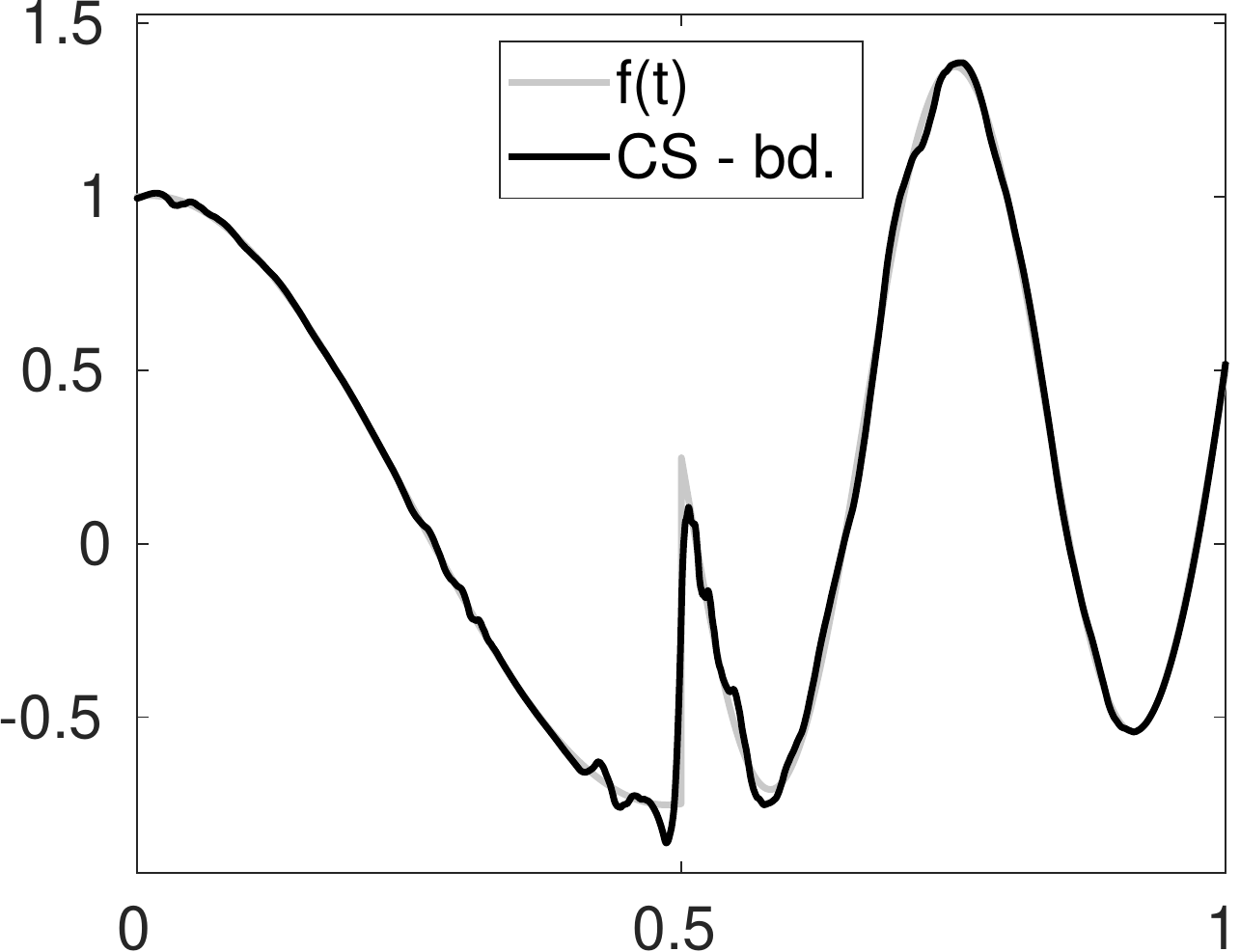}
    &
    \begin{tabular}[b]{@{}>{\centering\arraybackslash}m{0.32\textwidth}@{}}
         Sampling pattern linear methods \\ 
         \includegraphics[width=\linewidth]{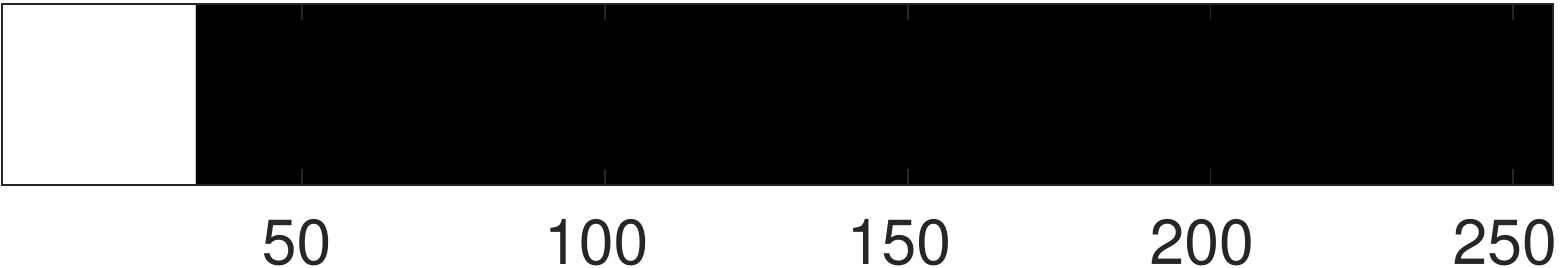} 
         Sampling pattern compressive sensing \\
        \includegraphics[width=\linewidth]{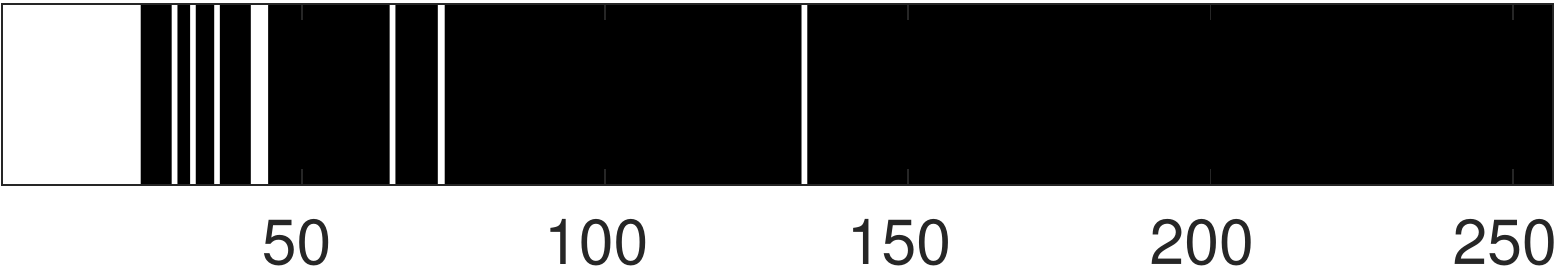} \\
     \end{tabular} \\
    Truncated Walsh (TW) &  Generalised sampling (GS) & PBDW-method\\ 
     \includegraphics[width=\linewidth]{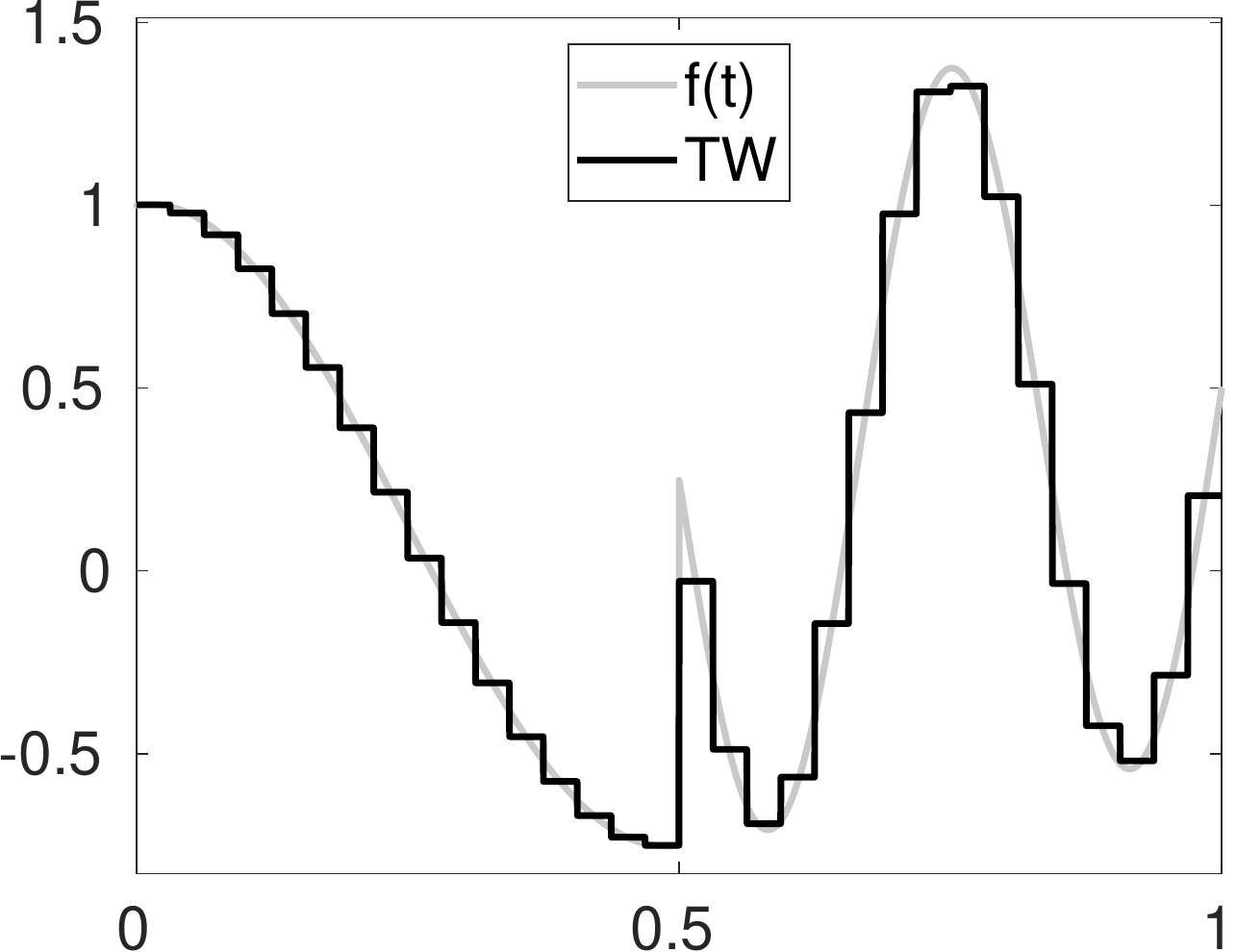} 
    &\includegraphics[width=\linewidth]{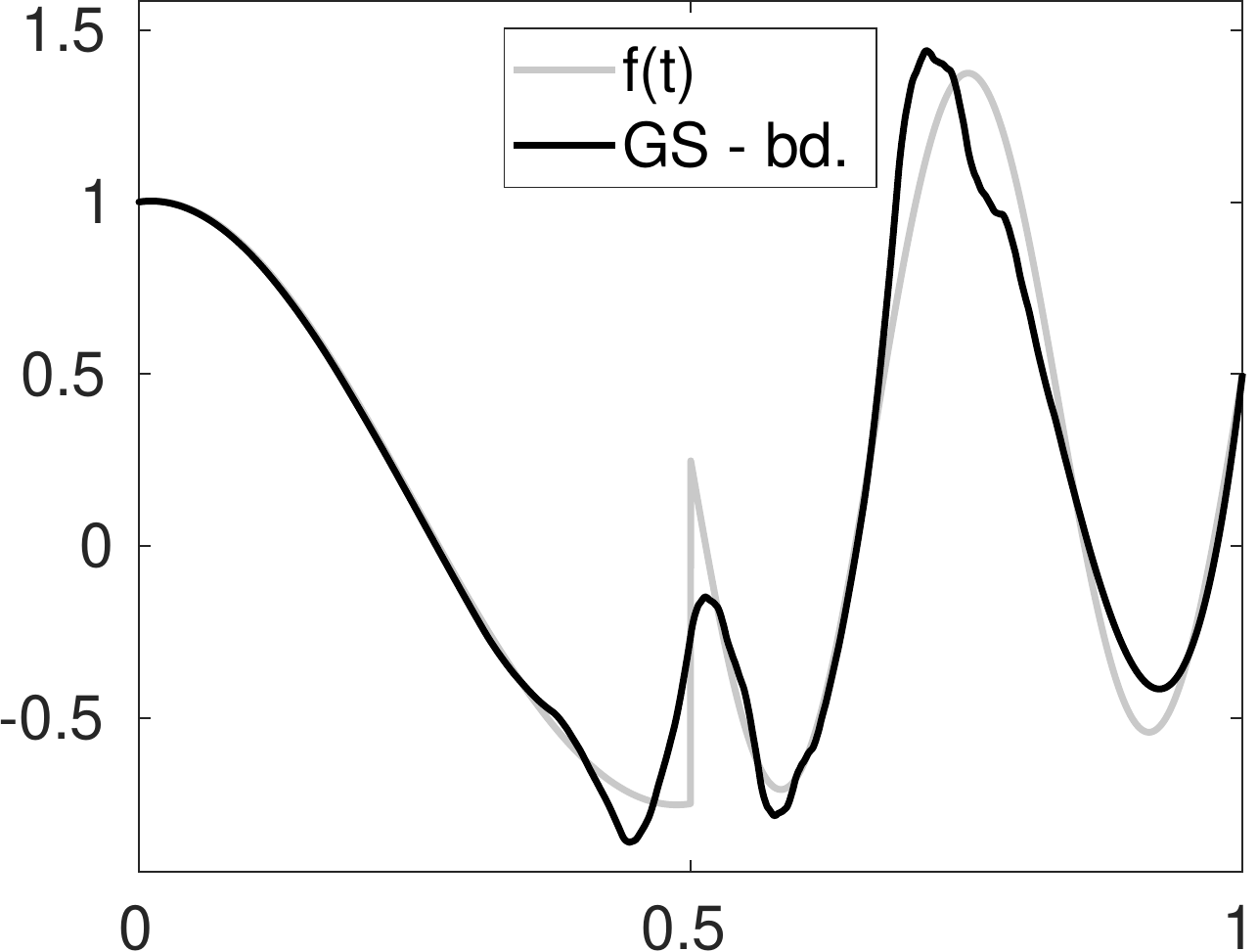} 
    &\includegraphics[width=\linewidth]{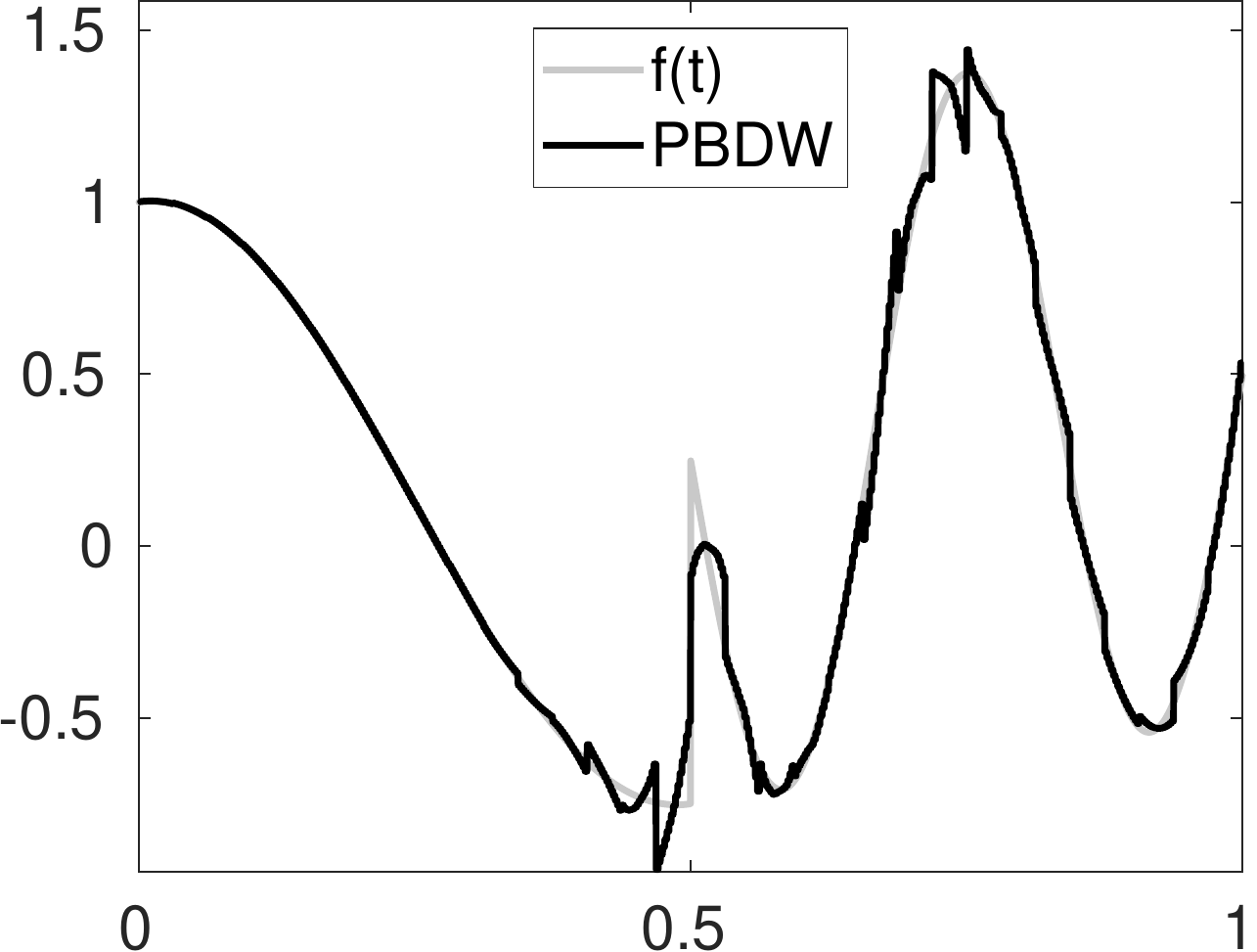} \\
    \end{tabular}
    \end{\textsizefig}
    \vspace{-3mm}
    \caption{\label{fig:comp_methods}
(\textbf{Comparison of reconstruction methods}).
We reconstruct the function $f = \cos(2\pi t) \chi_{[0,1/2]}(t) + \allowbreak \tfrac{1}{2} t \sin(6\pi t) \chi_{(1/2,1]}(t) $ from 32 Walsh samples using the methods: truncated Walsh series, generalised sampling, PBDW-method and compressive sensing. The first three of these methods are linear reconstruction methods, and the samples we acquire are shown in the upper right corner. Note that for the CS reconstruction, we use subsampled measurements (see upper right corner). %The sampling pattern is shown in the upper right corner.
}
%\vspace{-4mm}
\end{figure}

\paragraph{Example 2}
We explore how choosing the reconstruction space in relation to the function one would like to recover can improve the reconstruction quality. The reconstructions can be seen in Figure \ref{fig:GS_smoothness}. In this example  
    $f(t_1,t_2) = \allowbreak \cos\left(\tfrac{3}{2}\pi t_1\right) \sin \left( 3\pi t_2 \right) $
and we acquire $f$'s $32 \times 32 $ first Walsh samples. Using these samples, we compute a truncated Walsh series approximation to $f$, along with generalised sampling reconstructions with different wavelet smoothness. Note that the smoothness of the wavelet basis increases with $\nu$. Since $f$ is smooth we expect that the reconstruction improves with increasing values of $\nu$. In Figure \ref{fig:GS_smoothness}, we see this effect, as the reconstruction error decreases with increasing values of $\nu$.  

\begin{figure}[tb]
    
    \begin{tabular}{@{}l@{}}
    \begin{\textsizefig}
    \begin{tabular}{@{}>{\centering}m{0.23\textwidth}>{\centering}m{0.23\textwidth}>{\centering}m{0.23\textwidth}>{\centering\arraybackslash}m{0.23\textwidth}@{}}
      Truncated Walsh 
    & GS with DB2 
    & GS with DB4 
    & GS with DB6 \\
     \includegraphics[width=\linewidth]{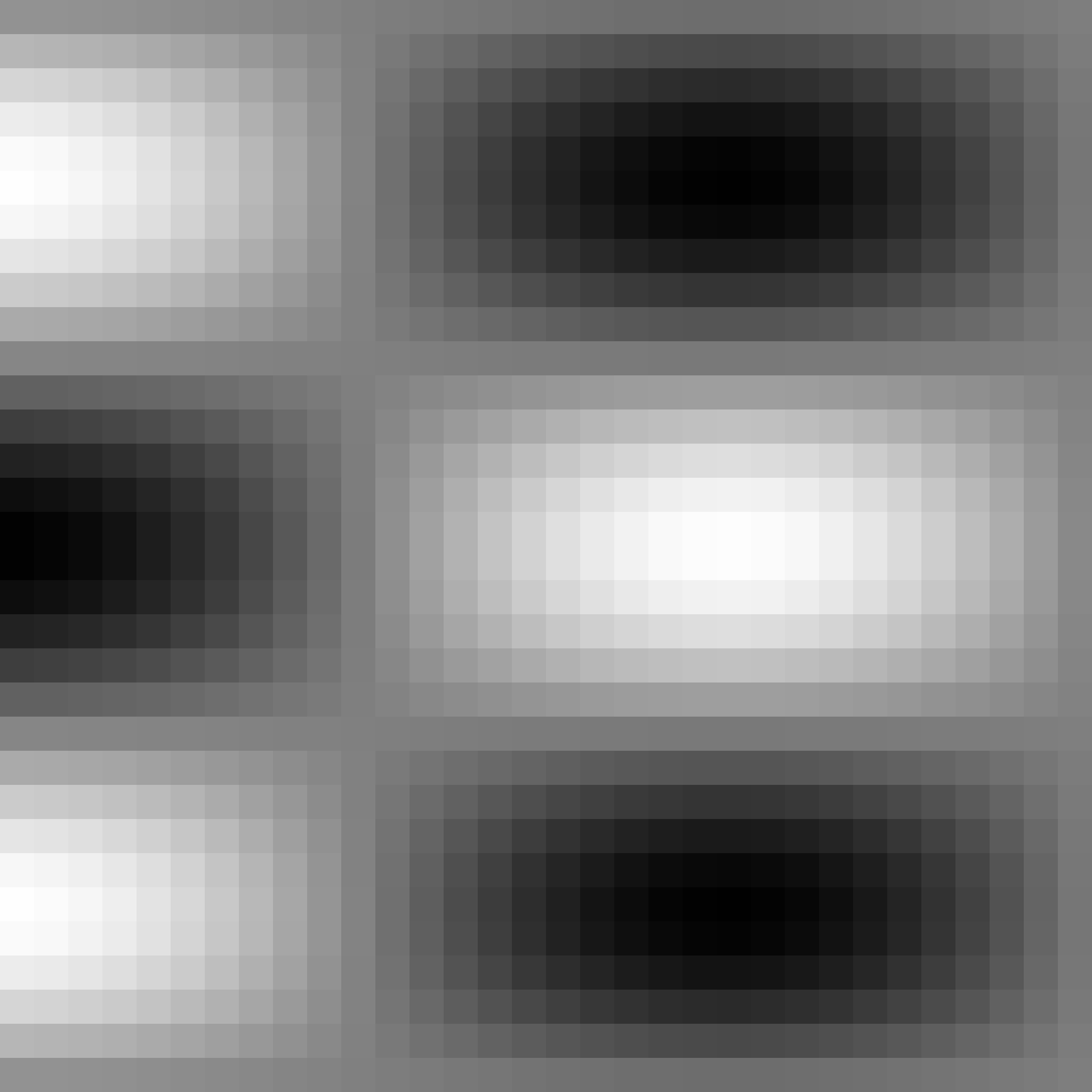}
    &\includegraphics[width=\linewidth]{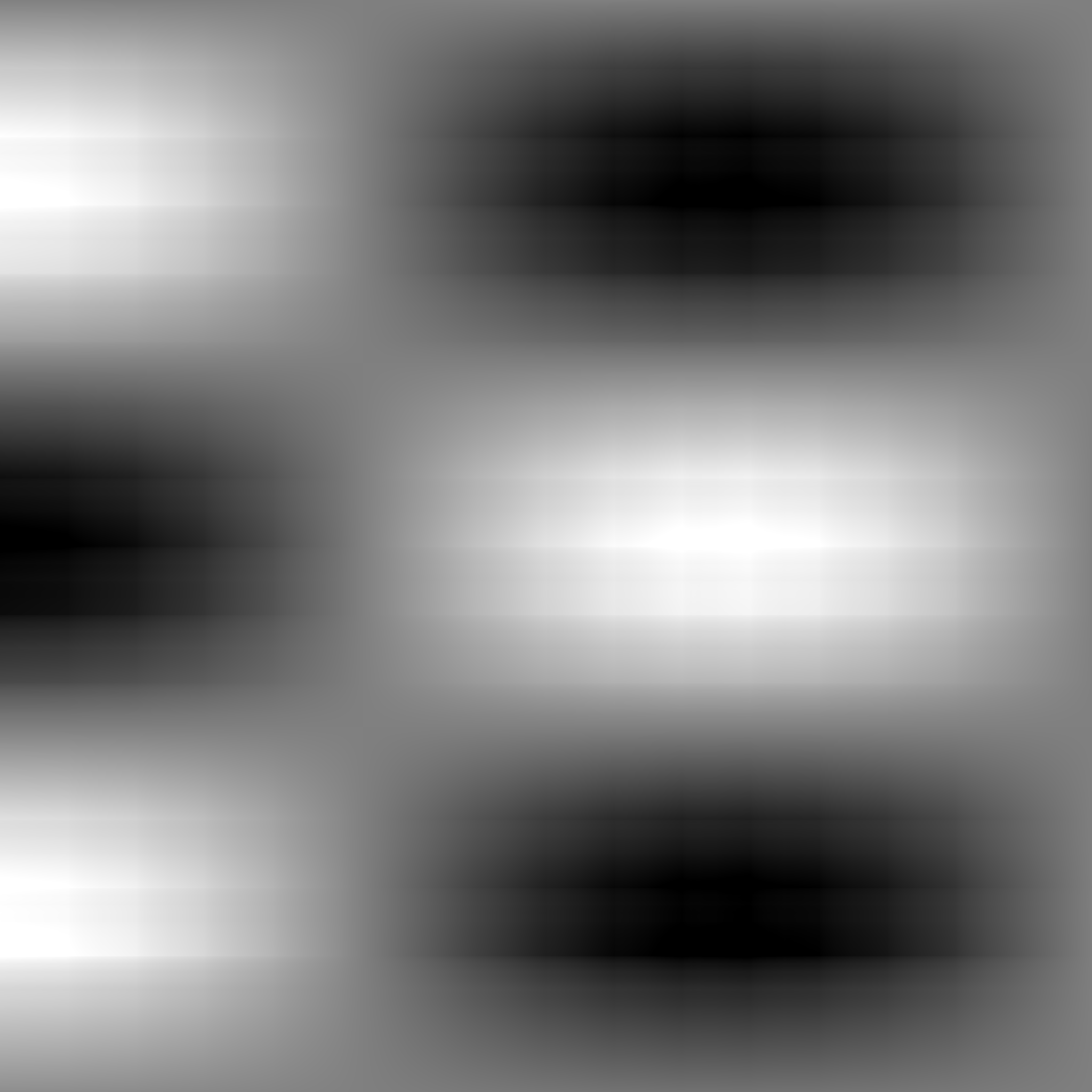} 
    &\includegraphics[width=\linewidth]{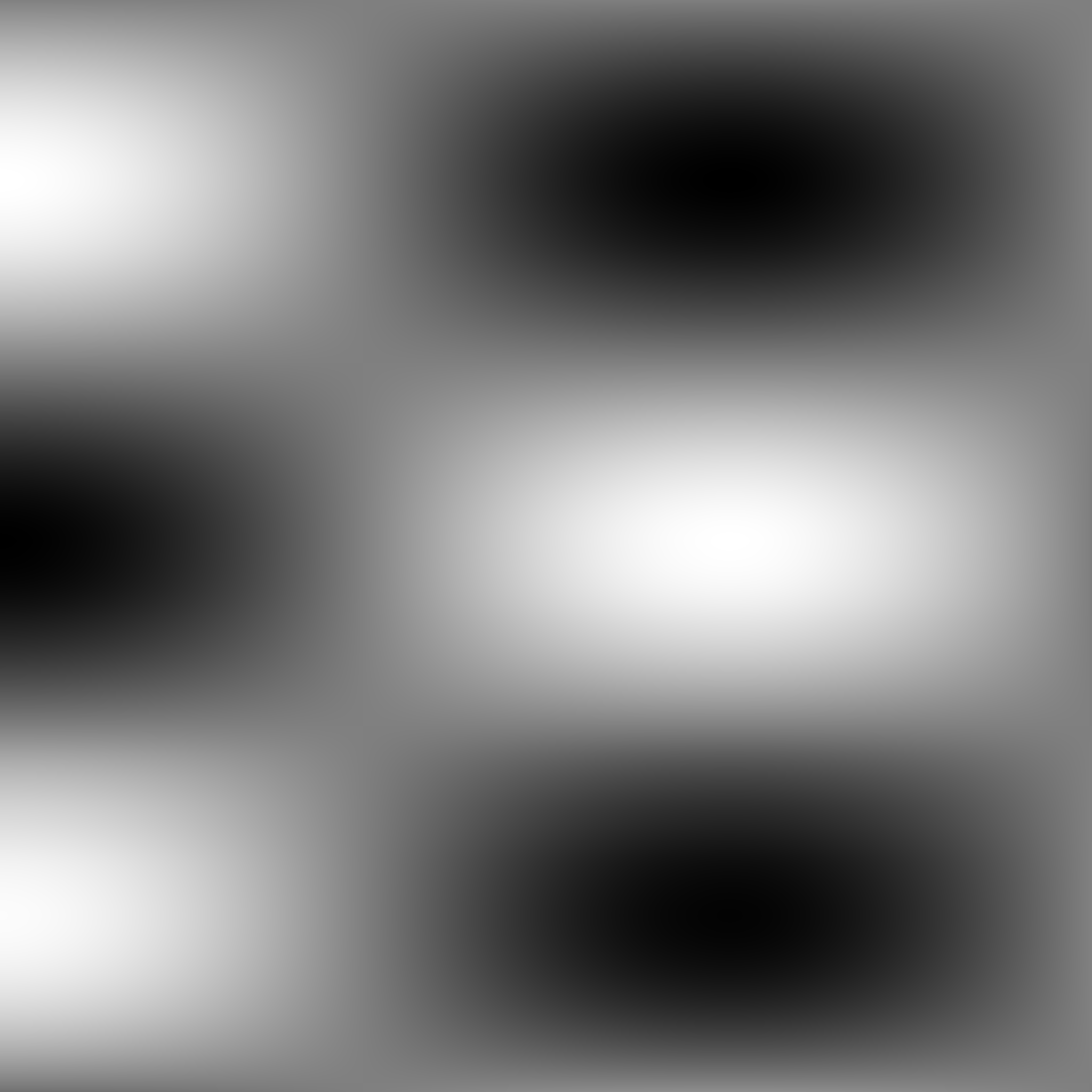} 
    &\includegraphics[width=\linewidth]{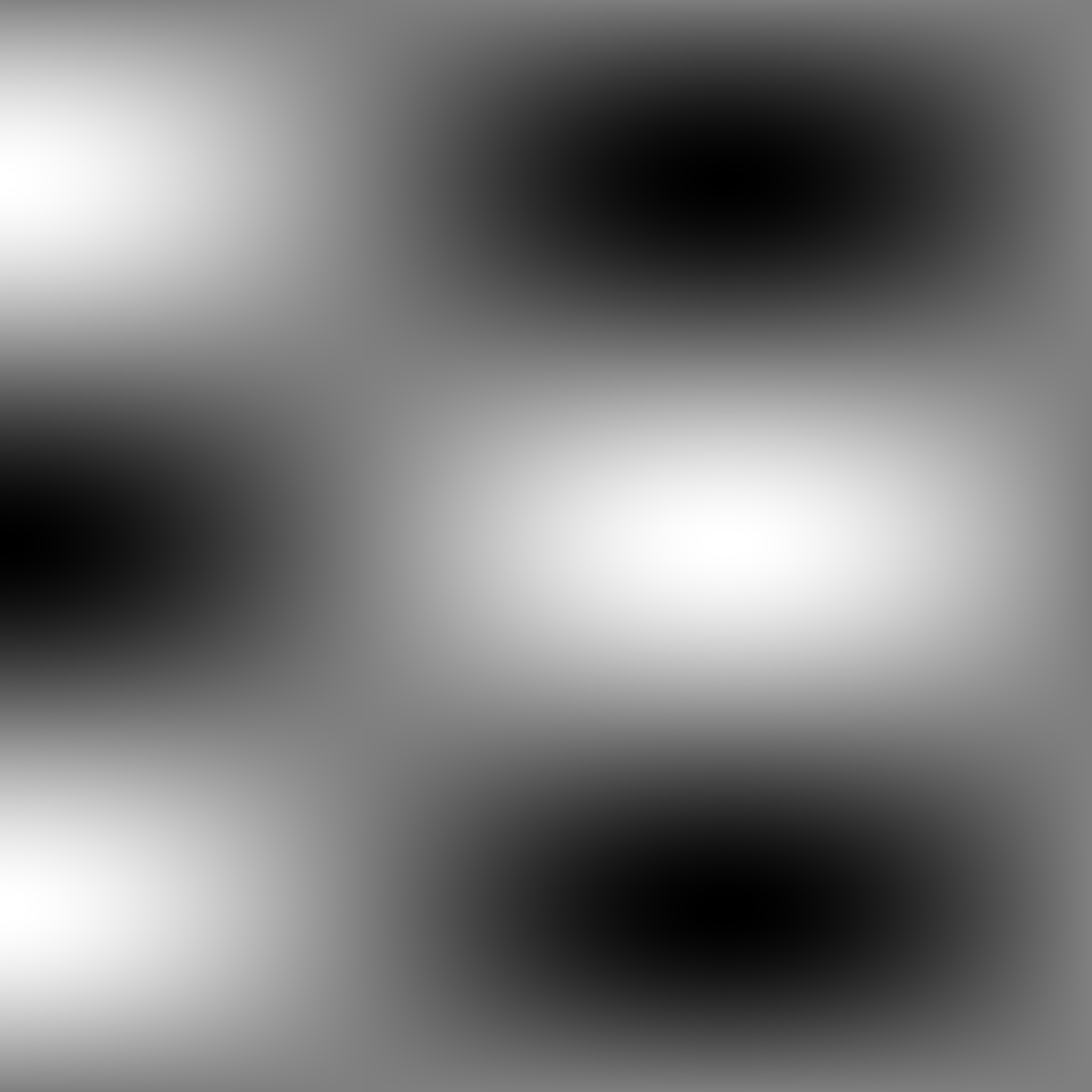}\\
      Error map, 
    & Error map, 
    & Error map, 
    & Error map, \\
      relative error $0.0950$
    & relative error $0.0518$
    & relative error $0.0290$
    & relative error $0.0215$ \\
     \includegraphics[width=\linewidth]{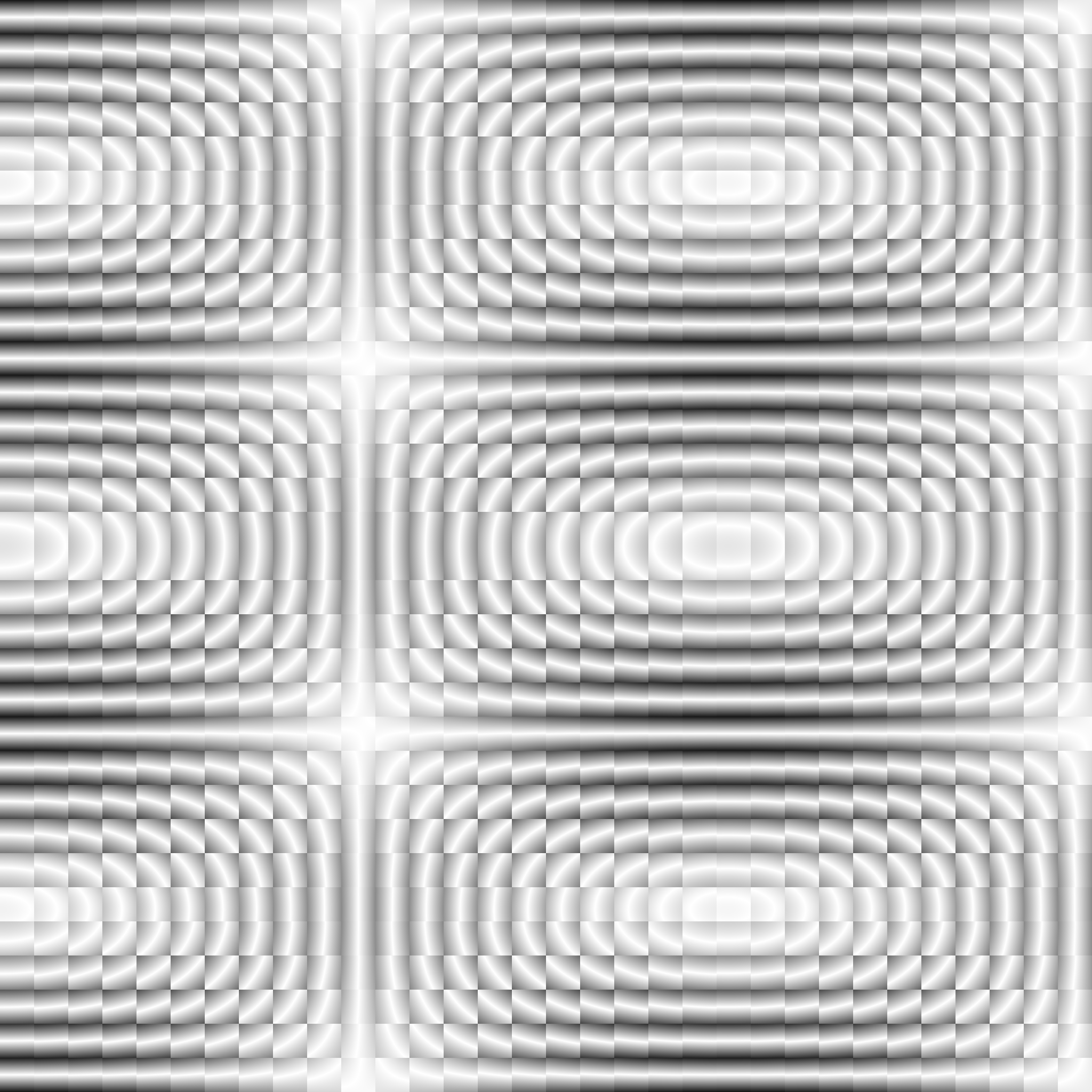}
    &\includegraphics[width=\linewidth]{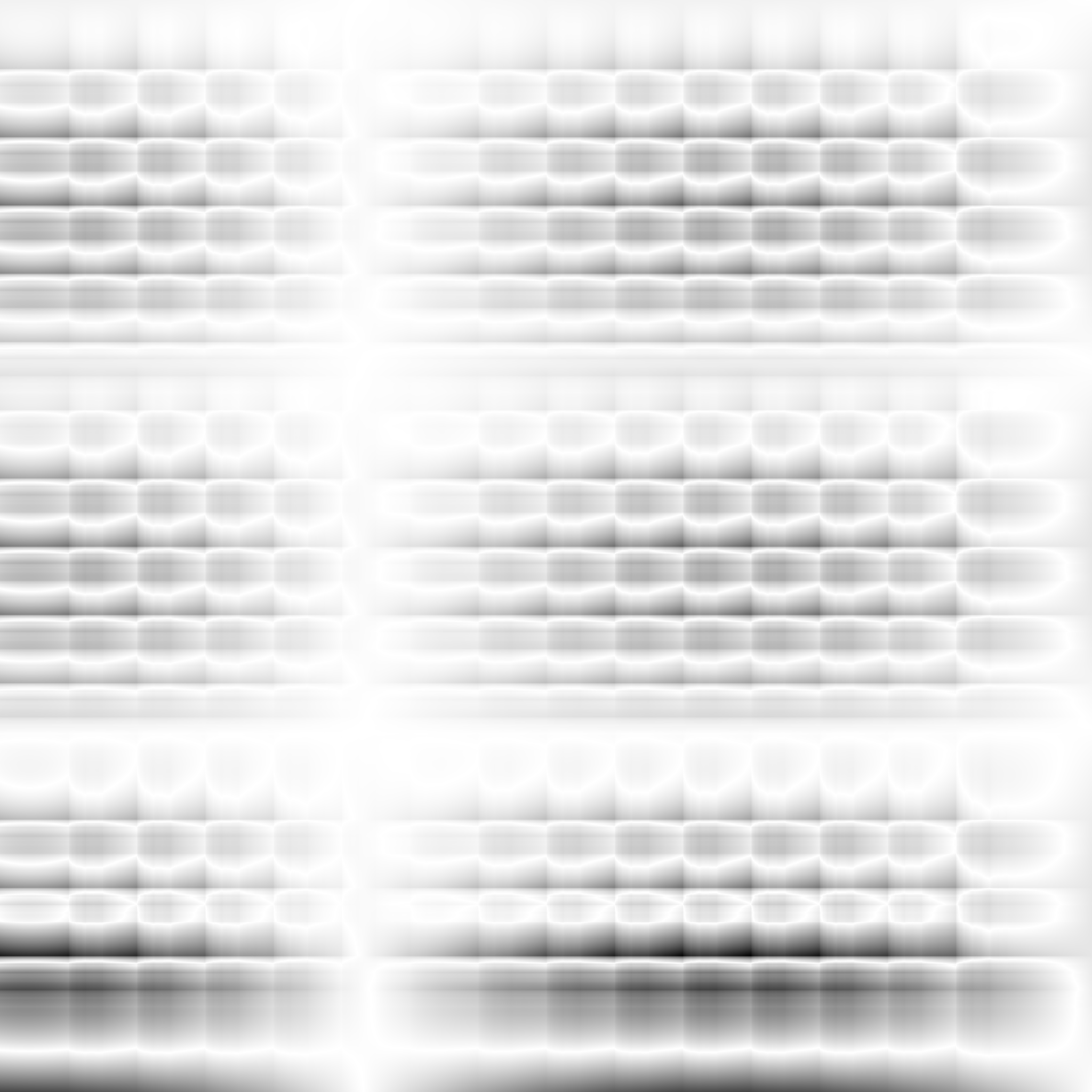} 
    &\includegraphics[width=\linewidth]{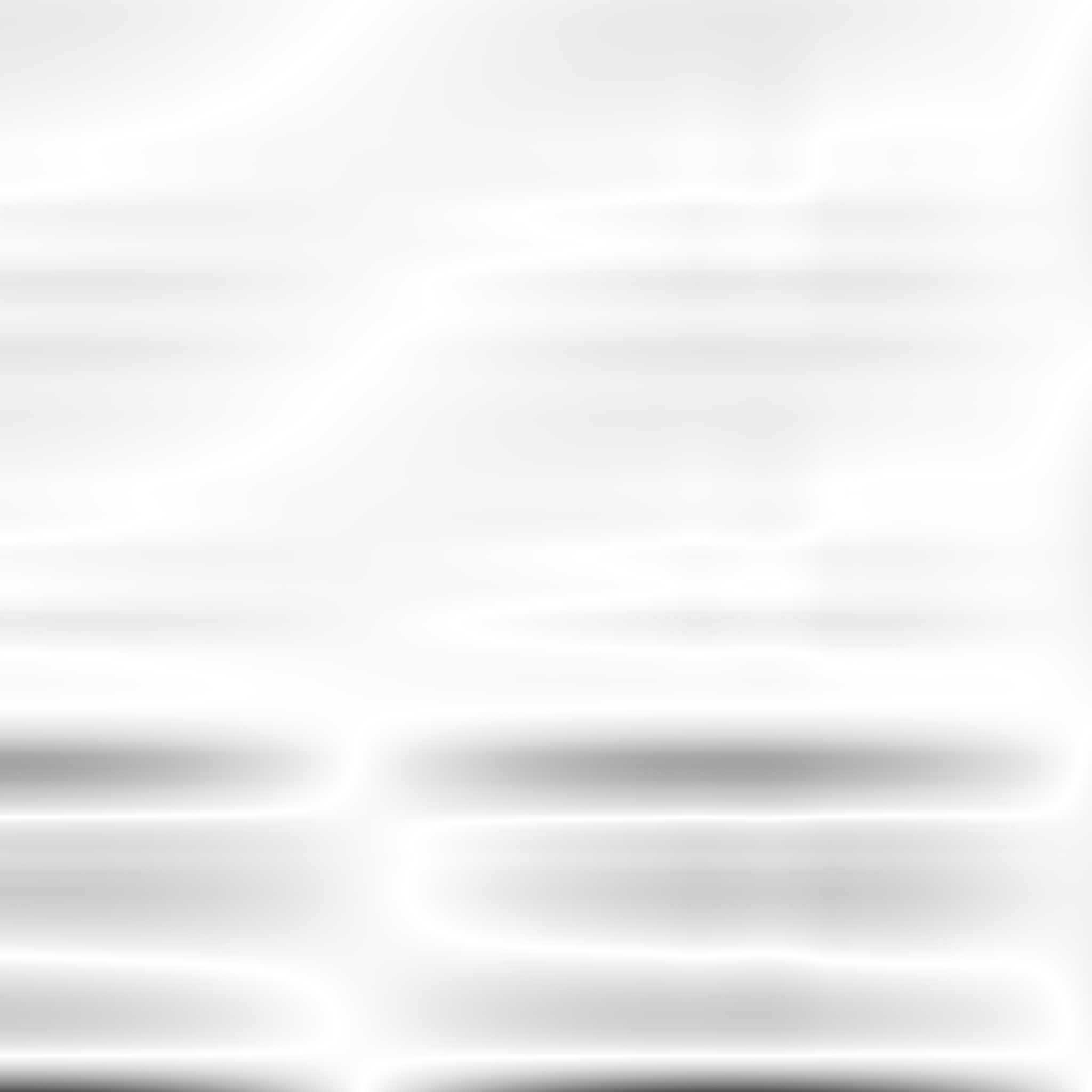} 
    &\includegraphics[width=\linewidth]{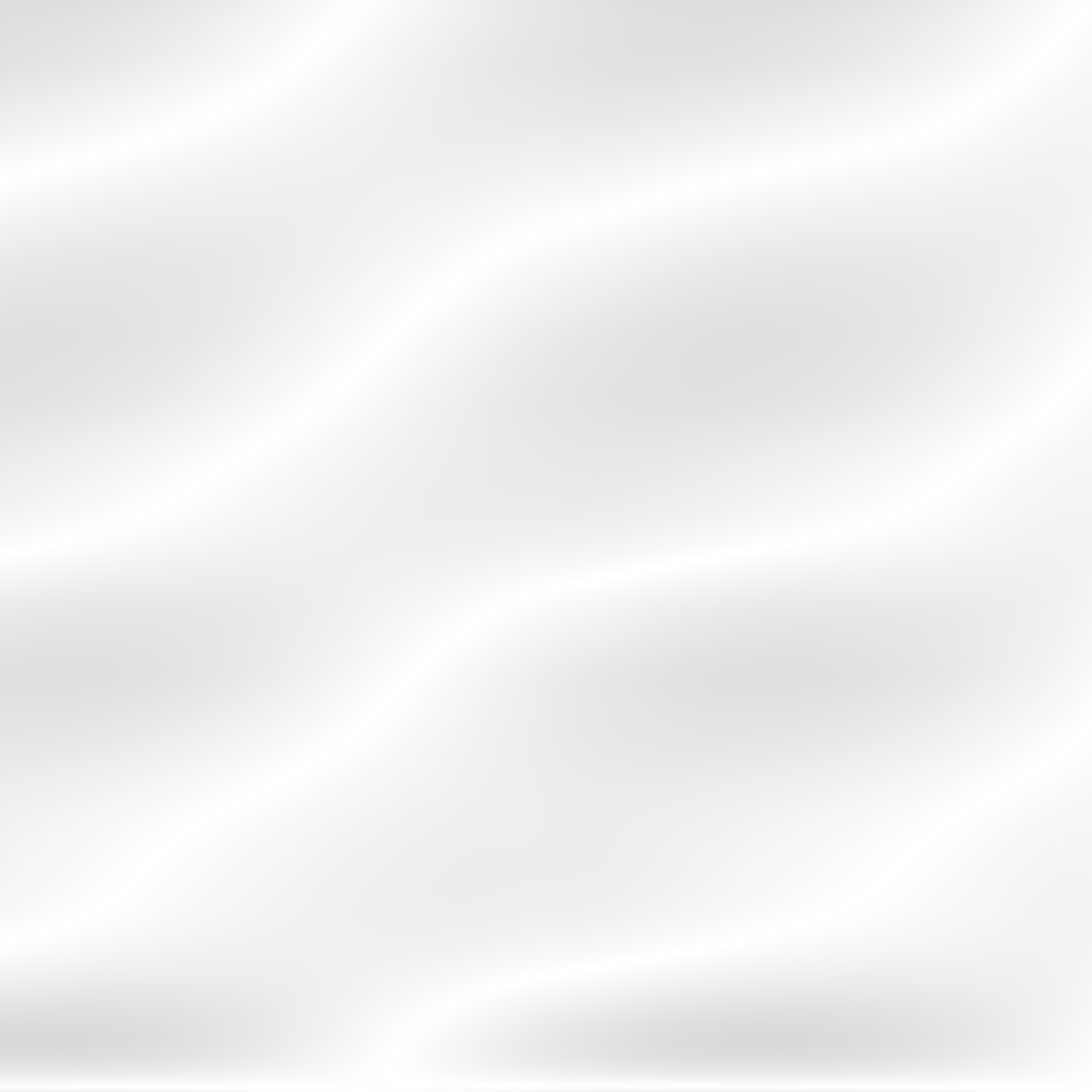}\\
    \end{tabular}
    \end{\textsizefig}
    \\
    \begin{tabular}{@{}m{0.69\textwidth}m{4pt}>{\centering\arraybackslash}m{0.23\textwidth}@{}}
    
   & & \begin{\textsizefig} Original function \end{\textsizefig}   \\
    \caption{\label{fig:GS_smoothness}
(\textbf{Reconstruction error decreases with increasing $\boldsymbol{\nu}$}).
We consider the function $f(t_1,t_2) = \cos\left(\tfrac{3}{2}\pi t_1\right) \sin \left( 3\pi t_2 \right)$ and approximate $f$ from its $32\times 32$ first Walsh samples using a truncated Walsh series (left column), and generalised sampling with different wavelets (column 2-4, row 1-2). In the top row we show the reconstructed functions and in the second row the show the absolute difference $|f-\tilde{f}|$ between $f$ and the computed approximations $\tilde{f}$. The relative error is computed as $\|f-\tilde{f}\|_{\ell^2}/\|f\|_{\ell^2}$, by evaluating $f$ and $\tilde{f}$ in a large number of points. To the right we show the function $f$. Notice how the reconstruction error decreases with increasing values of $\nu$.    
}
  & &
     \includegraphics[width=\linewidth]{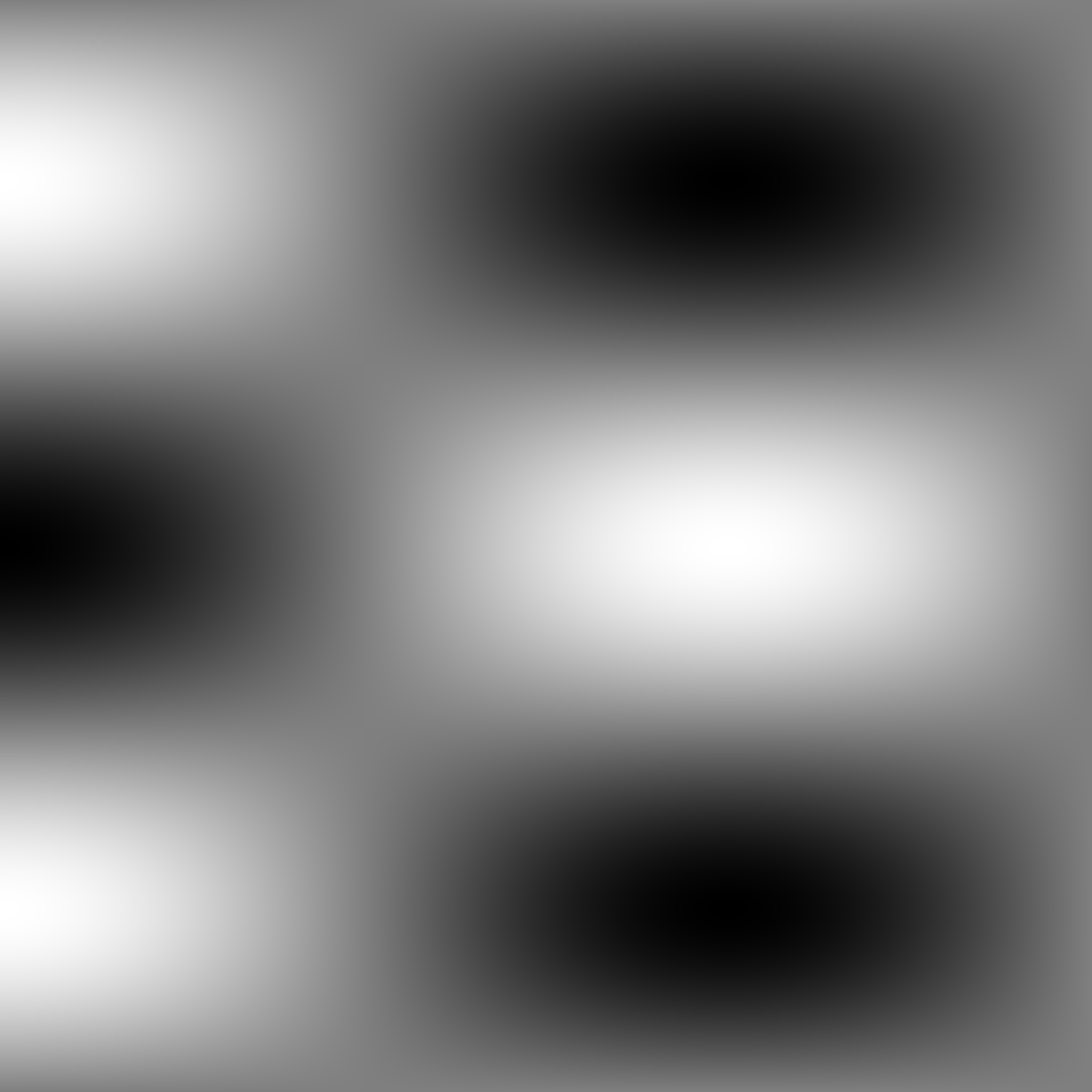} 
    \end{tabular}
    \end{tabular}
%\vspace{-10mm}
\end{figure}

\paragraph{Example 3} Finally, we consider an experiment using compressive sensing in two dimensions. As in Example 2, we consider a smooth function, but this time we also introduce a few discontinuities by adding different boxes in the image. The smooth part of the image is then well approximated by wavelets at coarse scales, whereas the discontinuity around the boxes will produce a few non-zero spikes among the wavelet coefficients at finer scales. The considered function is shown in Figure \ref{eq:CS_rec}, along with a truncated Walsh series approximation and a compressive sensing reconstruction, both from $128\times 128$ Walsh samples. The compressive sensing approximation is based on solving \eqref{eq:QCBP} with $\eta = 0.001$, using a DB4 wavelet reconstruction basis. As we can see from the figure, we can use compressive sensing to obtain a high-resolution image from relatively few measurements. In contrast, the naive Walsh approximation results in a low-resolution image where one can clearly see the pixels when zooming in. 
\begin{figure}[tb]
    \begin{center}
 \begin{\textsizefig}
   \begin{tabular}{@{}>{\centering}m{0.30\textwidth}>{\centering}m{0.30\textwidth}>{\centering\arraybackslash}m{0.30\textwidth}@{}}
      $f(t_1,t_2)$ 
    & CS sampling pattern 
    & TW sampling pattern \\
     \includegraphics[width=\linewidth]{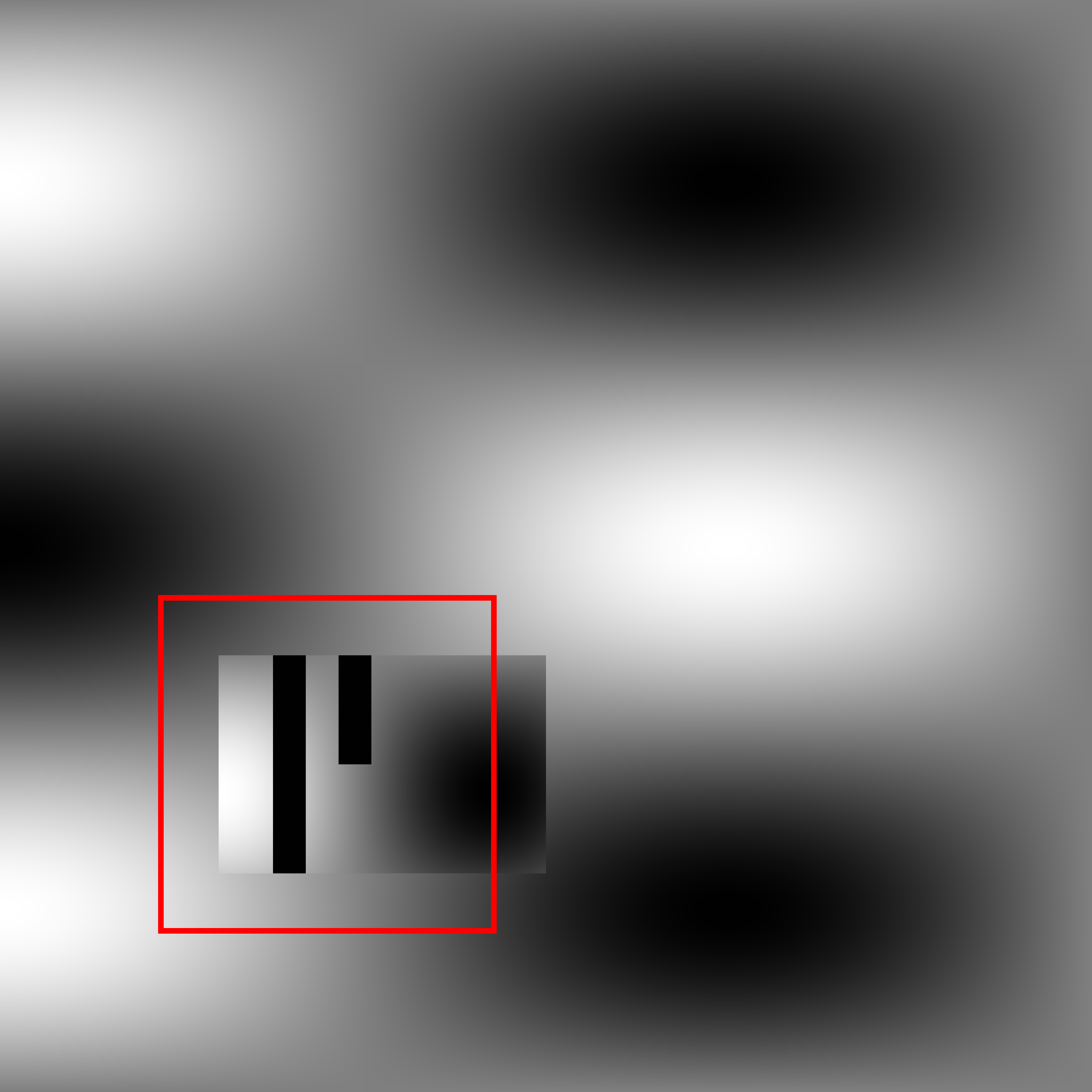}
    &\includegraphics[width=\linewidth]{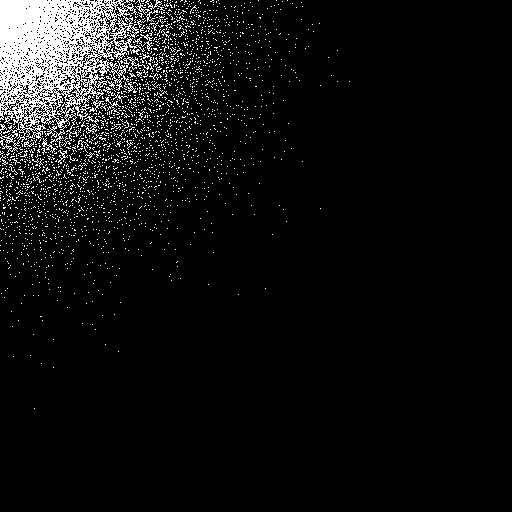} 
    &\includegraphics[width=\linewidth]{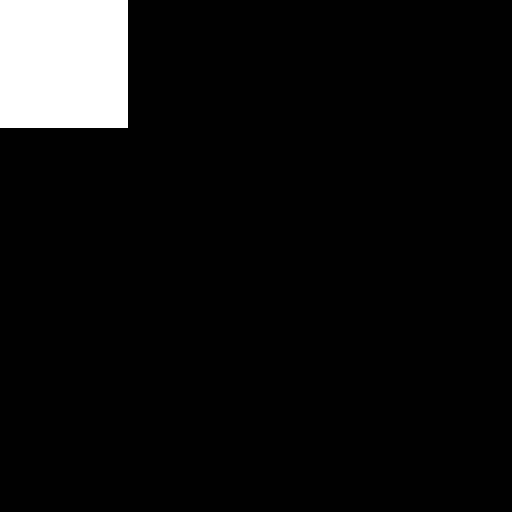} \\
      $f(t_1,t_2)$ 
    & CS reconstruction 
    & TW reconstruction \\
      (cropped) 
    & (cropped) 
    & (cropped) \\
     \includegraphics[width=\linewidth]{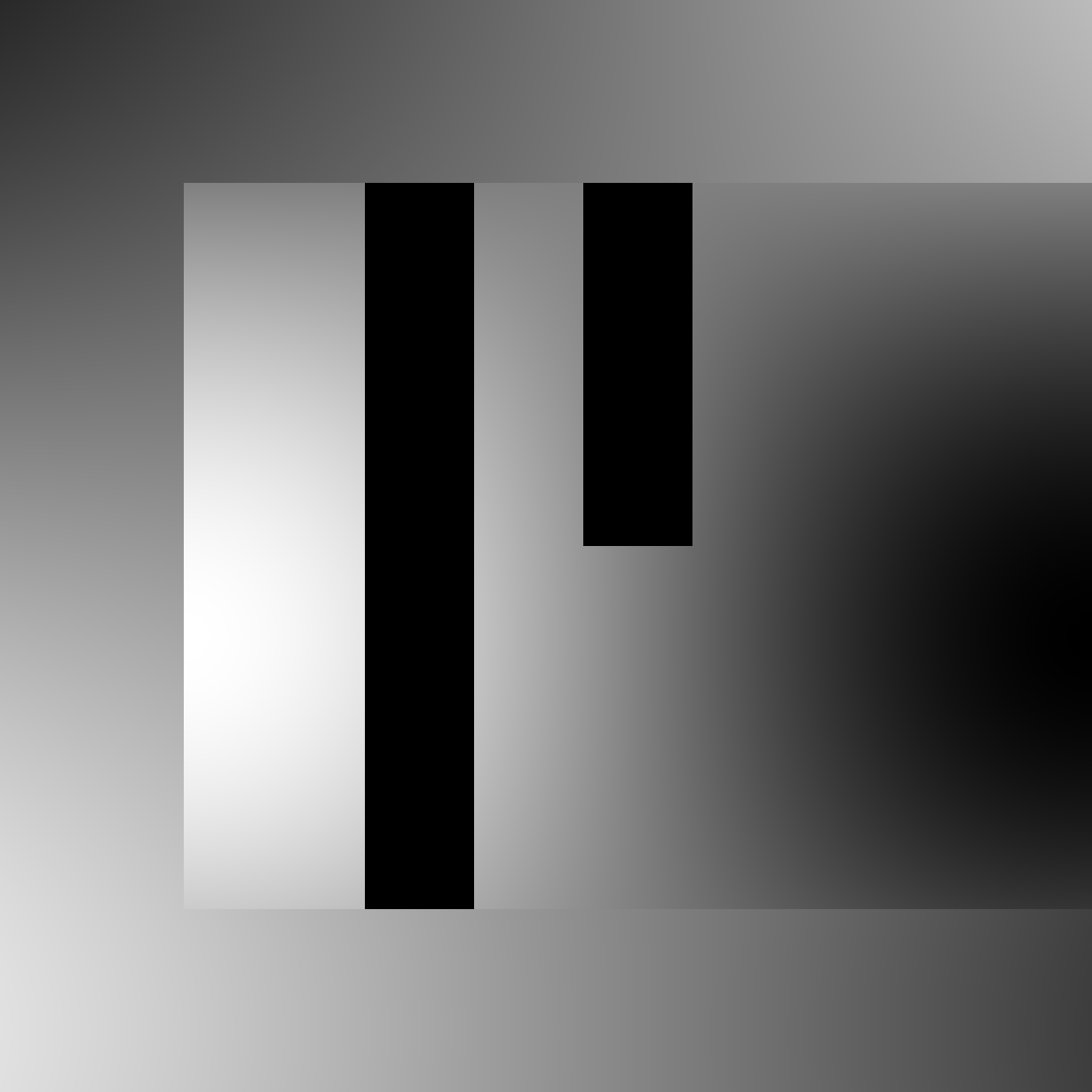}
    &\includegraphics[width=\linewidth]{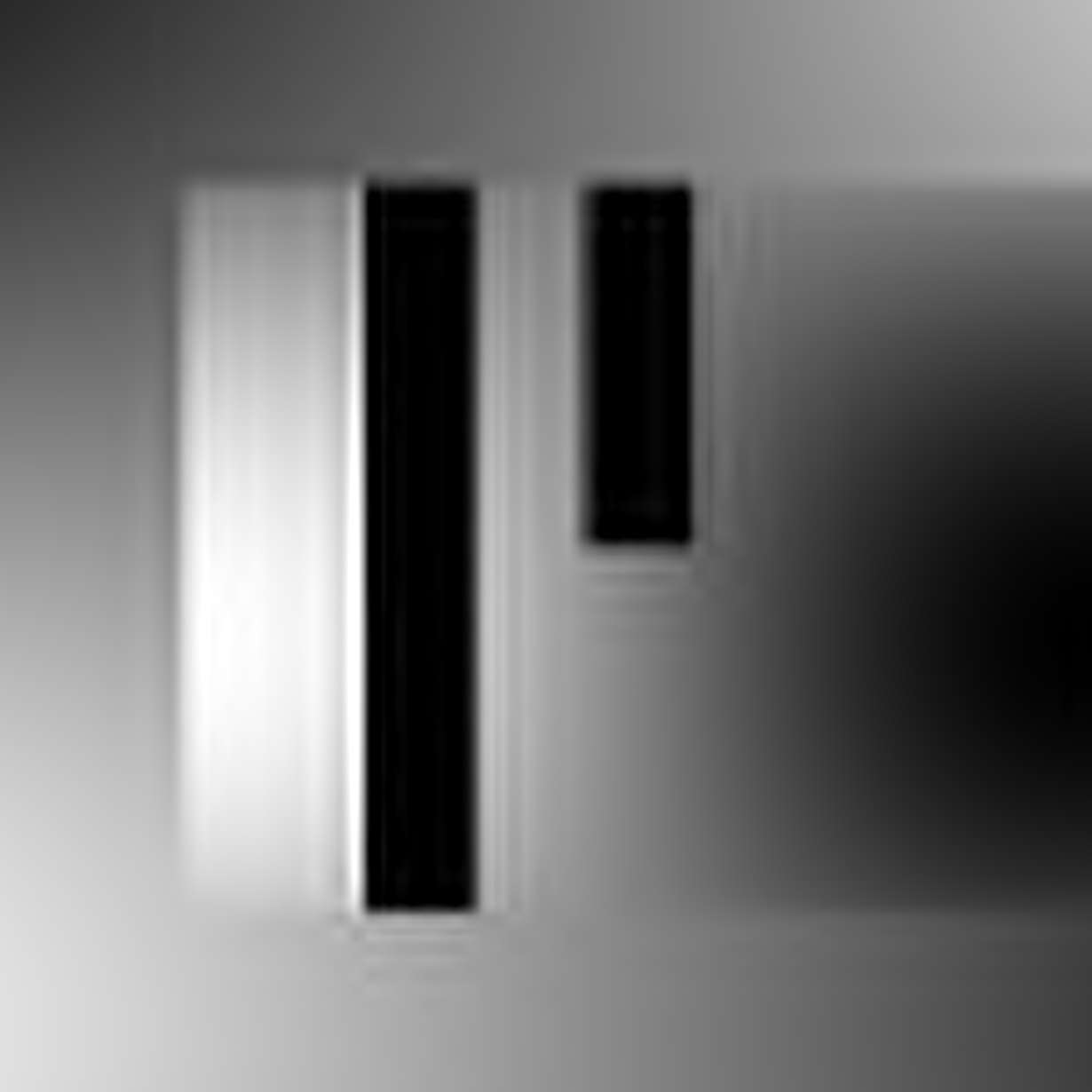} 
    &\includegraphics[width=\linewidth]{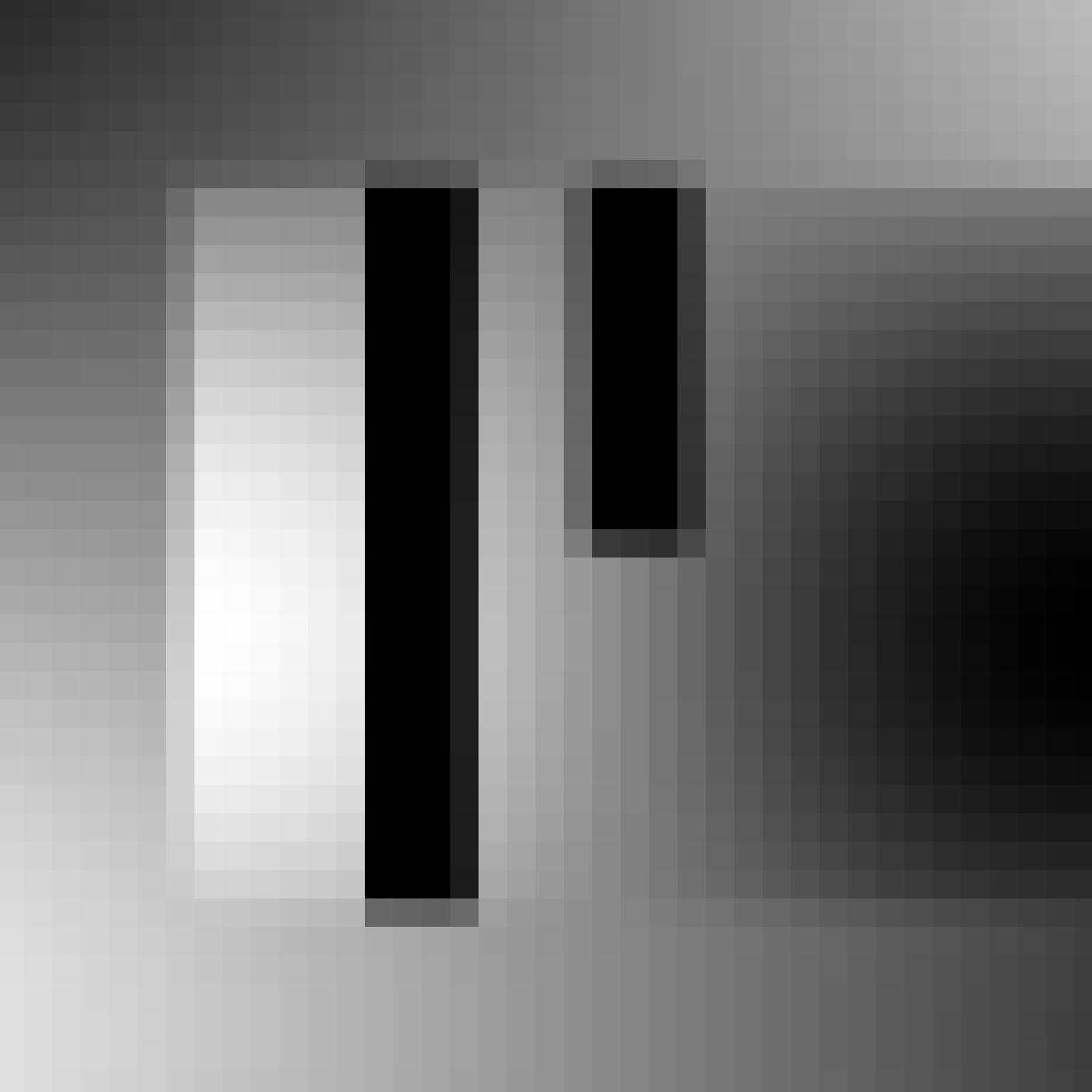}\\
    \end{tabular}
 \end{\textsizefig}
    \end{center}
    \vspace{-3mm}
    \caption{\label{eq:CS_rec}(\textbf{Compressive sensing allows for resolution enhancing}).
We consider the function $f\colon [0,1]^2 \to \mathbb{R}$ seen in the upper left corner, and acquire $128\times 128$ Walsh samples from $f$ using the two strategies seen in the upper right corner. The red squares indicate the cropped area. In the bottom row, we show from left to right the function $f$, the approximation using compressive sensing, and a truncated Walsh series approximation. From these crops, we can see that the compressive sensing reconstruction provides higher fidelity than the truncated Walsh series.}

%\vspace{-4mm}
\end{figure}

\section*{Acknowledgments}
The author would like to thank Anders C.\ Hansen for his comments. 
{\footnotesize
\bibliographystyle{abbrv}
\bibliography{references}

\begin{thebibliography}{10}

\bibitem{AntunUniform}
B.~Adcock, V.~Antun, and A.~C. Hansen.
\newblock Uniform recovery in infinite-dimensional compressed sensing and
  applications to structured binary sampling.
\newblock {\em Appl.\ Comput.\ Harmon.\ Anal.}, 55:1--40, 2021.

\bibitem{adcock2012generalized}
B.~Adcock and A.~C. Hansen.
\newblock A generalized sampling theorem for stable reconstructions in
  arbitrary bases.
\newblock {\em J.\ Fourier Anal.\ Appl.}, 18(4):685--716, 2012.

\bibitem{adcock2016generalized}
B.~Adcock and A.~C. Hansen.
\newblock Generalized sampling and infinite-dimensional compressed sensing.
\newblock {\em Found.\ Comput.\ Math.}, 16(5):1263--1323, 2016.

\bibitem{CSBook}
B.~Adcock and A.~C. Hansen.
\newblock {\em Compressive imaging: Structure, Sampling, Learning}.
\newblock Cambridge University Press (in press), 2021.

\bibitem{adcock2015linear}
B.~Adcock, A.~C. Hansen, G.~Kutyniok, and J.~Ma.
\newblock Linear stable sampling rate: Optimality of 2{D} wavelet
  reconstructions from {F}ourier measurements.
\newblock {\em SIAM J.\ Math.\ Anal.}, 47(2):1196--1233, 2015.

\bibitem{Beyond13}
B.~Adcock, A.~C. Hansen, and C.~Poon.
\newblock Beyond consistent reconstructions: optimality and sharp bounds for
  generalized sampling, and application to the uniform resampling problem.
\newblock {\em SIAM J.\ Math.\ Anal.}, 45(5):3132--3167, 2013.

\bibitem{adcock2017breaking}
B.~Adcock, A.~C. Hansen, C.~Poon, and B.~Roman.
\newblock Breaking the coherence barrier: A new theory for compressed sensing.
\newblock In {\em Forum Math., Sigma}, volume~5. Cambridge University Press,
  2017.

\bibitem{7446327}
B.~Adcock, A.~C. Hansen, and B.~Roman.
\newblock A note on compressed sensing of structured sparse wavelet
  coefficients from subsampled {F}ourier measurements.
\newblock {\em IEEE Signal Process.\ Lett.}, 23(5):732--736, 2016.

\bibitem{adcock2014stability}
B.~Adcock, A.~C. Hansen, and A.~Shadrin.
\newblock A stability barrier for reconstructions from {F}ourier samples.
\newblock {\em SIAM J.\ Numer.\ Anal.}, 52(1):125--139, 2014.

\bibitem{AntunRyan}
V.~Antun and {\O}.~Ryan.
\newblock On the unification of schemes and software for wavelets on the
  interval.
\newblock {\em Acta Appl.\ Math.}, 173(7), 2021.

\bibitem{mattersComp}
J.~Arndt.
\newblock {\em Matters Computational: ideas, algorithms, source code}.
\newblock Springer Science \& Business Media, 2010.

\bibitem{Bastounis17}
A.~Bastounis and A.~C. Hansen.
\newblock On the absence of uniform recovery in many real-world applications of
  compressed sensing and the restricted isometry property and nullspace
  property in levels.
\newblock {\em SIAM J.\ Imaging Sci.}, 10(1):335--371, 2017.

\bibitem{beauchamp75}
K.~G. Beauchamp.
\newblock {\em Walsh functions and their applications}.
\newblock Academic press, 1975.

\bibitem{Binev17}
P.~Binev, A.~Cohen, W.~Dahmen, R.~DeVore, G.~Petrova, and P.~Wojtaszczyk.
\newblock Data assimilation in reduced modeling.
\newblock {\em SIAM/ASA J.\ Uncertain.\ Quantif.}, 5(1):1--29, 2017.

\bibitem{lensless_im}
V.~Boominathan, J.~K. Adams, M.~S. Asif, B.~W. Avants, J.~T. Robinson, R.~G.
  Baraniuk, A.~C. Sankaranarayanan, and A.~Veeraraghavan.
\newblock Lensless imaging: A computational renaissance.
\newblock {\em IEEE Signal Proc\. Mag.}, 33(5):23--35, 2016.

\bibitem{wavelab}
J.~Buckheit, S.~Chen, D.~L. Donoho, I.~Johnstone, and J.~Scargle.
\newblock About {W}ave{L}ab, 1995.

\bibitem{chambolle2011first}
A.~Chambolle and T.~Pock.
\newblock A first-order primal-dual algorithm for convex problems with
  applications to imaging.
\newblock {\em J. Math. Imaging Vision}, 40(1):120--145, 2011.

\bibitem{chi2011sensitivity}
Y.~Chi, L.~L. Scharf, A.~Pezeshki, and A.~R. Calderbank.
\newblock Sensitivity to basis mismatch in compressed sensing.
\newblock {\em IEEE Trans.\ Signal Proces.}, 59(5):2182--2195, 2011.

\bibitem{Clemente:13}
P.~Clemente, V.~Dur\'{a}n, E.~Tajahuerce, P.~Andr\'{e}s, V.~Climent, and
  J.~Lancis.
\newblock Compressive holography with a single-pixel detector.
\newblock {\em Opt. Lett.}, 38(14):2524--2527, Jul 2013.

\bibitem{Cohen93}
A.~Cohen, I.~Daubechies, and P.~Vial.
\newblock Wavelets on the interval and fast wavelet transforms.
\newblock {\em Appl.\ Comput.\ Harmon.\ Anal}, 1(1):54--81, 1993.

\bibitem{Daubechies92}
I.~Daubechies.
\newblock {\em Ten lectures on wavelets}.
\newblock SIAM, 1992.

\bibitem{devore2017data}
R.~DeVore, G.~Petrova, and P.~Wojtaszczyk.
\newblock Data assimilation and sampling in {B}anach spaces.
\newblock {\em Calcolo}, 54(3):963--1007, 2017.

\bibitem{devore_1998}
R.~A. DeVore.
\newblock Nonlinear approximation.
\newblock {\em Acta Numer.}, 7:51–150, 1998.

\bibitem{eldar2003sampling}
Y.~C. Eldar.
\newblock Sampling with arbitrary sampling and reconstruction spaces and
  oblique dual frame vectors.
\newblock {\em J. Fourier Anal.\ Appl.}, 9(1):77--96, 2003.

\bibitem{eldar2004sampling}
Y.~C. Eldar.
\newblock Sampling without input constraints: Consistent reconstruction in
  arbitrary spaces.
\newblock In {\em Sampling, wavelets, and tomography}, pages 33--60. Springer,
  2004.

\bibitem{eldar2005general}
Y.~C. Eldar and T.~Werther.
\newblock General framework for consistent sampling in {H}ilbert spaces.
\newblock {\em Int.\ J.\ Wavelets.\ Multi.}, 3(04):497--509, 2005.

\bibitem{epstein2007introduction}
C.~L. Epstein.
\newblock {\em Introduction to the mathematics of medical imaging}.
\newblock SIAM, 2007.

\bibitem{Foucart13}
S.~Foucart and H.~Rauhut.
\newblock {\em A Mathematical Introduction to Compressive Sensing}.
\newblock Springer - Birk{\"a}user, 1th edition, 2013.

\bibitem{Gataric16}
M.~Gataric and C.~Poon.
\newblock A practical guide to the recovery of wavelet coefficients from
  {F}ourier measurements.
\newblock {\em SIAM J.\ Sci.\ Comput.}, 38(2):A1075--A1099, 2016.

\bibitem{Golubov91}
B.~Golubov, A.~Efimov, and V.~Skvortsov.
\newblock {\em Walsh series and transforms: theory and applications},
  volume~64.
\newblock Springer Science \& Business Media, 1991.

\bibitem{fastWaveletRecUnser}
M.~Guerquin-Kern, M.~Haberlin, K.~P. Pruessmann, and M.~Unser.
\newblock A fast wavelet-based reconstruction method for magnetic resonance
  imaging.
\newblock {\em IEEE Trans.\ Med.\ Imaging}, 30(9):1649--1660, 2011.

\bibitem{GLPU_phantom}
M.~Guerquin-Kern, L.~Lejeune, K.~P. Pruessmann, and M.~Unser.
\newblock Realistic analytical phantoms for parallel magnetic resonance
  imaging.
\newblock {\em IEEE Trans.\ Med.\ Imaging}, 31(3):626--636, 2012.

\bibitem{ThesingSSR}
A.~C. Hansen and L.~Thesing.
\newblock On the stable sampling rate for binary measurements and wavelet
  reconstruction.
\newblock {\em Appl.\ Comput.\ Harmon. Anal.}, 48(2):630--654, 2020.

\bibitem{hestenes1952methods}
M.~R. Hestenes and E.~Stiefel.
\newblock Methods of conjugate gradients for solving linear systems.
\newblock {\em J.\ Res. Natl. Bur. Stand.}, 49(6), 1952.

\bibitem{hirabayashi2007consistent}
A.~Hirabayashi and M.~Unser.
\newblock Consistent sampling and signal recovery.
\newblock {\em IEEE Trans.\ Signal Proces.}, 55(8):4104--4115, 2007.

\bibitem{hrycak10}
T.~Hrycak and K.~Gr{\"o}chenig.
\newblock Pseudospectral {F}ourier reconstruction with the modified inverse
  polynomial reconstruction method.
\newblock {\em J.\ Comput.\ Phys.}, 229(3):933--946, 2010.

\bibitem{jardine2009helium}
A.~Jardine, H.~Hedgeland, G.~Alexandrowicz, W.~Allison, and J.~Ellis.
\newblock Helium-3 spin-echo: Principles and application to dynamics at
  surfaces.
\newblock {\em Prog.\ Surf.\ Sci.}, 84(11-12):323--379, 2009.

\bibitem{jones2016continuous}
A.~Jones, A.~Tamt{\"o}gl, I.~Calvo-Almaz{\'a}n, and A.~Hansen.
\newblock Continuous compressed sensing for surface dynamical processes with
  helium atom scattering.
\newblock {\em Sci.\ rep.}, 6(1):1--11, 2016.

\bibitem{kutyniok2018optimal}
G.~Kutyniok and W.-Q. Lim.
\newblock Optimal compressive imaging of {F}ourier data.
\newblock {\em SIAM J. Imaging Sci.}, 11(1):507--546, 2018.

\bibitem{leary2013compressed}
R.~Leary, Z.~Saghi, P.~A. Midgley, and D.~J. Holland.
\newblock Compressed sensing electron tomography.
\newblock {\em Ultramicroscopy}, 131:70--91, 2013.

\bibitem{mri_principles}
Z.-P. Liang and P.~C. Lauterbur.
\newblock {\em Principles of magnetic resonance imaging: a signal processing
  perspective}.
\newblock SPIE Optical Eng.\ Press, 2000.

\bibitem{cs_MRI_lustig}
M.~Lustig, D.~L. Donoho, J.~M. Santos, and J.~M. Pauly.
\newblock Compressed sensing {MRI}.
\newblock {\em IEEE Signal Proc.\ Mag.}, 25(2):72--82, 2008.

\bibitem{ma2017generalized}
J.~Ma.
\newblock Generalized sampling reconstruction from {F}ourier measurements using
  compactly supported shearlets.
\newblock {\em Appl.\ Comput.\ Harmon.\ Anal.}, 42(2):294--318, 2017.

\bibitem{maday2015pbdw}
Y.~Maday, T.~Anthony, J.~D. Penn, and M.~Yano.
\newblock {PBDW} state estimation: Noisy observations; configuration-adaptive
  background spaces; physical interpretations.
\newblock {\em ESAIM: Proceedings and Surveys}, 50:144--168, 2015.

\bibitem{maday2013generalized}
Y.~Maday and O.~Mula.
\newblock A generalized empirical interpolation method: application of reduced
  basis techniques to data assimilation.
\newblock In {\em Analysis and numerics of partial differential equations},
  pages 221--235. Springer, 2013.

\bibitem{maday15}
Y.~Maday, A.~T. Patera, J.~D. Penn, and M.~Yano.
\newblock A parameterized-background data-weak approach to variational data
  assimilation: formulation, analysis, and application to acoustics.
\newblock {\em Int.\ J.\ Numer.\ Meth.\ Eng.}, 102(5):933--965, 2015.

\bibitem{Mallat09}
S.~Mallat.
\newblock {\em A wavelet tour of signal processing: {T}he sparse way}.
\newblock Academic Press, 3rd edition, 2008.

\bibitem{9088137}
A.~Moshtaghpour, J.~M. Bioucas-Dias, and L.~Jacques.
\newblock Close encounters of the binary kind: Signal reconstruction guarantees
  for compressive {H}adamard sampling with {H}aar wavelet basis.
\newblock {\em IEEE Trans. Inf. Theory}, 66(11):7253--7273, 2020.

\bibitem{muller2006introduction}
M.~Muller.
\newblock {\em Introduction to confocal fluorescence microscopy}, volume~69.
\newblock SPIE press, 2006.

\bibitem{poon2014consistent}
C.~Poon.
\newblock A consistent and stable approach to generalized sampling.
\newblock {\em J.\ Fourier Anal. Appl.}, 20(5):985--1019, 2014.

\bibitem{ravishankar2019image}
S.~Ravishankar, J.~C. Ye, and J.~A. Fessler.
\newblock Image reconstruction: From sparsity to data-adaptive methods and
  machine learning.
\newblock {\em Proc.\ of the IEEE}, 108(1):86--109, 2019.

\bibitem{studer2012compressive}
V.~Studer, J.~Bobin, M.~Chahid, H.~S. Mousavi, E.~Candes, and M.~Dahan.
\newblock Compressive fluorescence microscopy for biological and hyperspectral
  imaging.
\newblock {\em Proc. Natl. Acad. Sci. USA}, 109(26):E1679--E1687, 2012.

\bibitem{Tang00}
W.-S. Tang.
\newblock Oblique projections, biorthogonal {R}iesz bases and multiwavelets in
  {H}ilbert spaces.
\newblock {\em P.\ Amer.\ Math.\ Soc.}, 128(2):463--473, 2000.

\bibitem{Terhaar18}
L.~Thesing and A.~C. Hansen.
\newblock Linear reconstructions and the analysis of the stable sampling rate.
\newblock {\em Sampling Theory in Signal and Image Processing}, 2018.

\bibitem{thesing2021non}
L.~Thesing and A.~C. Hansen.
\newblock Non-uniform recovery guarantees for binary measurements and
  infinite-dimensional compressed sensing.
\newblock {\em J. Fourier Anal. Appl.}, 27(2):1--44, 2021.

\bibitem{Traonmilin15}
Y.~Traonmilin and R.~Gribonval.
\newblock Stable recovery of low-dimensional cones in {H}ilbert spaces: One
  {RIP} to rule them all.
\newblock {\em Appl.\ Comput.\ Harmon.\ Anal.}, 45(1):170--205, 2018.

\bibitem{unser1994general}
M.~Unser and A.~Aldroubi.
\newblock A general sampling theory for nonideal acquisition devices.
\newblock {\em IEEE Trans.\ Signal Proces.}, 42(11):2915--2925, 1994.

\bibitem{unser1998generalized}
M.~Unser and J.~Zerubia.
\newblock A generalized sampling theory without band-limiting constraints.
\newblock {\em IEEE tran. circuits-II}, 45(8):959--969, 1998.

\bibitem{splg1_paper}
E.~van~den Berg and M.~P. Friedlander.
\newblock Probing the {P}areto frontier for basis pursuit solutions.
\newblock {\em SIAM J.\ Sci.\ Comput.}, 31(2):890--912, 2008.

\end{thebibliography}
}
\end{document}